\documentclass[letterpaper,11pt]{amsart}

\usepackage{amsmath}
\usepackage{amsthm}
\usepackage{mathrsfs}
\usepackage{amssymb,amscd}
\usepackage[all]{xy}
\usepackage{mathtools}
\usepackage{fullpage}
\usepackage{comment}
\usepackage{tikz-cd}

\usepackage{xcolor}
\usepackage[colorlinks=true,linkcolor=blue,breaklinks,pagebackref]{hyperref}

\usepackage[lite,abbrev,msc-links,alphabetic]{amsrefs}

\newcommand{\fa}{\mathfrak{a}}

\newcommand{\fc}{\mathfrak{c}}

\newcommand{\fg}{\mathfrak{g}}

\newcommand{\fm}{\mathfrak{m}}
\newcommand{\fn}{\mathfrak{n}}
\newcommand{\fo}{\mathfrak{o}}
\newcommand{\fp}{\mathfrak{p}}

\newcommand{\fs}{\mathfrak{s}}

\newcommand{\fu}{\mathfrak{u}}
\newcommand{\fx}{\mathfrak{x}}
\newcommand{\fy}{\mathfrak{y}}

\newcommand{\bA}{\mathbb{A}}

\newcommand{\bH}{\mathbb{H}}

\newcommand{\C}{\mathbb{C}}

\newcommand{\R}{\mathbb{R}}

\newcommand{\Z}{\mathbb{Z}}

\newcommand{\cA}{\mathcal{A}}
\newcommand{\cB}{\mathcal{B}}
\newcommand{\cC}{\mathcal{C}}

\newcommand{\cF}{\mathcal{F}}

\newcommand{\cH}{\mathcal{H}}

\newcommand{\cK}{\mathcal{K}}
\newcommand{\cL}{\mathcal{L}}

\newcommand{\cO}{\mathcal{O}}

\newcommand{\cS}{\mathcal{S}}

\newcommand{\cU}{\mathcal{U}}
\newcommand{\cV}{\mathcal{V}}
\newcommand{\cW}{\mathcal{W}}
\newcommand{\cX}{\mathcal{X}}

\newcommand{\rd}{\mathrm{d}}

\newcommand{\tS}{\mathtt{S}}

\newcommand{\bs}{\backslash}

\newcommand{\re}{\mathrm{Re}\,}

\newcommand{\tp}[1]{\prescript{t}{}{#1}}

\DeclareMathOperator{\diag}{diag} 
 \DeclareMathOperator{\vol}{vol}

\DeclareMathOperator{\Gal}{Gal}

\providecommand{\abs}[1]{\lvert#1\rvert}
\providecommand{\aabs}[1]{\lVert#1\rVert}

\newcommand{\valP}[1]{\left|#1\right|}

\DeclareMathOperator{\Tr}{Tr}

\DeclareMathOperator{\Trace}{Trace}

\DeclareMathOperator{\GL}{GL}

\DeclareMathOperator{\U}{U}

\DeclareMathOperator{\gl}{\mathfrak{gl}}

\DeclareMathOperator{\Res}{Res}

\DeclareMathOperator{\Ind}{Ind}

\DeclareMathOperator{\Hom}{Hom}

\DeclareMathOperator{\id}{\mathbf{1}}

\DeclareMathOperator{\disc}{disc}

\DeclareMathOperator{\BC}{BC}

\DeclareMathOperator{\Ad}{Ad}

\DeclareMathOperator{\Temp}{Temp} \DeclareMathOperator{\temp}{temp}

\DeclareMathOperator{\rs}{rs}

\DeclareMathOperator{\otimeshat}{\widehat{\otimes}}

\DeclareFontFamily{U}{mathx}{\hyphenchar\font45}
\DeclareFontShape{U}{mathx}{m}{n}{
      <5> <6> <7> <8> <9> <10>
      <10.95> <12> <14.4> <17.28> <20.74> <24.88>
      mathx10
      }{}
\DeclareSymbolFont{mathx}{U}{mathx}{m}{n}
\DeclareMathAccent{\widecheck}{\mathalpha}{mathx}{"71}

\makeatletter
\def\Ddots{\mathinner{\mkern1mu\raise\p@
\vbox{\kern7\p@\hbox{.}}\mkern2mu
\raise4\p@\hbox{.}\mkern2mu\raise7\p@\hbox{.}\mkern1mu}}
\makeatother

\theoremstyle{definition}
\newtheorem{definition}{Definition}[section]

\theoremstyle{plain}
\newtheorem{theorem}[definition]{Theorem}
\newtheorem{prop}[definition]{Proposition}
\newtheorem{lemma}[definition]{Lemma}
\newtheorem{coro}[definition]{Corollary}

\theoremstyle{remark}
\newtheorem{remark}[definition]{Remark}

\numberwithin{equation}{section}

\begin{document}

\title{The global Gan--Gross--Prasad conjecture for Fourier--Jacobi periods on unitary groups II: Comparison of the relative trace formulae}

\author{Paul Boisseau}

\author{Weixiao Lu}

\author{Hang Xue}

\address{Paul Boisseau, Max Planck Institute for Mathematics, Vivatsgasse 7, 53111 Bonn, Germany}
\email{boisseau@mpim-bonn.mpg.de}

\address{Weixiao Lu, Aix Marseille Univ, CNRS, I2M, Marseille, 13009, France}
\email{weixiao.lu@univ-amu.fr}

\address{Hang Xue, Department of Mathematics, The University of Arizona, Tucson, AZ, 85721, USA}

\email{xuehang@arizona.edu}

\date{\today}

\begin{abstract}
This is the second of a series of three papers where we prove the Gan--Gross--Prasad conjecture for Fourier--Jacobi periods on unitary groups and an Ichino--Ikeda type refinement. Our strategy is based on the comparison of relative trace formulae formulated by Liu. The goal of this second paper is to compare the two relative trace formulae.
\end{abstract}

\maketitle

\tableofcontents

\section{Introduction}

This is the second in a series of three papers aiming at establishing the global Gan--Gross--Prasad (GGP) conjecture for Fourier--Jacobi periods on unitary groups. 
We refer the readers to the introduction to~\cite{BLX1} for some discussions on where the current paper stands in our approach to the GGP conjectures.

\subsection{What did we achieve in the first of this series?}

Let $E/F$ be a quadratic extension of number fields, and $\bA, \bA_E$ the adeles of $F$ and $E$ respectively. Put $G_n' = \GL_{n, F}$ and $G_n = \Res_{E/F} \GL_{n, E}$. Put $G_+ = G \times E_n$. For each cuspidal datum $\chi$ of $G$ and we defined a distribution $I_{\chi}$ on $G_+(\bA)$. These are called spectral terms in the relative trace formula (RTF), and we do not need them in this paper. Let $\cA$ be the $2n$-dimensional affine space over $E$ viewed as an algebraic variety over $F$, and for each $\alpha \in \cA(F)$ we define another distribution $I_\alpha$ on $G_+(\bA)$ (some of them are by definition zero). These are called geometric terms in the RTF. We will recall these terms in Section~2. The coarse form of the RTF on $G_+$ is the following identity
    \begin{equation}    \label{eq:coarse_RTF_GL}
    \sum_{\chi} I_{\chi}(f_+) =
    \sum_{\alpha} I_{\alpha}(f_+).
    \end{equation}
for all $f_+ \in \cS(G_+(\bA))$. Here the sum $\chi$ runs over all cuspidal data of $G$, and the sum $\alpha$ runs over all $\alpha \in \cA(F)$.

Similarly let $V$ be a skew-hermitian space of dimension $n$ and $\U(V)$ the unitary group. Let $\U_V = \U(V) \times \U(V)$, and $\U_{V, +} = \U_V \times V$.  For each cuspidal datum $\chi$ of $\U_V$ and we defined a distribution $J_{\chi}^V$ on $\U_{V, +}(\bA)$. These are called spectral terms in the RTF on unitary groups, and again we do not need them in this paper. For each $\alpha \in \cA(F)$ we define another distribution $J_\alpha^V$ on $\U_{V, +}(\bA)$ (some of them are by definition zero). These are called geometric terms in the RTF on unitary groups. We will recall these terms in Section~2. The coarse form of the RTF on $\U_{V, +}$ is the following identity
    \begin{equation}    \label{eq:coarse_RTF_U}
    \sum_{\chi} J_{\chi}^V(f_+^V) =
    \sum_{\alpha} J_{\alpha}^V(f_+^V).
    \end{equation}
for all $f_+^V \in \cS(\U_{V, +}(\bA))$. Here the sum $\chi$ runs over all cuspidal data of $\U_{V, +}$, and the sum $\alpha$ runs over all $\alpha \in \cA(F)$.

\subsection{What do we achieve in this paper?}

The first goal of this paper is to define a notion of matching of test functions. For regular semisimple $\alpha \in \cA(F)$, a notion which we will define in Section~4, the geometric terms $I_{\alpha}$ and $J_{\alpha}^V$ reduce to the more familiar orbital integrals. We define that a function $f_+ \in \cS(G_+(\bA))$ and a collection of functions $f_+^V \in \cS(\U_{V, +}(\bA))$ where $V$ runs over all skew-hermitian spaces of dimension $n$ match if 
    \begin{equation}    \label{eq:matching_intro}
    I_{\alpha}(f_+) = \sum_V J_{\alpha}^V(f_+^V),
    \end{equation}
for all regular semisimple $\alpha \in \cA(F)$. For such $\alpha$, there is at most one nonzero term on the right hand side. We will prove that for matching test functions, the identity~\eqref{eq:matching_intro} holds for all $\alpha \in \cA(F)$, not just the regular semisimple ones. This is what is referred to as the singular transfer of test functions. It is achieved via descending of the geometric terms $I_{\alpha}$ and $J_{\alpha}^V$ to the Lie algebra, and the identity~\eqref{eq:matching_intro} eventually reduces to the results of~\cite{CZ}. The fact that we are able to do this reduction ultimately relies on the particular way we setup and truncate the RTFs, as was done in~\cite{BLX1}. See also the introduction of~\cite{BLX1} for more discussions.

The matching of test functions is a local property. Let $v$ be a place of $F$. For a function $f_{+, v} \in \cS(G_+(F_v))$ and a collection of functions $f_{+, v}^V \in \cS(\U_{V, +}(F_v))$ where $V$ runs over all skew-hermitian spaces of dimension $n$ over $E_v$, we can define the notion of matching analogously. Indeed global (factorizable) test functions match if they match at all places. We prove two results for matching test functions $f_{+, v}$ and $f_{+, v}^V$. 
\begin{itemize}
\item Transfer of test functions. Given $f_{+,v}$ we can find a collection $f_{+, v}^V$ that matches it, and vice versa. If $v$ is archimedean this is proved only for a dense subset of test functions. 

\item The fundamental lemma. This means if all data at the place $v$ are unramified, then the unramified test functions should match. See Theorem~\ref{thm:FL} for the precise statement.
\end{itemize}
These two results form the pillars of the comparison of the RTFs at hand.
    
These results form the first part of this paper. The rest of this paper is devoted to proving a local character identity for matching test functions. Again let $v$ be a place of $F$. Let $\pi_v$ be an irreducible tempered representation of $\U_V(F_v)$ with a nonzero Fourier--Jacobi model, cf.~Subsection~6.1 for an explanation. We define a distribution $J_{\pi}^V$ on $\U_V^+(F_v)$. For its base change $\Pi_v$ to $G(F_v)$, we define another distribution $I_{\Pi_v}$ on $G_+(F_v)$. We prove that there is an explicit constant $\kappa(\pi_v)$ such that for all matching test functions $f_{+, v}$ and $\{f_{+, v}^V\}_V$, we have
    \[
    J_{\pi_v}^V(f^V_+) = \kappa(\pi_v) I_{\Pi_v}(f_+).
    \]
We also prove that this identity characterizes matching of test functions. This result will play a definitive role in the third of this series of papers. It is used in the proof of a refined version of the global GGP conjecture, and is also used in the process of isolation of a particular cuspidal datum on the spectral side of the RTFs.

\subsection{General notation and conventions}

We keep the same general notation and convention from~\cite{BLX1}. We recall the following.
\begin{itemize}
\item Let $E/F$ be a quadratic extension of either number fields or local fields of characteristic zero. In either case we denote by
$\textsf{c}$ the nontrivial Galois conjugation in $\Gal(E/F)$. If $A$ is a group on which $\Gal(E/F)$ acts, we denote by $A^-$ the purely imaginery elements in $A$, i.e. $x \in A$ such that $x^{\textsf{c}} = -x$. We fix a nonzero $\tau \in
E^-$. 

\item If $E/F$ are number fields, we denote by $\bA$ and $\bA_E$ the ring of adeles of $F$ and $E$ respectively. We write $\bA_{f}$
the ring of finite adeles of $\bA$. If $\tS$ is a finite set of places of $F$, we set $F_\tS := \prod_{v \in S} F_v$. When $\tS = V_{F,\infty}$ is the set of Archimedean places of $F$, we also write $F_{\infty} := F_{V_{F,\infty}}$. 

\item We fix a nontrivial additive character of $\psi$ of $F\bs \bA$, and let
$\psi^E$ and $\psi_E$ be the nontrivial additive characters of $\bA_E$ given
respectively by
    \[
    \psi^E(x) = \psi(\Tr_{E/F}(\tau x)), \quad \psi_E(x) = \psi(\Tr_{E/F} x).
    \]
Let $\eta$ be the quadratic character of $F^\times \bs \bA^\times$
associated to the extension $E/F$ by global class field theory. We let $\mu$
be a character of $E^\times \bs \bA_E^\times$ whose restriction to $\bA^\times$ is $\eta$. The same notation applies to the case when $E/F$ are local fields of characteristic zero.

\item For any reductive group $G$ over $F$, we let $X(G)$ be the space of $F$-rational characters on $G$. Put $\fa_G = \Hom(X(G), \R)$ and $\fa_G^* = X(G) \otimes_{\Z} \R$. Denote by $H_G: G(F) \to \fa_G$ the usual Harish-Chandra map, cf.~\cite{BLX1}*{Section~3.5}.

\item For $G$ an algebraic group over $F$, we defined the space of Schwartz functions $\cS(G(\bA))$ (if $F$ is a number field) and $\cS(G(F))$ (if $F$ is a local field of characteristic zero), cf.~\cite{BLX1}*{Section~3.3}.

\item Let $G$ be a reductive group over $F$, and $P$ a parabolic subgroup. We often write $M_P$ for its Levi subgroup and $N_P$ for its unipotent radical. When we write $P = MN$, it always means $M$ is the Levis subgroup and $N$ is the unipotent radical.  When $F$ is a number field, we put $[G]_P =  N(\bA)M(F) \bs G(\bA)$, and  $[G] = [G]_G$.

\item Let $G'_n = \GL_{n, F}$, $G_n = \Res_{E/F} \GL_{n, E}$. Put $G = G_n \times G_n$, $H$ the diagonal subgroup of $G$ which is isomorphic to $G_n$, and $G' = G_n' \times G_n'$ which embeds in $G$ componentwise. If $g \in G$ or $G'$ we always write $g = (g_1, g_2)$ for the two components. Define a quadratic character $\eta_{G'}$ of $G'(\bA)$ (or $G'(F)$ when $F$ is local) by $\eta_{G'}(g') = \eta(\det g_1'g_2')^{n+1}$.

\item Let $V$ be an $n$-dimensional skew-hermitian space over $E$, and $\U(V)$ the unitary group. Let $\U_V = \U(V) \times \U(V)$ and $\U_V'$ the diagonal subgroup, which is isomorphic to $\U(V)$. 

\item Comments on the measures. In Part~1 of this paper, we do need the precise choice of the measures. So to avoid any further complication of notation, we just assume that we have prefixed a measure on each space we integrate. In Part~2 of this paper, we will establish a precise formula relating spherical characters, and we will specify what measures we use in Section~5.
\end{itemize}

\subsection{Acknowledgments}
We thank Rapha\"el Beuzart-Plessis and Wei Zhang for many helpful discussions.
PB was partly funded by the European Union ERC Consolidator Grant, RELANTRA, project number 101044930. Views and opinions expressed are however those of the author only and do not necessarily reflect those of the European Union or the European Research Council. Neither the European Union nor the granting authority can be held responsible for them. WL was partially supported by the National Science Foundation under Grant No. 1440140, while he was in residence at the Mathematical Sciences Research Institute in Berkeley, California, during the semester of Spring 2023. HX is partially supported by the NSF grant DMS~\#2154352.

\part{Matching test functions}

\section{Recollections}
\subsection{Geometric distributions on the general linear groups}

We developed in~\cite{BLX1} the coarse form of the relative trace formulae on the general linear groups. We recall the geometric side in this subsection. 

One major innovation in~\cite{BLX1} is the use of the Jacobi groups. Let $L = E^n$, $L^\vee = E_n$, and $S = L^\vee + L + E$ the Heisenberg group. Let $J_n = S \rtimes G_n$ be the Jacobi group. The group $J_n$ can be realized as a subgroup of $G_{n+2}$ consisting of matrices of the form
    \[
    \begin{pmatrix} 1 &\bullet &\bullet &\bullet &\star \\
    &*&*&*&\circ\\ &*&*&*&\circ \\
    &*&*&*&\circ \\ &&&& 1\end{pmatrix}.
    \]
Here the $*$'s give the group $G_n$, $\bullet$'s give $E_n$, $\circ$'s give $E^n$ and $\star$ give the center $E$. The later three combined give the Heisenberg group $S$. 

A standard parabolic subgroup of $G_n$ are those block upper triangular ones. We introduced the notion of standard D-parabolic subgroups of $J_n$ in~\cite{BLX1}. Let $\cF$ be the set of standard D-parabolic subgroups for $J_n$. D-parabolic subgroups in $\cF$ can be described in concrete terms as follows. For $P \in \cF$, we put $P_n = P
\cap G_n$. Then $P_n$ is a standard parabolic subgroup of $G_n$. Assume that $P_n$ is the stabilizer of the flag
    \begin{equation}    \label{eq:increasing_sequence}
     0 = L_0 \subset L_1 \subset \cdots \subset L_r=L.
    \end{equation}
Then  $P$ is either of the form
    \begin{equation}    \label{eq:Type 1 D parabolic}
    P = (L_k^\perp \times L_k \times E) \rtimes P_n
    \end{equation}
for some $0 \leq k \leq r$, in which case we call $P$ of type I,  or of the form
    \begin{equation}    \label{eq:Type 2 D parabolic}
    P  = (L_k^\perp \times L_{k+1} \times E) \rtimes P_n
    \end{equation}
for $0 \le k \le r-1$, in which case we call $P$ of type II. If $P$ is of type I,
the D-Levi decomposition for $P$ is
    \[
     M_{P}  = (0 \times 0 \times E) \rtimes M_{P_n}, \quad
    N_{P}  = (L_k^\perp \times L_k \times 0) \rtimes N_{P_n}.
    \]
If $P$ is of type II, the D-Levi decomposition for $P$ is
    \[
        M_{P}
        = (L_k^\perp/L_{k+1}^{\perp} \times L_{k+1}/L_k \times E)
        \rtimes M_{P_n}, \quad
        N_{P}  = (L_{k+1}^\perp \times L_k \times 0) \rtimes N_{P_n}.
    \]
We showed in~\cite{BLX1}*{Lemma~4.2} that $\cF$ is in bijection with the set of parabolic subgroups $P_{n+1} \times P_n$ of $G_{n+1} \times G_n$ such that $P_n = P_{n+1} \cap G_n$ and is standard.
    
In terms of matrices, a type I D-parabolic subgroup looks like
    \[
    \begin{pmatrix} 1 & & &\bullet &\star \\
    &*&*&*&\circ\\ &*&*&*&\circ \\
    &&&*& \\ &&&& 1\end{pmatrix}, \qquad
    \begin{pmatrix} 1 & & & &\star \\
    &*&*&&\\ &*&*&& \\
    &&&*& \\ &&&& 1\end{pmatrix},
    \]
and its D-Levi (the right matrix) is a product of smaller $\GL$s (and the upper right corner center). A Type II D-parabolic subgroup looks like
    \[
    \begin{pmatrix} 1 & &\bullet &\bullet &\star \\
    &*&*&*&\circ\\ &&*&*&\circ \\
    &&&*& \\ &&&& 1\end{pmatrix}, \qquad
    \begin{pmatrix} 1 & &\bullet & &\star \\
    &*&&&\\ &&*&& \circ \\
    &&&*& \\ &&&& 1\end{pmatrix}
    \]
and its D-Levi subgroup (the right matrix) is a product of $\GL$s and a smaller Jacobi group.

For $P = MN \in \cF$, where $M$ is the D-Levi and $N$ is the D-unipotent radical, we write $P_n = P \cap G_n = M_n N_n$, which is a parabolic subgroup of $G_n$, and $M_G = M_n \times M_n$, $N_G = N_n \times N_n$. We also denote by $M_{L^\vee} = M \cap L^\vee$ and $M_{L} = M \cap L$ which are subgroups of $S$.

To state the geometric side of the relative trace formula, we recall the following notation from~\cite{BLX1}*{Section~6}.
\begin{itemize}
\item Let $A_{n+1}$ be the diagonal maximal torus of $G_{n+1}'$ and $\fa_{n+1}' = \fa_{A_{n+1}}'$.

\item Put $G^+ = G \times L^{\vee, -} \times L^-$ and give it the product group structure. Recall that for any subset $A$ in $L^\vee$ or $L$, the subset of purely imaginery elements in $A$ is denoted by $A^-$. The group $H \times G'$ acts from the right on $G^+$ by
    \begin{equation}    \label{eq:action_G^+}
    (g,w,v) \cdot (h,g') = (h^{-1}gg',w g_1',g_1'^{-1} v).
    \end{equation}

\item Let $P = MN \in \cF$. Put $M^+ = M_G \times
    M_{L^{\vee}}^{-} \times M_L^{-}$, $N^+ = N_G \times N_{L^{\vee}}^-
    \times N^{-}_L$, and $P^+ = M^+ N^+$. We often write an element in $P^+$ as  $m^+n^+$ accordingly.

\item Let $q: G^+ \to \cA= G^+//(H \times G')$ be     the GIT quotient. We define a morphism $G^+ \to \Res_{E/F} \mathbf{A}_{2n, E}$ by
        \begin{equation} \label{eq:gl_GIT_map}
        ((g_1, g_2), w, v) \mapsto (a_1, \hdots, a_n; b_1, \hdots, b_n),
        \end{equation}
        where $\mathbf{A}_{2n, E}$ denotes the $2n$-dimensional affine space over $E$, and we put
        \[
        s = (g_1^{-1} g_2)(g_1^{-1}g_2)^{\mathsf{c}, -1}, \quad
        a_i = \Trace \wedge^i s,
        \quad b_i = w  s^i v, \quad i = 1, 2, \hdots, n.
        \]
    This map descends to a locally closed embedding $\cA \hookrightarrow \mathrm{Res}_{E/F} \mathbf{A}_{2n, E}$. We always consider $\cA$ as a locally closed subscheme of $\Res_{E/F}\mathbf{A}_{2n, E}$.

\item If $\alpha \in \cA(F)$, we define $G^+_{\alpha}$ to be the inverse
    image of $\alpha$ (as a closed subscheme of $G^+$). For $P = MN \in \cF$, we put
    $M^+_{\alpha} = G^+_\alpha \cap M^+$.
\end{itemize}

For
$\alpha \in \cA(F)$, $f^+ \in \cS(G^+(\bA))$, and $P \in \cF$,
we defined a kernel function on $[H]_{P_H} \times [G']_{P_{G'}}$ by
    \[
    k_{f^+, P, \alpha}(h,g')=
    \sum_{m^+ \in M^+_{\alpha}(F)}
    \int_{N^+(\bA)}
    f^+(m^+ n^+ \cdot (h,g'))
    \rd n^+.
    \]
We defined a modified kernel
    \[
    k_{f^+, \alpha}^T(h, g') = \sum_{P \in \cF} \epsilon_P
    \sum_{\substack{\gamma \in P_H(F) \bs H(F)\\
    \delta \in P_{G'}(F) \bs G'(F)}}
    \widehat{\tau}_{P_{n+1}}(H_{P_{n+1}}(\delta_1 g'_1) - T_{P_{n+1}})
    k_{f^+, P, \alpha}(\gamma h, \delta g').
    \]
Here $\epsilon_P$ is a certain sign attached to $P$, $T \in \fa_{n+1}'$ is a truncation parameter and $\widehat{\tau}(...)$ is a the characteristic function of a certain cone in $\fa_{n+1}'$. The definitions of the unexplained terms are in~\cite{BLX1}*{Section~3.5}, but for the purpose of this paper, we do not need the precise definitions of them.

With these kernel functions, we defined in~\cite{BLX1}*{Section~6} the distribution
    \[
    i_\alpha^T(f^+) = \int_{[G']} \int_{[H]} k^T_{f^+, \alpha}(h, g')
    \eta_{G'}(g') \rd h \rd g'.
    \]
When $T$ is sufficiently positive (meaning that $T$ lies in a certain cone, cf.~\cite{BLX1}*{Subsection~3.5.3} for the definition), this expression is absolutely convergent. The function $T \mapsto i^T_\alpha(f)$ is the restriction of a
polynomial exponential function whose purely polynomial part is a constant. This constant is denoted by $i_{\alpha}(f^+)$. 

We have defined an integral transform
    \[
    -^\dag: \cS(G_+(\bA)) \to \cS(G^+(\bA)), \quad f_+ \mapsto f^+ = (f_+)^{\dag}.
    \]
This is essentially a partial Fourier transform. We will encounter a local analogue of it in Part II of the paper. With this we defined the distributions
$I_{\alpha}^T$ and $I_{\alpha}$ on $\cS(G_+(\bA))$ by
    \[
    I_{\alpha}^T(f_+) = i_{\alpha}^T(f^+), \quad
    I_{\alpha}(f_+) = i_{\alpha}(f^+).
    \]
They are the geometric terms appearing in the relative trace formulae.

\subsection{Geometric distributions on the unitary groups}  \label{subsec:recall_U}

Like the case of general linear groups, we introduced Jacobi groups in~\cite{BLX1}.

Let $\Res V$ be the symplectic space over $F$ whose underline vector space is
$V$, and the symplectic form $\mathrm{Tr}_{E/F} \circ q_V$.
Let $S(V) = \Res V \times F$ be the Heisenberg group $J(V) = S(V) \rtimes \U(V)$ the Jacobi group. The group $J(V)$ can be realized as a subgroup of $\U(V + \bH)$ where $\bH$ is the two dimensional split skew-hermitian space. If we realize $\U(V+ \bH)$ as matrices, then $J(V)$ consists of the matrices of the form
    \[
    \begin{pmatrix} 1 &\bullet &\bullet &\bullet &\bullet \\
    &*&*&*&\bullet\\ &*&*&*&\bullet \\
    &*&*&*&\bullet \\ &&&& 1\end{pmatrix}.
    \]
where $\bullet$'s give the Heisenberg group $S(V)$, and $*$'s give $\U(V)$.

Let $\cF_V$ be the set of standard D-parabolic subgroups for $J(V)$. Standard D-parabolic subgroup of $J(V)$ can be described in concrete terms as follows. We fix $P_0 = M_0 N_0$ a minimal parabolic subgroup of $\U(V)$. Let $P$ and $P' = P \cap \U(V)$. Then $P$ is called standard if $P'$ contains $P_0$. If $P'$  If $P'$ stabilizes an isotropic flag of the form 
\begin{equation} \label{eq:isotropic flag}
0=X_0 \subset X_1 \subset \cdots \subset X_r
    \end{equation}
in $V$. Then the corresponding $P \in \cF_V$ is
given by $P = (X_r^\perp \times F) \rtimes P'$. We showed in~\cite{BLX1}*{Lemma~4.7} that $P \mapsto P'$ is a bijection between $\cF_V$ and the set of standard parabolic subgroups of $\U(V)$. In terms of matrices, a D-parabolic subgroups of $J(V)$ looks like
    \[
    \begin{pmatrix} 1 & &\bullet &\bullet &\bullet \\
    &*&*&*&\bullet\\ &&*&*&\bullet \\
    &&&* \\ &&&& 1\end{pmatrix}, \qquad
    \begin{pmatrix} 1 & &\bullet & &\bullet \\
    &*&&&\\ &&*&&\bullet \\
    &&&* \\ &&&& 1\end{pmatrix},
    \]
and its D-Levi subgroup (the right matrix) is the product of $\GL$s and a Jacobi group attached to a smaller skew-hermitian space $V_0$. Note that $X_r^{\perp} = X_r + V_0$.

For $P = MN \in \cF_V$, we let $P_{\mathrm{\U}} = P' \times P'$ be a parabolic subgroup of $\U_V$. 

To state the geometric side of the relative trace formula, we recall the following notation from~\cite{BLX1}*{Section~8}.

\begin{itemize}
\item Let $A_0$ be the maximal split torus in $P_0$ and $\fa_0 = \fa_{A_0}$.

\item We put $\U_V^+ = \U_V \times V$ and give it the product group structure. 
    The group $\U_V' \times \U_V'$
    acts on $\U_V^+$ from the right by $(g, v) \cdot (x, y) = (x^{-1} g y, y^{-1} v)$.

\item  Put 
        \[
        P^+ = P_{\U} \times X_r^{\perp}, \quad
        M_P^+ = M_{P_{\U}} \times V_0, \quad
        N_P^+ = N_{P_{\U}} \times X_r .
        \]
These are subgroups of $\U_V^+$.

\item By~\cite{CZ}*{Lemma~15.1.4.1} the categorical quotient $\U_V^+//(\U'_V \times \U'_V)$ is canonically identified with $\cA = G^+//(H \times G')$. The canonical morphism $q_V: \U_V^+ \to \cA$ is given by
        \[
        ((g_1, g_2), v) \mapsto (a_1, \hdots, a_n; b_1, \hdots, b_n)
        \]
    where
        \[
        a_i = \Trace \wedge^i (g_1^{-1} g_2), \quad
        b_i = 2 (-1)^{n-1} \tau^{-1} q_V( g_1^{-1} g_2 v, v).
        \]
    Here $\cA$ is viewed as a locally closed subscheme of $\Res_{E/F} \mathbf{A}_{2n, E}$ as in the case of general linear groups. The choice of the extra factor $2 (-1)^{n-1} \tau^{-1}$ in
$b_i$'s will be justified later when we compare the relative trace formulae, cf.~\ref{remark:Vtau}.

\item For $\alpha \in \cA(F)$, let $\U^+_{V, \alpha}$ be the preimage of $\alpha$ in
    $\U^+_V$ as a closed subscheme. We also put
        \[
        M^+_{P, \alpha} = \U^+_{V, \alpha} \cap M^+_P.
        \]
\end{itemize}

For $f^+ \in \cS(\U_V^+(\bA))$, $P \in \cF_V$, and $\alpha \in \cA(F)$, we defined the kernel function
    \[
    k_{f^+,P,\alpha}(x,y) = \sum_{m^+ \in M_{P,\alpha}^{+}(F)}
    \int_{N_P^+(\bA)} f^+(m^+ n^+ \cdot (x,y)) \rd n^+.
    \]
For $T \in \fa_0$ and $x, y \in [\U_V']$ and $\alpha \in \cA(F)$ we put
    \[
     k_{f^+,\alpha}^T(x,y) =
    \sum_{P \in \cF_V} \epsilon_P
    \sum_{\substack{\gamma \in P'(F) \backslash \U_V'(F)
    \\
    \delta \in P'(F) \backslash \U_V'(F)}}
    \widehat{\tau}_{P'}(H_{P'}(\delta y)-T_{P'})
    k_{f^+, \alpha, P}(\gamma x, \delta y).
    \]
Here again $\epsilon_P$ is a certain sign attached to $P$, $T \in \fa_{0}$ is a truncation parameter and $\widehat{\tau}(...)$ is a the characteristic function of a certain cone in $\fa_{0}$. Again the definitions of the unexplained terms are in~\cite{BLX1}*{Section~3.5}, but for the purpose of this paper, we do not need the precise definitions of them.

We defined a distribution
    \[
        j^T_\alpha(f^+) = \int_{[\U_V'] \times [\U_V']}
        k_{f^+, \alpha}^T(x,y) \rd x \rd y.
        \]
When $T$ is sufficiently regular, $j^T_\alpha$ coincides with the restriction of an exponential-polynomial function of $T$ whose purely polynomial term is a constant $j_\alpha(f^+)$.

We have defined an integral transform
    \begin{equation}    \label{eq:global_ddag_map}
    -^\ddag: \cS(\U_{V,+}(\bA)) \to \cS(\U_{V}^+((\bA)), \quad f_+ \mapsto f^+ = (f_+)^{\ddag}.
    \end{equation}
This is essentially a partial Fourier transform. We will encounter a local analogue of it in Part II of the paper. With this we defined the distributions
$J_{\alpha}^T$ and $J_{\alpha}$ on $\cS(\U_{V,+}(\bA))$ by
    \[
    J_{\alpha}^T(f_+) = j_{\alpha}^T(f^+), \quad
    J_{\alpha}(f_+) = j_{\alpha}(f^+).
    \]
They are the geometric terms appearing in the relative trace formulae.

\section{Infinitesimal geometric distributions} \label{sec:infinitesimal_geo}

\subsection{Simplification: general linear groups}
The goal of this subsection is to simplify the distribution $i_{\alpha}$.
Let us introduce the following additional spaces and group actions.

\begin{itemize}
\item We put
    \[
    S_n = \{ g \in G_n \mid  g g^{\mathsf{c}, -1} = 1  \},
    \]
which is a closed subvariety of $G_n$ over $F$. The group $G_n'$ acts on
$S_n$ from the right by the usual conjugation, and we have an isomorphism
of $G_n'$-varieties $S_n \simeq G_n/G_n'$.
Let
    \[
    \nu: G_n \to S_n, \quad \nu(g) = g g^{\mathsf{c}, -1}
    \]
be the natural projection.

\item Put $X = S_n \times L^{\vee, -} \times L^-$, on which $G_n'$ acts
    from the right by
    \[
    (\gamma, w, v)\cdot g'  = (g'^{-1} \gamma g', wg', g'^{-1} v).
    \]
    Put $G_{n}^+ = G_n \times L^{\vee, -} \times L^-$, on which $G'$ acts on the right by
    \[
        (g,w,v) \cdot (g_1',g_2') = (g_1^{\prime -1}gg_2',wg_1',g_1^{\prime -1}v).
    \]
    We extend the map
    $\nu$ to a map $G_n^+ \to X$ by setting $\nu(g, u, v) = (\nu(g), u,
    v)$. By composing this with the map $G_n \times G_n \to G_n$, $(g_1, g_2)
    \mapsto g_1^{-1} g_2$, we further obtain to a map $G^+ \to X$. Note
    that $G^+(\bA) \to X(\bA)$ is surjective.

\item Note that the quotient morphism $G/H \to      G_n$, $(g_1, g_2) \mapsto g_1^{-1}g_2$ and $\nu$ are both principal homogeneous spaces, the categorical quotients $G_n^+//G'$, $X//G_n'$ and $G^+//(H
    \times G')$ are canonically identified, and are all denoted by $\cA$.
    The categorical quotient $q: G^+ \to \cA$ factors through
    the map $\nu: G^+ \to X$, and we denote again by $q:X \to \cA$ the
    induced map. If $\alpha \in \cA(F)$, we let $X_{\alpha}$ be the inverse
    image of $\alpha$ in $X$ (as a closed subscheme).

\item Let          $P = MN \in \cF$, write $P_n = P \cap G_n = M_n
    N_n$. We put
    \[
        M_n^+ = M_n \times M_{L^\vee}^- \times M_{ L}^-, \quad   N_n^+ = N_n \times N_{L^\vee}^{\vee,-} \times N_{L}^{-}.
    \]
    The right action of $G_n'$ on $G_n^+$ restricts to an action of $M_n'$ on $M_n^+$.

    \item Composing the embedding $M_n^+ \subset G_n^+$ with the map $G_n^+
        \to \cA$ we obtain a map $M_n^+ \to \cA$, and for any $\alpha \in
        \cA(F)$ we denote by $M^+_{n, \alpha}$ the inverse image of
        $\alpha$ in $M^+_n$ (as a closed subscheme).
\end{itemize}

For $\varphi \in \cS(X(\bA))$, define a function $k_{\varphi,P,\alpha}$ on
$[G_n']_{P_n'}$ by
    \[
    k_{\varphi,P,\alpha}(h) =
    \sum_{m \in M_{n}^+(F) \cap X_\alpha(F)} \int_{N_{n}^+(\bA)/N'_n(\bA)}
    \varphi(h^{-1}\nu(x n)h) \rd n,
    \]
where $x$ is any element in $M_{n}^+$ such that $\nu(x)=m$, and $N_n^+/N_n'$ stands for $N_n/N_n' \times N_{P,L^\vee}^{\vee,-} \times N_{P,L}^{-}$.  The same argument
as~\cite{BLX1}*{Lemma~6.1} gives that
    \[
    \sum_{\alpha \in \cA(F)} \sum_{m \in M_{n,\alpha}^+(F)}
    \int_{N_{n}^+(\bA)/N_n'(\bA)}
    \lvert \varphi(h^{-1}\nu(x n)h) \rvert \rd n
    \]
is convergent.

Let us define an integral transform
    \[
    \cS(G^+(\bA)) \to \cS(X(\bA)), \quad
    f^+ \mapsto \varphi_{f^+}
    \]
as follows. For any $(\gamma, w, v) \in X(\bA)$ we put
    \begin{equation}    \label{eq:simplified_test_function_GL}
    \varphi_{f^+}(\gamma, w, v) =
    \left\{ \begin{aligned}
    &\int_{G_n(\bA)} \int_{G_n'(\bA)} f^+((h^{-1}, h^{-1} x g_n' ), w, v)
    \rd h \rd g_n', &&\text{$n$ odd};\\
    &\int_{G_n(\bA)} \int_{G_n'(\bA)} f^+((h^{-1}, h^{-1}x g_n'), w, v)
    \mu(\det xg_n')
    \rd h \rd g_n', &&\text{$n$ even}.
    \end{aligned} \right.
    \end{equation}
where $x \in G_n(\bA)$ is any element such that $\nu(x) = \gamma$. We note that
because the map $G^+ \to X$ is a smooth morphism, the
integral transform $f^+ \mapsto \varphi_{f^+}$ is surjective,
cf.~\cite{AG}*{Theorem~B.2.4}. If $\tS$ is a finite set of places, and $f^+
\in \cS(G^+(F_{\tS}))$, then we define $\varphi_{f^+} \in \cS(X(F_{\tS}))$ in
the same way, integrating over $F_{\tS}$-points of the groups.

The next lemma gives a relationship between the kernel functions $k_{f^+}$ and $k_{\varphi_{f^+}}$ for $f^+ \in \cS(G^+(\bA))$.  For elements or subgroups in $G' = G_n' \times G_n'$, we use the subscripts $1$ and $2$ to indicate the first and the second components.

\begin{lemma}   \label{lem:geometric_linear_convergence}
For all $\alpha \in \cA(F)$ and $f^+ \in \cS(G^+(\bA))$, the expression
    \[
    \int_{[H]} \int_{[G_n']}
    \sum_{ \substack{ \gamma \in P_H(F) \backslash H(F) \\
    \delta_2 \in P_2'(F) \backslash G_n'(F) }}
    \sum_{m \in M_{\alpha}^+(F)} \int_{N^+(\bA)}
    \left| f^+(mn \cdot (\gamma h, \delta_2 g')) \right| \rd n \rd g_2' \rd h
    \]
is convergent and we have
    \begin{equation}    \label{eq:geo_descent_GL_step_1}
    \int_{[H]} \int_{[G_n']}
    \sum_{ \substack{ \gamma \in P_H(F) \backslash H(F) \\
    \delta_2 \in P_2'(F) \backslash G_n'(F) }}
    k_{f^+, P, \alpha}(\gamma h,\delta_2 g') \eta_{G'}(g') \rd g_2' \rd h =
    k_{\varphi_{f^+}, P, \alpha}(g_1') \eta(g_1').
    \end{equation}
Here we recall that $g' = (g_1', g_2') \in G'(\bA)$, $g_1', g_2' \in G_n'(\bA)$, and $\delta_2 g'$ stands for the element $(g_1', \delta_2 g_2') \in G'(\bA)$. The integration is over $h$ and $g_2'$, and we left with a function on $g'_1$.
\end{lemma}

\begin{proof}
We will prove~\eqref{eq:geo_descent_GL_step_1}. The absolute convergence can
be obtained by the same computation, with only obvious minor modifications.

We first unfold the sum over $\gamma$ and $\delta$ to conclude
that~\eqref{eq:geo_descent_GL_step_1} equals
    \[
    \int_{[H]_{P_H}} \int_{[G_n']_{P_{n}'}}
    k_{f^+, P, \alpha}(h,  g') \eta_{G'}(g') \rd g_2' \rd h.
    \]

We fix good maximal compact subgroups $K_H$ and $K_2'$ of $H(\bA)$ and $G_n'(\bA)$ respectively. By the Iwasawa decomposition the left hand side
of~\eqref{eq:geo_descent_GL_step_1} equals
    \begin{align*}
    \eta(\det g_1')^{n+1} &
    \sum_{m \in M_{\alpha}^+(F)}
    \int_{K_H}
    \int_{K_2'} \int_{[M_{P_H}]} \int_{[M_n']} \int_{N^+(\bA)}
     e^{\langle
    -2\rho_{P_H},H_{P_H}(m_H)\rangle+\langle -2\rho_{P_n'}, H_{P_n'}(m')
    \rangle }
    \\ &f^+ (mn \cdot (m_H k_H, (g_1',m'k')))
     \eta(\det m'k')^{n+1} \rd n \rd m'  \rd m_H  \rd k' \rd k_H.
    \end{align*}
Unfolding the sum over $m \in M_{\alpha}^+(F)$ we conclude that this equals
    \begin{align*}
    &\eta(\det g_1')^{n+1}
    \sum_{m \in M^+_{\alpha}(F)/(M_{P_H}(F)  \times
    M_2'(F))}
    \int_{K_H} \int_{K_2'} \int_{M_{P_H}(\bA)}
    \int_{M_n'(\bA)} \int_{N^+(\bA)}
      \\
    &e^{\langle
    -2\rho_{P_H},H_{P_H}(m_H)\rangle
    +\langle -2\rho_{P_n'},H_{P_n'}(m') \rangle }
    f^+ \left(nm \cdot (m_H k_H, (g_1',m'k'))\right)
    \eta(\det m'k')^{n+1} \rd n \rd m'  \rd m_H \rd k' \rd k_H.
    \end{align*}

We have
    \[
    M^+_{\alpha}(F)/ (M_{P_H}(F)\times
    M_n'(F))=M^+_{n,\alpha}(F)/M_n'(F).
    \]
We also decompose the integral over $N^+(\bA)$ into an integral over ${N_H(\bA)}$ and $N_n^+(\bA)$. An element in $N^+_n(F)$ is considered as an element in $N^+(F)$ via
$(g, x, y) \mapsto ((1, g), x, y)$. Then the above integral equals
    \begin{align*}
    \eta(\det g_1')^{n+1}
    &\sum_{m \in M^+_{n,\alpha}(F)/M_n'(F)}
    \int_{K_H} \int_{K_2'}
    \int_{M_{P_H}(\bA)} \int_{M_n'(\bA)}
    \int_{N_H(\bA)} \int_{N_{n}^+(\bA)}   \\
    & e^{\langle
    -2\rho_{P_H},H_{P_H}(m_H)\rangle+
    \langle -2\rho_{P_n'},H_{P_n'}(m') \rangle }
    f^+ \left(n_H m n_2 \cdot (m_H k_H, (g_1',m'k'))\right)\\
    &\eta(\det m'k')^{n+1} \rd n_2 \rd n_H \rd m' \rd m_H \rd k' \rd k_H.
    \end{align*}
We further break up the integral over $N_{n}^+(\bA)$ into
$N_{n}^+(\bA)/N_n'(\bA)$ and $N_n'(\bA)$ and arrive at
    \begin{equation}    \label{eq:descent_simplification_final_step}
    \begin{aligned}
    \eta(\det g_1')^{n+1}
    &\sum_{m \in M^+_{n,\alpha}(F)/M_n'(F)}
    \int_{K_H} \int_{K_2'}
    \int_{M_{P_H}(\bA)} \int_{M_n'(\bA)}
    \int_{N_H(\bA)} \int_{N_{n}^+(\bA)/N_n'(\bA)} \int_{N_n'(\bA)}   \\
    & e^{\langle
    -2\rho_{P_H},H_{P_H}(m_H)\rangle+
    \langle -2\rho_{P_n'},H_{P_n'}(m') \rangle }
    f^+ \left(n_H m n_2 n_2' \cdot (m_H k_H, (g_1', m'k'))\right)\\
    &\eta(\det m'k')^{n+1}
    \rd n_2' \rd n_2 \rd n_H \rd m' \rd m_H \rd k' \rd k_H.
    \end{aligned}
    \end{equation}
Using the Iwasawa decomposition, we combine the integrals for $m', n'$ and
$k'$ as an integral over $G_n'(\bA)$, and the integrals over $m_H, n_H$ and
$k_H$ as an integral over $H(\bA)$. Note that $\mu(\det mn_2)=1$ for $m \in
M_n'(F)$ and $n \in N_n(\bA)$. We thus conclude
that~\eqref{eq:descent_simplification_final_step} equals
    \[
    \eta(\det g_1') \sum_{m \in M^+_{n,\alpha}(F)/M_n'(F)}
    \int_{N_{n,+}(\bA)/N_n'(\bA)} \varphi_{f^+}
    \left(  g_1'^{-1}\nu(mn_2) g_1' \right)
    \rd n_2.
    \]
Note that there is a slight difference in the definition of $\varphi_{f^+}$
for $n$ even or odd, which eventually simplifies to a uniform expression.
Finally, the map $\nu$ induces a natural bijection between $M_{n,
\alpha}^+(F)/M_n'(F)$ and $M_{n}^+(F) \cap X_{\alpha}(F)$. This concludes
the proof.
\end{proof}

\subsection{Infinitesimal variant}
\label{subsec:infinitesimal_GL}

We use gothic letters to denote the Lie algebra of the
corresponding group, e.g. if $G$ is a group over $F$ then we denote by $\fg$ its Lie algebra (which as an algebraic variety is an affine space over $F$).
We slightly extend this convention to the case of an algebraic
variety with a distinguished point $e$ where
we use the corresponding gothic
letter to denote its tangent space at $e$. For instance, $G^+ = G \times L^{\vee, -} \times L^{-}$ is a variety with a distinguished point $(1, 0, 0)$, and $\fg^+ = \fg \times L^{\vee, -} \times L^{-}$ stands for the tangent space at this point.

We now introduce an infinitesimal variant of the distributions $i_\alpha$,
and relate $i_{\alpha}$ to it. This infinitesimal variant in fact is essentially the same
as the one that arises from the Jacquet--Rallis relative trace formulae,
which are used to study the Bessel periods on unitary groups.

We define $\fs_n := M_n(E)^-$, the subspace of $M_n(E)$ with pure imaginary entries, it can be identified with the tangent space of $S_n$ at $1$. Let $\fx = \fs_n \times L^{\vee, -} \times L^-$, viewed as an
algebraic variety over $F$. We write an element in it as $(\gamma, w, v)$. It
admits a right action of $G_n'$ given by
    \begin{equation}    \label{eq:action_s_infinitesimal}
    (\gamma, w, v) \cdot g' = (g'^{-1}\gamma g', wg', g'^{-1}v), \quad
    (\gamma, w, v) \in \fx, \quad g' \in G'_n(F).
    \end{equation}

Let $\fx \to \cB = \fx //G_n'$ be the categorical quotient. A concrete description of $\cB$ is given by~\cite{Zhang1}*{Lemma~3.1}. The categorical quotient $\cB$ is isomorphic
to the affine space over $F$ of dimension $2n$. We will identify $\cB$ with a
closed subspace of the affine space over $E$ of dimension $2n$, consisting of
elements
    \[
    (a_1, \hdots, a_n; b_1, \hdots, b_n)
    \]
satisfying
    \[
    a_i^{\mathsf{c}} = (-1)^i a_i, \quad b_i^{\mathsf{c}} = (-1)^{i-1} b_i, \quad
    i = 1, \hdots, n.
    \]
The quotient map $q: \fx \to \cB$ is given by
    \[
    (\gamma, w, v) \mapsto (a_1, \hdots, a_n;
    b_1, \hdots, b_n)
    \]
where
    \[
    a_i = \Trace \wedge^i \gamma, \quad b_i = w \gamma^{i-1} v, \quad i = 1, \hdots, n.
    \]
For any $\alpha \in \cB(F)$, we denote by $\fx_{\alpha}$ the inverse image of
$\alpha$ in $\fx$ (as a closed subscheme).

Let $P = MN \in \cF$, we put $\fx_M = \fm_n^+ \cap \fx$ and $\fx_N = \fn_n^+ \cap \fx$. For $\alpha \in \cB(F)$, we put $\fx_{M,\alpha} = \fx_M \cap \fx_\alpha$. Here we view all spaces as subspaces of $\Res_{E/F} \gl_{n, E} \times L^{\vee, -} \times L^-$ and take the intersections in it.

Take $\varphi \in \cS(\fx(\bA))$ and $\alpha \in \cB(F)$. For $g' \in
[G'_n]_{P_{n}'}$ we define a kernel function
    \[
    k_{\varphi, P, \alpha}(g') = \sum_{m \in \fx_{M,\alpha}(F)}
    \int_{\fx_N(\bA)} \varphi((m+n) \cdot g') \rd n.
    \]
For $T \in \fa_{n+1}'$, we define a modified kernel function
    \[
    k_{\varphi, \alpha}^T(g') = \sum_{P \in \cF} \epsilon_P
    \sum_{\delta \in P'_n(F) \bs G'_n(F)}
    \widehat{\tau}_{P_{n+1}'}(H_{P_{n+1}'}(\delta g') - T_{P_{n+1}'})
    k_{\varphi, P, \alpha}(\delta g').
    \]
As in the case of other modified kernels, for a fixed $g'$ the sum over $\delta$ is finite.

The next proposition summarizes the main results of~\cite{Zydor2}, see also
Th\'{e}or\`{e}mes~5.1.5.1 and~5.2.1.1 of~\cite{CZ}.

\begin{prop}    \label{prop:infinitesimal_GL}
We have the following assertions.
\begin{enumerate}
\item The expression
    \[
    \sum_{\alpha \in \cB(F)} \int_{[G'_n]}
   \lvert  k_{\varphi, \alpha}^T(g') \rvert \rd g'
    \]
    is finite when $T$ is sufficiently positive. 

\item We put
    \[
    \mathfrak{i}_{\alpha}^T(\varphi) =
    \sum_{\alpha \in \cB(F)} \int_{[G'_n]}
    k_{\varphi, \alpha}^T(g') \eta(\det g') \rd g',
    \]
 then the function $T \mapsto \mathfrak{i}_{\alpha}^T(\varphi)$
    is the restriction of a exponential-polynomial whose purely polynomial
    term is a constant which we denote by $\mathfrak{i}_{\alpha}(\varphi)$.

\item The distribution $\varphi \mapsto \mathfrak{i}_{\alpha}(\varphi)$ is
    a continuous linear form on $\cS(\fx(\bA))$ that is $(G'_n, \eta)$
    invariant. This means
        \[
        \mathfrak{i}_{\alpha}(g' \cdot \varphi) = \eta(\det g')
        \mathfrak{i}_{\alpha}(\varphi)
        \]
    for all $g' \in G'_n(\bA)$ where $g' \cdot \varphi(x) =
    \varphi(g'^{-1} x g)$.
\end{enumerate}
\end{prop}

We now introduce the Cayley transform. Let $E^1 = \{ x \in E \mid N_{E/F}x =
x x^{\mathsf{c}} = 1\}$ and take a $\xi \in E^1$. We put
    \[
    X^{\xi} = \{ (\gamma, w, v) \in X \mid \gamma + \xi \text{ is invertible} \}
    \]
Then $X^{\xi}$ is an open subscheme of $X$ and the action
of $G_n'$ preserves it. We also put
    \[
    \fx'
    = \{(A, w, v) \mid A^2 - 4 \text{ is invertible}\}.
    \]
Then $\fx'$ is an open subscheme of $\fx$ and the action of $G_n'$
preserves it.

We define a Cayley transform
    \[
    \fc_{\xi}: \fx' \to X^{\xi}, \quad
    (A, w, v) \mapsto \left( \xi(1 + A/2)(1-A/2)^{-1}, w, v \right).
    \]
Note that $\fc_\xi$ is $G_n'$-equivariant and there are $\xi_1, \hdots, \xi_{n+1} \in E^1$, such that the images
of $\fc_{\xi_i}$ form an open cover of $X$. Let $\cB'$ be the image of
$\fx'$ in $\cB$ and let $\cA^\xi$ be the image of $X^\xi$ in $\cA$. Then the Cayley transform $\fc_{\xi}$ induces an isomorphism $\cB' \to \cA^\xi$,
which we still denote by $\fc_{\xi}$.

Let $\tS$ be a finite set of places. For $\varphi \in \cS(X^{\xi}(F_{\tS}))$ we define a function $\varphi_{\natural} \in \cS(\fx'(F_{\tS}))$ as follows. If $n$ is
odd, we put
    \[
        \varphi_\natural(A, w, v) =  \mu(\xi_1)^{-(n-1)^2} \varphi(\fc_{\xi}(A), w, v).
    \]

If $n$ is even, we put
    \[
         \varphi_{\natural}(A, v, w) = \mu(\xi_1)^{-n(n-2)}\mu(\det(1-A/2))
     \varphi(\fc_{\xi}(A), w, v).
    \]

Then $\varphi_{\natural} \in \cS(\fx'(F_{\tS}))$. If $f^+ \in
\cS(G^+(F_{\tS}))$ then we write $f^+_{\natural} =
(\varphi_{f^+})_{\natural}$. The various extra factors $\mu(\det(1-X/2))$ and
powers of $\mu(\xi_1)$ will be posteriori justified by
Lemma~\ref{lemma:Lie_algebra_matching}. They all come from the comparison
between the transfer factors on $G^+(F_{\tS})$ and on $\fx(F_\tS)$.

\begin{prop} \label{prop:infitesimal_decent_gl}
Fix an $\alpha \in \cB'(F)$. Let $\tS$ be a finite set of places including all
Archimedean places, such that for $v \not \in \tS$ we have the following properties.
\begin{itemize}
    \item $E/F$ is unramified at $v$.

    \item $\mu$ is unramified at any places $w$ of $E$ above $v$.

    \item $2, \tau, \xi_1$ are in $\cO^\times_{E_w}$ for any place of $w$
        of $E$ above $v$.

    \item $\alpha \in \cB'(\cO_{F_v})$.
\end{itemize}
Then for any $f_\tS^+ \in \cS(G^{+}(F_\tS))$ and $T \in \fa_0$ sufficiently
positive, we have
    \[
    i_{\fc_{\xi}(\alpha)}^T\left(f^+_\tS \otimes 1_{G^+(\cO_F^\tS)}\right) =
    \mathfrak{i}_\alpha^T\left((f_{\tS}^+)_\natural
    \otimes 1_{\fg^+(\cO_F^\tS)} \right).
    \]
\end{prop}

\begin{proof}
We write $a$ for $\fc_\xi(\alpha)$ and let $f^+ := f_\tS^+ \otimes 1_{G^+(\cO_F^\tS)} \in \cS(G^+(\bA))$. It follows from the definitions that $\varphi_{f^+} =
    \varphi_{f^+_{\tS}} \otimes \id_{X(\cO_F^{\tS})}$.
Then we have
\begin{align*}
    i_a^T(f^+) &= \int_{[H]} \int_{[G']} k_f^T(h,g') \eta_{G'}(g') \rd h \rd g' \\
    &= \int_{[G_n']} k^T_{\varphi_{f^+}, a}(g') \eta(\det g') \rd g' \quad (\text{By Lemma \ref{lem:geometric_linear_convergence}}) \\
    &= i^T_a(\varphi_{f^+_\tS} \otimes 1_{X(\cO_F^\tS)}).
\end{align*}

We are reduced to show that
    \[
    k_{\varphi \otimes \id_{X(\cO_F^{\tS})}, P, a}
     =
    k_{\varphi_{\natural} \otimes \id_{\fx(\cO_F^{\tS})}, P, \alpha}
    \]
for all $\alpha \in \cB(F)$, $P \in \cF_{\mathrm{RS}}$ and $\varphi \in
\cS(X(F_{\tS}))$. This follows from the discussion
in~\cite{Zydor3}*{Section~5}, cf.~\cite{Zydor3}*{Corollaire~5.8}.
\end{proof}

Let $\tS$ be a finite set of places and $f_{+, \tS} \in \cS(G_+(F_\tS))$ then
we put
    \[
    f_{+, \tS, \natural} = \left( f_{+, \tS}^\dag \right)_{\natural}.
    \]
We have the following corollary of Proposition~\ref{prop:infitesimal_decent_gl}.
\begin{coro}    \label{coro:infinite_descent_GL}
Let $\tS$ be a finite set of places satisfying the conditions in
Proposition~\ref{prop:infitesimal_decent_gl}, and $f_+ =
\id_{G_+(\cO_F^{\tS})} \otimes f_{+, \tS}$ where $f_{+, \tS} \in
\cS(G_+(F_{\tS}))$, then
    \[
    I_{\fc_{\xi}(\alpha)}^T\left(f_{+, \tS} \otimes 1_{G_+(\cO_F^\tS)}\right) =
    \mathfrak{i}_\alpha^T\left(f_{+, \tS, \natural}
    \otimes 1_{\fg^+(\cO_F^\tS)} \right)
    \]
\end{coro}

\subsection{Infinitesimal distributions: unitary groups} \label{subsec:infinitesimal_U}
The geometric distributions $J_{\alpha}^T$ can be simplified and related to
their infinitesimal invariants as in the case of general linear groups. We
only state the results but omit the proofs, as they are essentially the same as the general
linear group case.

We retain the notation from Subsection~\ref{subsec:recall_U}.
Let $V$ be an $n$-dimensional skew-Hermitian space. We put $\fu_V^+ = \fu_V \times V$, $Y^V =
\U(V) \times V$, and $\fy^V = \fu(V) \times V$. For $P \in \cF_V$, we define 
    \[
    \fp^+ = \fp \times X_r^{\perp}, \quad \fm^+ = \fm \times V_0, \quad 
    \fn^+ = \fn \times X_r.
    \]
They are subspace of $\fy^V$.

The group $\U(V)$ has a right action on $Y^V$ or $\fy^V$ by $(A, v) \cdot
g=(g^{-1}Ag, g^{-1} v)$. The categorical quotient $q_V: \U_V^+ \to \cA$ factors through $Y^V$ and identifies the categorical quotient $Y^V//\U(V)$ with $\cA$. By~\cite{Zhang1}*{Lemma~3.1} the categorical quotient $\fy^V//\U_V'$ is identified with $\cB$.
The natural maps $Y^V \to \cA_V$ and $\fy^V \to \cB$, both denoted by $q_V$,
are given by (the same formula)
    \[
    (A, b) \mapsto (a_1, \hdots, a_n; b_1, \hdots, b_n),
    \]
where
    \[
    a_i = \mathrm{Tr} \wedge^i A, \quad
    b_i = 2 (-1)^{n-1} \tau^{-1} q_V( A^{i-1} b, b), \quad i = 1, \hdots, n.
    \]
For $\alpha \in \cA(F)$ (resp. $\cB(F)$), we denote by $Y^V_{\alpha}$ (resp.
$\fy^V_{\alpha}$) the inverse image of $\alpha$ in $Y^V$ (resp. $\fy^V$) (as
closed subvarieties). 

For standard parabolic $P \in \cF_V$ and $f \in \cS(\fy^V(\bA))$, define a
kernel function on $[\U(V)]_P$ by
    \[
    k_{f,P}(g)=
    \sum_{m \in \fm^+(F)} \int_{\fn_+(\bA)} f((m+n) \cdot g) \rd n.
    \]
For $\alpha \in \cA(F)$, we define
    \[
    k_{f,P,\alpha}(g)
    =\sum_{m \in \fm^+(F) \cap \fu^+_{V, \alpha}(F)} \int_{\fn_+(\bA)}
    f((m+n)\cdot g) \rd n.
    \]

For $f \in
\cS(\fu^+(\bA))$ and $T \in \fa_0$, we put
    \[
    k_{f}^T(g) = \sum_{P \in \cF_V} \epsilon_P
    \sum_{\delta \in P(F) \backslash \U(V)(F)}
    \widehat{\tau}_P(H_P(\delta g)-T_P) k_{f,P}(\delta g),
    \]
where $g \in [\U(V)]$. Similarly, for $\alpha \in \cB(F)$, put
    \[
    k_{f,\alpha}^T(g) =
    \sum_{P \in \cF_V} \epsilon_P \sum_{\delta \in P(F) \backslash \U(V)(F)}
    \widehat{\tau}_P(H_P(\delta g)-T_P)
    k_{f,P,\alpha}(\delta g).
    \]

The following theorem summarizes~\cite{Zydor2}*{Theorem 3.1,Theorem 4.5}.

\begin{theorem}
For $T$ sufficiently positive, we have the following assertions.
\begin{enumerate}
    \item The expression
    \[ \sum_{\alpha \in \cB(F)} \int_{[\U(V)]}
    \left|  k^T_{f,\alpha}(g) \right| \rd g
    \]
    is finite.

    \item For $\alpha \in \cB(F)$ and $f \in \cS(\fu_V^+(\bA))$ , the
        function of $T$
        \[
	    \mathfrak{j}^{T}_\alpha(f) =
         \int_{[\U(V)]} k_{f,\alpha}^T(g) \rd g
        \]
    is a restriction of an exponential polynomial, and its purely
    polynomial term is constant, denoted by $\mathfrak{j}_{\alpha}(f)$.

\item The distribution $f \mapsto \mathfrak{j}_{\alpha}(f)$ is a continuous
    linear form on $\cS(\fu_V^+(\bA))$ that is $H$ invariant.
\end{enumerate}
\end{theorem}

We note relate the geometric distributions on $\U_V^+$ and on $\fy^V$. If
$f^+ \in \cS(\U_V^+(\bA))$ we put
    \begin{equation}    \label{eq:simplified_test_function_U}
    \varphi_{f^+}(\delta, b) = \int_{\U(V)(\bA)}
    f^+((g^{-1}, g^{-1} \delta ), v) \rd g.
    \end{equation}
Then $\varphi_{f^+} \in \cS(Y^V(\bA))$.

Let $Y_{V}^{\xi}$ be the open subscheme of $Y_V^+$ consists of $(g,v)$
such that $\det(g-\xi) \not= 0$. Denote by $\U_{V, \xi}$ the open subvariety
of $\U_V^+$ consisting of elements $(g, v)$ such that $\det(g - \xi) \not=0$,
and by $\fy'_{V}$ the open subvariety of $\fy_V$ consists of
$(A,b)$ such that $\det(A- 2) \ne 0$. Then the (scheme-theoretic) images of
$Y_{V}^{\xi}$ and $\fy'_{V}$ in $\cA$ and $\cB$ respectively are denoted by
$\cA^\xi$ and $\cB'$.

We define the Cayley transform
    \[
    \fc_\xi: \fy'_{V} \to Y_{V}^{\xi}, \quad
    (A,b) \mapsto \left( \xi (1+A/2)(1-A/2)^{-1},
    b \right) \in \U_V^{+, \xi}.
    \]
It induces a map $\fc_\xi: \cB' \to \cA^\xi$. We use the same notation $\fc_{\xi}$ for Cayley transforms on general linear groups and unitary groups, it nevertheless should cause no confusion.

Let $\tS$ be a finite set of places. For $\varphi \in \cS(Y_V^{ \xi}(F_\tS))$,
define $\varphi_\natural \in \cS(\fy'_V(F_S))$ by
\[
    \varphi_\natural(X) =  \varphi(\fc_\xi(X)).
\]

If $f^+ \in \cS(\U_V^+(F_\tS))$, then we write $f^+_{\natural} =
(\varphi_{f^+})_{\natural}$.

\begin{prop} \label{prop:infitesimal_decent_u}
Fix $\alpha \in \cB(F)$. If $\tS$ is a set of places satisfying the same
conditions as Proposition~\ref{prop:infitesimal_decent_gl}. Then for any
$f_\tS \in \cS(\U_V^{+,\tau}(F_\tS))$ and $T \in \fa_0$ sufficiently
positive, we have
    \[
    j_{\fc_{\xi}(\alpha)}^T\left(f_\tS \otimes 1_{\U_V^+(\cO^S)}\right) =
    \mathfrak{j}_{\alpha}^T\left((f_{\tS})_{\natural}
    \otimes 1_{\fy^V(\cO^S)}\right).
    \]
\end{prop}

As the proof of Proposition~\ref{prop:infitesimal_decent_u}, this follows
from the discussion in~\cite{Zydor3}*{Section~5},
cf.~\cite{Zydor3}*{Corollaire 5.8}.

Let $\tS$ be a finite set of places and $f_{+, \tS} \in \cS(\U_{V,
+}(F_\tS))$ then we put
    \[
    f_{+, \tS, \natural} = \left( f_{+, \tS}^\ddag \right)_{\natural}.
    \]

\begin{coro}    \label{coro:infinite_descent_U}
Let $\tS$ be a finite set of places satisfying the conditions in
Proposition~\ref{prop:infitesimal_decent_gl}, and $f_+ =
\id_{\U_{V,+}(\cO_F^{\tS})} \otimes f_{+, \tS}$ where $f_{+, \tS} \in
\cS(\U_{V,+}(F_{\tS}))$, then
    \[
    J_{\fc_{\xi}(\alpha)}^T
    \left(f_{+, \tS} \otimes \id_{\U_{V,+}(\cO_F^\tS)}\right) =
    \mathfrak{j}_\alpha^T\left(f_{+, \tS, \natural}
    \otimes \id_{\fy^V(\cO_F^\tS)} \right)
    \]
\end{coro}

\section{Matching of test functions}
\label{sec:matching}

\subsection{Regular semisimple orbits}
Let $(\gamma, w, v)$ be an element in either $X(F)$ or $\fx(F)$. We say
that $(\gamma, w, v)$ is regular semisimple if and only if
    \[
    \det(w, w\gamma, \hdots, w \gamma^{n-1}),
    \quad \det(v, \gamma v, \hdots, \gamma^{n-1} v)
    \]
are both nonzero. Let $X_{\rs}$ and $\fx_{\rs}$ be the open subvariety
of $X$ and $\fx$ corresponding to the regular semisimple elements. The
actions of $G_n'$ on $X_{\rs}$ and $\fx_{\rs}$ are free. Let $(g, w,
v)$ be an element in $G^+(F)$. We say it is regular semisimple if its image
in $X(F)$ is regular semisimple. Let $G^+_{\rs}$ be the open subscheme of regular semisimple elements.

Let $(\delta, v)$ be an element in $Y^V(F)$ or $\fy^V(F)$. We say that it is
regular semisimple if
    \[
    v, \delta v, \hdots, \delta^{n-1} v
    \]
are linearly independent in $V$. Let $Y^V_{\rs}$ and $\fy^V_{\rs}$ be the
open subvarieties of $Y^V$ and $\fy^V$ respectively corresponding to the regular semisimple
elements. The actions of $\U(V)$ on $Y^V_{\rs}$ and $\fy^V_{\rs}$ are free.
Let $(g, v)$ be an element in $\U_V^+(F)$ and $g = (g_1, g_2)$, $g_i \in
\U(V)(F)$. We say it is regular semisimple if its image in $Y^V(F)$ is
regular semisimple. Let $\U^+_{V, \rs}$ be the open subvariety of regular
semisimple elements.

We say a regular semisimple element in $G^+(F)$ and a regular semisimple
element in $\U_V^+(F)$ match if their images in $\cA(F)$ are coincide.
Similarly we say a regular semisimple element in $\fx_{n, \rs}(F)$ and a
regular semisimple element in $\fu_{V, \rs}^+(F)$ math if their images in
$\cB(F)$ coincide. The match of regular semisimple elements depends only on
the orbits of the corresponding elements, and hence we can speak of the matching
of orbits.

\begin{remark}  \label{remark:Vtau}
We denote by $V_{\tau}$ the Hermitian space whose underline space is $V$ and
the Hermitian form is
    \[
    q_{V_{\tau}} = 2 (-1)^{n-1} \tau^{-1}  q_V.
    \]
Then $\U(V_\tau)$ and $\U(V)$ are physically the same group. Define
    \[
    \U_{V_\tau}^+ = \U_{V_{\tau}} \times V_{\tau},
    \quad \fy^{V_{\tau}} = \fu(V_{\tau}) \times V_{\tau}.
    \]
They are physically the same as $\U_{V}^+$ and $\fy^V$ respectively. We
define regular semisimple elements in $\U_{V_{\tau}}^+$ and $\fy^{V_{\tau}}$
as in the case of $V$. The spaces $\U_{V_{\tau}}^+$ and $\fy^{V_{\tau}}$ are
the ones appearing in the Jacquet--Rallis relative trace formulae that are
used to study the Bessel periods, cf.~\cites{CZ,Zhang1} for instance. While
the identification of $\fy^V$ and $\fy^{V_{\tau}}$, our notion of matching is
exactly the same as that in~\cites{BP1,CZ,Zhang1}. More precisely regular
semisimple elements $(\gamma, v, w) \in \fx(F)$ and $(\delta, b) \in \fy^V$
match if and only if they match in the sense of~\cites{BP1,CZ,Zhang1} when
$(\delta, b)$ is viewed as an element in $\fy^{V_{\tau}}$. This
identification will make the local calculations later easier to trace.
\end{remark}

We denote by $\cH$ the set of (isometric classes of) $n$-dimensional hermitian spaces over $E$. If $\tS$ is a finite set of places of $F$, we denote by $\cH_{\tS}$ the (isometric classes of) $n$-dimensional hermitian spaces over $E_{\tS}$.

\begin{prop}    \label{prop:regular_ss_orbits_matching}
The matching of regular semisimple orbits gives bijections
    \begin{equation}    \label{eq:regular_ss_orbits_matching}
    \begin{aligned}
     G^+_{\rs}(F)/H(F) \times G'(F) &\leftrightarrow \coprod_{V \in \cH}
    \U_{V, \rs}^+(F) /\U_V'(F) \times \U_V'(F),\\
    \fx_{\rs}(F)/G'_n(F) &\leftrightarrow \coprod_{V \in \cH} \fy^V_{rs}(F)/\U(V)(F).
    \end{aligned}
    \end{equation}
\end{prop}

Let $\tS$ be a finite set of places of $F$. Then we can define the matching
of regular semisimple orbits in $G^+_{\rs}(F_\tS)$ and in $\U_{V,
\rs}^+(F_\tS)$, and in $\fx(F_{\tS})$ and $\fy^V(F_{\tS})$ in the same way.
This again gives the bijection~\eqref{eq:regular_ss_orbits_matching}, with
$F$ replaced by $F_{\tS}$. We record it again for further references.
    \begin{equation}    \label{eq:matching}
    \begin{aligned}
     G^+_{\rs}(F_{\tS})/H(F_{\tS}) \times G'(F_{\tS}) &\leftrightarrow \coprod_{V \in \cH_\tS}
     \U_{V, \rs}^+(F_\tS) /\U_V'(F_\tS) \times \U_V'(F_\tS),\\
    \fx_{\rs}(F_{\tS})/G'_n(F_{\tS}) &\leftrightarrow
    \coprod_{V \in \cH_\tS} \fy^V_{\rs}(F_{\tS})/\U(V)(F_{\tS}),
    \end{aligned}
    \end{equation}

\subsection{Regular semisimple orbital integrals}
We introduce regular semisimple orbital integrals in this subsection. Let us begin by noting that all these definitions depend on various choices of the measures. This is also implicitly there when we introduce other terms in the relative trace formulae. We will prefix some choices of the measures, and only mention the specific measures when we need to, e.g. in the fundamental lemma.

Let $(g, w, v) \in G^+(F)$ be regular semisimple which maps to $\alpha
\in \cA(F)$. Let $f_+ \in \cS(G_+(\bA))$. Then the geometric distribution
$I_{\alpha}(f_+)$ simplifies to
    \begin{equation}    \label{eq:global_orb_GL}
    I_{\alpha}(f_+) = \int_{H(\bA)}\int_{G'(\bA)}
    f_+^\dag(h^{-1} g g', w g_2', g_2'^{-1} v)
    \eta_{G'}(g') \rd g' \rd h.
    \end{equation}
The integral on the right hand side depends only on $\alpha$ but not the
specific element $(g, w, v)$ that maps to it. This is the usual regular
semisimple orbital integral of $f_+$.

The global orbital integral factorizes. Take $f_{+, v} \in \cS(G_+(F_v))$ and let
$(g, w, v)$ be a regular semisimple element in $G^+(F_v)$, we define the
local orbital integral $O((g, w, v), f_{+, v})$ the same way by the same
formula~\eqref{eq:global_orb_GL}, integrating over $H(F_v)$ and $G'(F_v)$
instead. Then if $f_+ = \otimes f_{+, v}$ where $f_{+, v} \in \cS(G_+(F_v))$,
then
    \[
    I_{\alpha}(f_+) = \prod_v O((g, w, v), f_{+, v}).
    \]
Let $\tS$ be a finite set of places of $F$ and $f_{+, \tS} \in
\cS(G_+(F_\tS))$, then we can define the regular semisimple orbital integrals
$O((g, w, v), f_{+, \tS})$ in the same way.

For each $\alpha \in \cA(F)$ we have a distribution $J_{\alpha}$ on
$\cS(\U_V^+(\bA))$. We will consider this distribution for all $V$ at the
same time, so we add a supscript $V$ and write $J_{\alpha}^V$ to emphasize
the dependance on $V$. Let $(g, v) \in \U_V^+(F)$ be regular semisimple
which maps to $\alpha \in \cA(F)$. Let $f^{V}_+ \in \cS(\U_{V, +}(\bA))$.
Then the geometric distribution $J^V_{\alpha}(f^V_+)$ simplifies to
    \begin{equation}    \label{eq:global_orb_U}
    J^V_{\alpha}(f^V_+) = \int_{\U_V'(\bA)} \int_{\U_V'(\bA)}
    f^{V,\ddag}_+ (x^{-1} g y, y^{-1} v) \rd x \rd y.
    \end{equation}
The integral on the right hand side depends only on $\alpha$ but not the
specific element $(\delta,b)$ that maps to it. This is the usual regular
semisimple orbital integral of $f^{V}_+$.

The global orbital integral factorizes. Take $f^V_{+, v} \in
\cS(\U_{V,+}(F_v))$ and let $(\delta, b)$ be a regular semisimple element in
$\U_{V,+}(F_v)$, we define the local orbital integral $O((\delta,b), f^V_{+,
v})$ the same way by the same formula~\eqref{eq:global_orb_U}, integrating
over $\U(V)(F_v)$ instead. Then if $f_+^V = \otimes f^V_{+, v}$ where
$f^V_{+, v} \in \cS(\U_{V, +}(F_v))$, then
    \[
    J^V_{\alpha}(f^V_+) = \prod_v O((\delta, b), f^V_{+, v}).
    \]
Let $\tS$ be a finite set of places of $F$ and $f_{+, \tS} \in \cS(\U_{V,
+}(F_\tS))$, then we can define the regular semisimple orbital integrals
$O((\delta, b), f^V_{+, \tS})$ in the same way.

We now turn to the infinitesimal invariant. Let $(X, w, v) \in \fx(F)$ be
regular semisimple which maps to $\alpha \in \cB(F)$. Let $f^+ \in
\cS(\fx(\bA))$. Then the geometric distribution
$\mathfrak{i}_{\alpha}(f^+)$ simplifies to
    \begin{equation}    \label{eq:global_orb_gl}
    \mathfrak{i}_{\alpha}(f^+) = \int_{G'_n(\bA)}
    f^+(g^{-1}X g, w g, g^{-1} )
    \eta(\det g) \rd g.
    \end{equation}
The integral on the right hand side depends only on $\alpha$ but not the
specific element $(X, w,v)$ that maps to it. This is the usual regular
semisimple orbital integral of $f^+$ that appeared in~\cites{CZ,Zhang1}. This
orbital integral factorize. If $f_{v}^+ \in \cS(\fx(F_v))$ and $(X, w, v)$
be a regular semisimple element in $\fx(F_v)$, we define the local orbital
integral $O((X, w, v), f_{v}^+)$ the same way by the same
formula~\eqref{eq:global_orb_gl}, integrating over $G'_n(F_v)$ instead. Then
if $f^+ = \otimes f_{v}^+$ where $f_{v}^+ \in \cS(\fx(F_v))$, then
    \[
    \mathfrak{i}_{\alpha}(f^+) = \prod_v O((X, w, v), f_{v}^+).
    \]
Let $\tS$ be a finite set of places of $F$ and $f_{\tS}^+ \in
\cS(\fx(F_\tS))$, then we can define the regular semisimple orbital
integrals $O((X, w, v), f_{\tS}^+)$ in the same way.

For each $\alpha \in \cB(F)$ we have a distribution $\mathfrak{j}_{\alpha}$
on $\cS(\fy^V(\bA))$. Again we will consider this distribution for all $V$ at
the same time, so we add a superscript $V$ and write
$\mathfrak{j}_{\alpha}^V$ to emphasize the dependance on $V$. Let $(\delta,
b) \in \fy^V(F)$ be regular semisimple which maps to $\alpha \in \cB(F)$. Let
$f^{V,+} \in \cS(\fy^V(\bA))$. Then the geometric distribution
$\mathfrak{j}^V_{\alpha}(f^{+,V})$ simplifies to
    \[
    \mathfrak{j}^V_{\alpha}(f^{+, V}) = \int_{\U(V)(\bA)}
    f^{+, V} (g^{-1} \delta g, g^{-1} v) \rd g.
    \]
The integral on the right hand side depends only on $\alpha$ but not the
specific element $(\delta,b)$ that maps to it. This is the usual regular
semisimple orbital integral of $f^{+, V}$ that appeared in~\cites{CZ,Zhang1}.

The global orbital integral factorizes. If $f^{+,V}_{v} \in \cS(\fy^V(F_v))$
and $(\delta, b)$ be a regular semisimple element in $\fy^V(F_v)$, we define
the local orbital integral $O((\delta,b), f^{+, V}_{v})$ the same way by the
same formula~\eqref{eq:global_orb_U}, integrating over $\U(V)(F_v)$ instead.
Then if $f^{+, V} = \otimes f^{+, V}_{v}$ where $f^{+, V}_{v} \in
\cS(\fy^V(F_v))$, then
    \[
    \mathfrak{j}^V_{\alpha}(f^{+,V}) = \prod_v O((\delta, b), f^{+, V}_{v}).
    \]
Let $\tS$ be a finite set of places of $F$ and $f^{+, V}_{\tS} \in
\cS(\fy^V(F_{\tS}))$, then we can define the regular semisimple orbital
integrals $O((\delta, b), f^{+, V}_{\tS})$ in the same way.

\subsection{Local transfer}

We now compare the local regular semisimple orbital integrals. Fix a place $v$ of $F$. For regular semisimple $x = (X, w, v) \in X(F_v)$ or
$\fx(F_v)$ we put
    \[
    \Delta_+(x) =
    \det(w, wX, w X^2, \hdots, w X^{n-1}).
    \]
Define a transfer factor $\Omega_v$ for $X_{rs}(F_v)$ or $\fx_{rs}(F_v)$ by
$\Omega_v(x) = \mu((-1)^n\Delta_+(x))$. The sign $(-1)^n$ is included to make
it compatible with the transfer factor in~\cite{BP1}. The function $x \mapsto
\Omega_v(x) O(x, \varphi)$ where $x \in \fx_{rs}(F_v)$ and $\varphi \in
\cS(\fx(F_v))$ descends to a function on $\cB_{rs}(F_v)$, which we denote
by the same notation.

We define a transfer factor $\Omega_v$ on $G^+(F_v)$ as follows. Let $((g_1,
g_2), w, v) \in G^+(F_v)$ be a regular semisimple element. Put $\gamma = \nu(g_1^{-1} g_2) \in S_n(F_v)$. We put
    \[
    \Omega_v((g_1, g_2), w, v) = \mu \left( (-1)^n (\det g_1^{-1} g_2)^{-n+1}
    \det \Delta_+(\gamma, w, v) \right).
    \]
For $f_+ \in \cS(G_+(F_v))$, the function $x \mapsto \Omega_v(x) O(x, f_+)$ where $x \in G_{rs}^+(F_v)$ descends to a function on
$\cA_{\mathrm{rs}}(F_v)$, which we also denote it by $x \mapsto \Omega_v(x)O(x,f_+)$.

A function $f_+ \in \cS(G_+(F_v))$ and a collection of functions
$\{f_+^V\}_{V \in \cH_v}$ where $f_+^V \in \cS(\U_{V,+}(F_v))$ match if for
all matching regular semisimple elements $x \in G^+(F_v)$ and $y\in
\U_V^+(F_v)$, we have
    \[
    \Omega_v(x) O(x, f_+) = O(y, f_+^V).
    \]
We say a function on $G_+(F_v)$ is transferable if there is a collection of
functions on $\U_{V, +}(F_v)$ that matches it. We say a collection of
functions on $\U_{V, +}(F_v)$ is transferable if there is a function on
$G_+(F_v)$ that matches it. We say a single function $f_+^V$ on $\U_{V
,+}(F_v)$ for a fixed $V$ is transferable if the collection of functions
$(f_+^V, 0, \hdots, 0)$ is transferable.

We say $\varphi' \in \cS(\fx(F_v))$ and a collection of functions
$\{\varphi^V\}_{V \in \cH_v}$ where $\varphi^V \in \cS(\fy^{V}(F_v))$ match
if for all matching regular semisimple elements $x \in \fx_{n}(F_v)$ and $y
\in \fy^{V}(F_v)$, we have
    \[
    \Omega_v(x) O(x, \varphi') = O(y, \varphi^V).
    \]

We say a function on $\fx(F_v)$ is transferable if there is a collection of
functions on $\fy^V(F_v)$ that matches it. We say a collection of functions
on $\fy^V(F_v)$ transferable if there is a function on $\fx(F_v)$ that
matches it. We say a single function $\varphi^V$ on $\fy^V(F)$ for a fixed
$V$ is transferable if the collection of functions $(\varphi^V, 0, \hdots,
0)$ is transferable.

\begin{theorem} \label{thm:matching}
If $F$ is non-Archimedean, all functions in
    \[
    \cS(G_+(F_v)), \ \cS(\fx(F_v)),
    \ \cS(G_+^V(F_v)),\
    \ \cS(\fy^V(F_v))
    \]
are transferable. If $F$ is Archimedean, transferable functions form dense
subspaces of these spaces.
\end{theorem}

\begin{proof}
If $F$ is non-Archimedean, then the cases of $\cS(\fx(F_v))$ and
$\cS(\fy^V(F_v))$ are proved in~\cite{Zhang1}. The cases of $\cS(G_+(F_v))$
and $\cS(\U_{V,+}(F_v))$ are explained in~\cite{Xue1}. It follows from the
fact that the integral transforms~\eqref{eq:simplified_test_function_GL}
and~\eqref{eq:simplified_test_function_U} are surjective.

If $F_v$ is Archimedean, the cases of $\cS(\fx(F_v))$ and $\cS(\fy^V(F_v))$
are proved in~\cite{Xue3}. The cases of $\cS(G_+(F_v))$ and
$\cS(\U_{V,+}(F_v))$ follow from the surjectivity
of~\eqref{eq:simplified_test_function_GL}
and~\eqref{eq:simplified_test_function_U} and the open mapping theorem,
cf.~\cite{Xue3}*{Section~2}. The surjectivity is a little more complicated
than its non-Archimedean counterpart. Let us explain the case of
$\cS(G_+(F_v))$. First the Fourier transform $-^\dag$ is a continuous
bijection. The map $f^+ \mapsto \varphi_{f^+}$ is surjective since $G^+(F_v)
\to X(F_v)$ is surjective and submersive, cf.~\cite{AG}*{Theorem~B.2.4}.
The open mapping theorem then ensures that the inverse image of a dense
subset in $\cS(X(F_v))$ is again dense in $\cS(G_+(F_v))$.
\end{proof}

Recall that we have defined maps $f_+ \mapsto f_{+, \natural}$ and $f_+^V
\mapsto f^V_{+, \natural}$ in Subsections~\ref{subsec:infinitesimal_GL}
and~\ref{subsec:infinitesimal_U}. Such maps depend on an element $\xi \in
E^1$. Fix such a $\xi$. The
following lemma follows from the definition and a short calculation of the
transfer factors.

\begin{lemma}   \label{lemma:Lie_algebra_matching}
If $f_+$ and the collection $\{f_+^V\}_{V \in \cH_v}$ match, then so do
$f_{+, \natural}$ and the collection $\{f^V_{+, \natural}\}_{V \in \cH_v}$.
\end{lemma}

There is a pairing on $\fx(F_v)$ given by
    \begin{equation}    \label{eq:pairing_x_n}
    \langle (X_1, w_1, v_1), (X_2, w_2, v_2) \rangle =
    \Trace X_1X_2 + w_1 v_2 + w_2v_1.
    \end{equation}
If $\varphi \in \cS(\fx(F_v))$ we define its Fourier transform
    \[
    \widehat{\varphi}(y) = \int_{\fx(F_v)} \varphi(x)
    \overline{\psi(\langle x, y \rangle)} \rd x.
    \]
Here $\rd x$ is the selfdual measure.
Let $V$ be a nondegenerate skew-Hermitian space. Define a bilinear form on $\fy^{V}(F_v)$
by
    \[
    \langle (X, v), (Y, w) \rangle =
    \Trace XY + 2 (-1)^{n-1} \Tr_{E/F} \tau^{-1} q_{V}(v, w),
    \]
and a Fourier transform
    \[
    \cF_{V}\varphi^V(y) = \int_{\fy^{V}(F_v)} \varphi^V(x)
    \overline{\psi(\langle x, y \rangle)} \rd x.
    \]
Here $\rd x$ is the selfdual measure.

\begin{remark}
This essentially means that we consider $\varphi^V$ as an element in
$\cS(\fy^{V_\tau}(F_v))$ an take the Fourier transform as
in~\cites{BP1,CZ,Zhang1}, with $\psi$ replaced by our $\overline{\psi}$.
\end{remark}

Let $v$ be a place of $F$. Set $\lambda(\psi_v) =\epsilon(\frac{1}{2},
\eta_v, \psi_v)$ where $\epsilon(s, \eta_v, \psi_v)$ is the local root
number. This is a fourth root of unity and satisfies the property that
    \[
    \lambda(\overline{\psi_v}) = \eta_v(-1) \lambda(\psi_v).
    \]
In the case $E_v/F_v = \C/\R$ and $\psi_v(x) = e^{2\pi\sqrt{-1} x}$ we have
$\lambda(\psi_v) = -\sqrt{-1}$.

\begin{theorem} \label{thm:Fourier_matching}
Suppose that $\varphi \in \cS(\fx(F_v))$ and a collection of functions
$\varphi^V \in \cS(\fy^V(F_v))$ match. Then so do $\widehat{\varphi}$ and
the collection $\lambda(\psi_v)^{\frac{n(n+1)}{2}} \eta_v((2\tau)^n \disc
V)^n \cF_{V}\varphi^V$.
\end{theorem}

This is proved in~\cite{Zhang1} if $v$ is non-Archimedean and~\cite{Xue3} if
$v$ is Archimedean. See~\cite{BP1}*{Theorem~5.32, Remark~5.33} for
corrections to~\cite{Zhang1}.

\begin{remark}
Note that $\disc V_{\tau} = (2\tau)^n \disc V$. We also note that we used
$\overline{\psi_v}$ in the Fourier transform. So $\lambda(\psi_v)$
in~\cite{BP1} is replaced by $\lambda(\overline{\psi_v})$ and this cancels
the factor $\eta_v(-1)^{\frac{n(n+1)}{2}}$ that appeared in~\cites{BP1}.
\end{remark}

Finally we turn to the matching of test functions in the unramified
situation, i.e. the fundamental lemma.

\begin{theorem} \label{thm:FL}
Assume the following conditions.
\begin{enumerate}
\item The place $v$ is odd and unramified. The element $\tau \in
    \fo_{E_v}^\times$. The characters $\psi_v$ and $\mu_v$ are unramified.

\item The skew-Hermitian space $V$ is split and contains selfdual
    $\fo_{E_v}$ lattice.
\end{enumerate}
Then the functions
    \[
    \vol H(\fo_{F_v})^{-1} \vol G'(\fo_{F_v})^{-1} \id_{G_+(\cO_{F_v})}
    \quad \text{and} \quad
    \vol \U(V)(\fo_{F_v})^{-2} \id_{\U_{V, +}(\cO_{F_v})}
    \]
match. The functions
    \[
    \vol H(\fo_{F_v})^{-1} \vol G'(\fo_{F_v})^{-1} \id_{\fx(\cO_{F_v})}
    \quad \text{and} \quad
    \vol \U(V)(\fo_{F_v})^{-2} \id_{\fy^V(\cO_{F_v})}
    \]
also match. The volumes are computed using the same measures as we used to define orbital integrals.
\end{theorem}

\begin{proof}
The infinitesimal version was first proved by~\cite{Yun} when the residue
field characteristic is large. The group version is deduced from the
infinitesimal version in~\cite{Liu}*{Theorem~5.15} under the same assumption.
The case of small residue characteristic in the infinitesimal case was later
proved by~\cites{BP3,Zhang3}. The only reason~\cite{Liu} needs the large
residue characteristic is because~\cites{BP3,Zhang3} were not available that
time. Once this is available, the result in~\cite{Liu} holds without the
assumption on the residue characteristic (apart from being odd). Indeed the
group version of the fundamental lemma in the theorem is called the
``semi-Lie'' case in~\cite{Zhang3}.
\end{proof}

\subsection{Global transfer}
\label{subsubsec:global_transfer}

We now compare global geometric terms $I_{\alpha}$ and $J_{\alpha}^V$. If we assume
furthermore that $\tS$ contains all Archimedean places and ramified places,
then we denote by $\cH^{\tS}$ the set of all isomorphism classes of nondegenerate
skew-Hermitian spaces $V$ over $E$ of dimension $n$, such that $V_v = V
\otimes E_v$ contains a self-dual $\cO_{E_v}$ lattice for all $v \not\in
\tS$.

We say a test
function $f_+ \in \cS(G_+(\bA))$ and a collection of test functions
$\{f_+^{V}\}_{V \in \cH}$ where $f_+^{V} \in \cS(\U^+_V(\bA))$ match, if
there exists a finite set of places $\tS$ of $F$ containing all Archimedean
place and ramified place in $E$, such that we have the following conditions.
\begin{itemize}
    \item $f^{V}_+ = 0$ for $V \not \in \cH^{\tS}$.

    \item For each $V \in \cH^{\tS}$, $f^{V}_+ = 
        f^{V}_{+,\tS} \otimes \id_{\U^+_V(\cO^\tS)}$ where $f^{V}_{+,\tS}
        \in \cS(\U_V^+(F_\tS))$.

    \item $f_+ = f_{+,\tS} \otimes
        \id_{G_+(\cO^\tS)}$, where $f_{+,\tS} \in \cS(G_+(F_\tS))$.

    \item $f_{+,\tS}$ and $\{f^{V}_{+,\tS}\}_{V \in \cH^{\tS}}$ match.

    \item The measures in defining the orbital integrals satisfy that $\vol H(\fo_{F_v})= \vol G'(\fo_{F_v})= \U(V)(\fo_{F_v}) = 1$ for all $v \not \in \tS$.
\end{itemize}

Similarly, we say $f \in \cS(\fx(\bA))$ and  the collection $\{f^{V}\}_{V
\in \cH}$ where $f^{V} \in \cS(\fu^+_V(\bA))$ match if there exists a finite
set of places $\tS$ of $F$, such that we have the following conditions.
\begin{itemize}
    \item $f^{V} = 0$ for $V \not \in \cH^{\tS}$.

    \item For each $V \in \cH^{\tS}$, $f^{V} =
        f^{V}_\tS \otimes \id_{\fu^+_V(\cO^\tS)}$ where $f^{V}_\tS \in
        \cS(\fu_V^+(F_\tS))$.

    \item $f = f_\tS \otimes
        \id_{\fx(\cO^\tS)}$, where $f_\tS \in \cS(\fx(F_\tS))$.

    \item $f_\tS$ and $\{f^{V}_\tS\}_{V \in \cH^{\tS}}$ match.

    \item The measures in defining the orbital integrals satisfy that $\vol G'(\fo_{F_v})= \U(V)(\fo_{F_v}) = 1$ for all $v \not \in \tS$.
\end{itemize}

Note that both in the group case and Lie algebra case, $f$ and $\{ f^V \}_{V
\in \cH}$ match for the set of places $\tS$ as above, then they also match for any set of places containing $\tS$.

The following theorem is proved in~\cite{CZ}*{Theorem~13.3.4.1}. This is what
is referred to as the singular transfer in~\cite{CZ}.

\begin{theorem} \label{thm:infintesimal_singular_transfer}
If $f \in \cS(\fx(\bA))$ and the collection $\{f^{V}\}_{V \in \cH}$ where
$f^V \in \cS(\fy^V(\bA))$ match, then for each $\alpha \in \cB(F)$, we have
    \[
    \mathfrak{i}_\alpha(f) =
    \sum_{V \in \cH} \mathfrak{j}^V_{\alpha}(f^{V}).
    \]
\end{theorem}

From this we deduce our group version of singular transfer.

\begin{theorem} \label{thm:singular_transfer}
If $f_+ \in \cS(G_+(\bA))$ and  $\{f^{V}\}_{V \in \cH}$ where $f_+^V \in
\cS(\U_{V,+}(\bA))$ are transfer, then for each $\alpha \in \cA(F)$, we have
    \[
    I_\alpha(f_+) = \sum_{V \in \cH} J^V_{\alpha}(f_+^{V}).
    \]
\end{theorem}

\begin{proof}
Fix an $\alpha \in \cA(F)$. Choose a large finite set of places $\tS$ as in
the definition of matching of global test function. Then
    \[
    f_+ = \id_{G_+(\cO^\tS_F)} \otimes f_{+, \tS}, \quad
    f^V_+ = \id_{\U_{V,+}(\cO_{F}^{\tS})} \otimes f_{+, \tS}^V.
    \]
We may enlarge $\tS$ such that the conditions of
Proposition~\ref{prop:infitesimal_decent_gl} hold.

Choose $\xi \in E$ with norm 1 such that $\alpha \in \cA^\xi(F)$. Then we have the functions $f_{+, \tS, \natural}$ and $f^V_{+, \tS, \natural}$ and they still match by Lemma~\ref{lemma:Lie_algebra_matching}. By
Proposition~\ref{prop:infitesimal_decent_gl}
and~\ref{prop:infitesimal_decent_u}, we have
    \[
    I_\alpha(f_+) = \mathfrak{i}_\alpha
    (\id_{\fx(\cO^\tS)}
    \otimes f_{+, \tS, \natural}),
    \quad
    J^V_\alpha(f^V_+) = \mathfrak{j}^V_\alpha
    (\id_{\fy^V(\cO^\tS)}
    \otimes f^V_{+, \tS, \natural}).
    \]
By definition, $f_{+, \tS, \natural}$ and $f^V_{+, \tS, \natural}$ match. The
theorem then follows from Theorem~\ref{thm:infintesimal_singular_transfer}.
\end{proof}

\part{Spectral characterizations of matching}

For the rest of this paper, we are in the local situation. So $F$ is a
local field of characteristic zero, and $E$ a quadratic etale algebra over
$F$. We will assume that $E$ is a field except in
Section~\ref{sec:split_place_comparison}.

\section{Preliminaries on the spectral comparison}

\subsection{General notation and conventions}
Let us introduce some general notation.

Let $X$ be an algebraic variety over $F$. We usually just write $X$ for the
$F$-points of $X$. All algebraic groups are $F$-groups. If $V$ is a vector
space over $E$, it is viewed as an affine variety over $F$ via restriction of
scalars.

Let $G$ be a reductive group over $F$. Let $\Xi^G$ be the Harish-Chandra Xi
function on $G$. It depends on the choice of a maximal compact subgroup $K$
of $G$. Different choices lead to equivalent functions. Since we use $\Xi^G$
only for the purpose of estimates, the choice of $K$ does not matter. Its
properties we will make use of are recorded
in~\cite{BP2}*{Proposition~1.5.1}. We also fix a logarithmic height function
$\varsigma^G$ on $G(F)$, cf.~\cite{BP2}*{Section~1.5}.

Let $X$ be a smooth algebraic variety over a local field $F$ of
characteristic zero. Let $\cS(X)$ be the Schwartz space. If $F$ is
non-Archimedean this is $C_c^\infty(X)$, the space of locally constant
compactly supported functions. If $F$ is Archimedean, this is the space of
Schwartz function on $X$ where $X$ is viewed as a Nash manifold,
cf.~\cite{AG}.

Let $\cC(G)$ and $\cC^w(G)$ be the spaces of Harish-Chandra Schwartz functions and
tempered functions on $G$ respectively, cf.~\cite{BP1}*{Section~2.4}. Assume
$G$ is quasi-split, and let $N$ be the unipotent radical of a Borel subgroup,
$\xi$ be a generic character of $N$. We also have the spaces $\cS(N \bs G,
\xi)$ and $\cC^w(N \bs G, \xi)$ as defined in~\cite{BP1}*{Section~2.4}. These
are nuclear Fr\'{e}chet spaces or LF spaces. An important property we will use is
that $\cS(G)$ is dense in $\cC^w(G)$.

\subsection{Measures}   \label{subsec:measure_local}
The goal of this part of the paper is to prove an exact equality between to spherical characters, so unlike the first part of this paper, now we need to specify the measures. We retain the notation from Part~1.

We fix self-dual measures on $F$ and $E^-$ with respect to $\psi$, and on $E$
with respect to $\psi_E$. We have
    \[
    \rd_{E^-} (\tau^{-1}y) = \abs{\tau}_E^{-\frac{1}{2}} \rd_F y, \quad
    \rd_E z = 2 \rd_F x \rd_{E^-} y, \text{ if $z = x+y$}.
    \]
The (normalized) absolute values on $E$ and $F$ are denoted by
$\abs{\cdot}_E$ and $\abs{\cdot}$ respectively. They satisfy the properties
that $\rd (ax) = \abs{a} \rd x$ and $\abs{x}_E = \abs{x}^2$ if $x \in F$.

We equip $G'_n = \GL_n(F)$ (resp. $G_n=\GL_n(E)$) with the Haar measure
\begin{equation*}
   \rd g= \frac{ \prod_{i,j} \rd_F g_{i,j}}{\valP{\det g}^n}, \quad \left( \text{resp. }  \rd g= \frac{ \prod_{i,j} \rd_E g_{i,j}}{\valP{\det g}_E^n} \right).
\end{equation*}
By taking products, we obtain Haar measures on $G_+$, $G^+$, $G$, $G'$ and $H$.

The tangent space of $S_n$ at $s = 1$ is $\fs_n$. There is a bilinear form on
$\fs_n$ given by
    \[
    \langle X, Y \rangle = \Tr XY.
    \]
Then, using the character $\psi$, we put a self-dual measure on $\fs_n$. This
is the same as the measure obtained by identifying $\fs_n$ with
$(E^-)^{n^2}$. Recall that we have defined the Cayley transform (where we
consider only the case $\xi = 1$ now and we omit the subscript $\xi$)
    \[
    \fc: \fs_n \dashrightarrow S_n, \quad
    X \mapsto \frac{1+X/2}{1-X/2}.
    \]
We equip $S_n$ with the unique $G_n'$-invariant measure such that $\fc$ is locally measure preserving near $X = 0$. If $F$ is
non-Archimedean, this means that $\fc(\cU)$ and $\cU$ have the same volume if $\cU$ is
a small enough neighbourhood of zero. If $F$ is Archimedean, this means that the
ratio of the volumes of $\fc(\cU)$ and $\cU$ tends to $1$ as we shrink the
neighbourhood $\cU$ to $0$.

Let $V$ be a Hermitian or skew-Hermitian space. There is a bilinear form on
the Lie algebra $\fu(V)$ of $\U(V)$, given by
    \[
    \langle X, Y \rangle = \Tr XY.
    \]
Then using the character $\psi$, we put a self-dual measure on $\fu(V)$. We
have a Cayley transform
    \[
    \fc: \fu(V) \dashrightarrow \U(V), \quad
    X \mapsto \frac{1+X/2}{1-X/2}.
    \]
We equip $\U(V)$ with the Haar measure $\rd h$ normalized so that the Cayley transform is
locally measure preserving near $X = 0$, which is interpreted in the same way
as in the case of $\fs_n$. This yields measures on $\U_V'$ and $\U_V$. By taking products, we get measures on $\U_{V, +}$ and $\U_V^+$.

Recall that in Section~\ref{sec:matching} we defined the notation of matching of regular semisimple orbits, which gives bijections~\eqref{eq:matching} (the set $\tS$ consists merely the single place we care about in this part). The actions appearing in~\eqref{eq:matching} are all free, as we
consider only the regular semisimple elements. Then the measures on
$G^+_{\rs}$ and etc. descend naturally to quotient measures on the
various quotients in~\eqref{eq:matching}. With our choice of the measures, the same argument
as~\cite{BP1}*{Lemma~5.11} gives the following lemma.

\begin{lemma}   \label{lemma:matching_measure}
The bijections~\eqref{eq:matching} are measure preserving.
\end{lemma}

\subsection{Representations}
Let $G$ be a reductive group over a local field $F$ of characteristic zero.
If $F$ is non-Archimedean, a representation of $G$ means a smooth
representation. If $F$ is Archimedean, a representation of $G$ means a smooth
representation of moderate growth.
For $\alpha \in \fa_{G,\C}^*$ and a representation $\pi$ of $G$, we put
    \[
    \pi_\alpha(g) = \pi(g) e^{\langle \alpha, H_G(g)\rangle}.
    \]

We denote by $\Pi_2(G)$ (resp. $\Temp(G)$) the set of isomorphism classes of
irreducible square integrable (resp. tempered) representations of $G(F)$.

Let $\cX_{\temp}(G)$ be the set of isomorphism classes of the form
$\Ind_{P}^{G} \sigma$, where $P = MN$ ranges over all parabolic subgroups of
$G$ and $\sigma \in \Pi_2(M)$. Since each $\pi \in \Temp(G)$ is a
subrepresentation of an element in $\cX_{\temp}(G)$. We therefore have a map
$\Temp(G) \to \cX_{\temp}(G)$. The set $\cX_{\temp}(G)$ has a structure of a
topological space with infinitely many connected components, and the connected
components are of the shape
    \[
    \cO =
    \{ \Ind_{P}^{G} \sigma_{\alpha} \mid \alpha \in \sqrt{-1} \fa_M^*\}.
    \]
Let $W(G, M) = N_{G}M/M$ be the Weyl group and
    \[
    W(G, \sigma) = \{ w \in W(G, M) \mid w \sigma \simeq \sigma\}.
    \]
Then the map
    \[
    \Ind_{P}^G: \sqrt{-1}\fa_M^* /W(G, \sigma) \to \cO, \quad
    \alpha \mapsto \Ind_{P}^{G} \sigma_{\alpha},
    \]
is surjective, and a local
homeomorphism at $0$.

Let $\cV$ be a Fr\'{e}chet space or more generally an LF space.
We shall frequently use the notion of Schwartz
functions on $\cX_{\temp}(G)$ valued in $\cV$. This space is denoted by
$\cS(\cX_{\temp}(G), \cV)$. If $\cV = \C$ we simply write
$\cS(\cX_{\temp}(G))$. This notion is defined and discussed in detail
in~\cite{BP1}*{Section~2.9}. In particular $\cS(\cX_{\temp}(G))$ is a Fr\'{e}chet
space if $F$ is Archimedean and an LF space if $F$ is non-Archimedean. The
most important property of a Schwartz function on $\cX_{\temp}(G)$ valued in
$\cV$ is that it is absolutely integrable with respect the measures on
$\cX_{\temp}(G)$ that we define below.

If $f \in \cS(G)$, we put for each $\pi \in \cX_{\temp}(G)$
    \[
    f_{\pi}(g) = \Trace(\pi(g^{-1}) \pi(f)), \quad g \in G.
    \]
Then $f_{\pi} \in \cC^w(G)$ and by~\cite{BP1}*{Proposition~2.13.1} the map
    \begin{equation}    \label{eq:Temp_is_Schwartz}
    \cX_{\temp}(G) \to \cC^w(G), \quad \pi \mapsto f_{\pi}
    \end{equation}
is Schwartz.

We now specialize to the case of unitary groups. Let $V$ be a $n$-dimensional hermitian space and $G = \U(V)$. If $\pi \in \Temp(G)$, we denote by $\BC(\pi)$ its local base change to $\GL_{n}(E)$, cf.~\cites{Mok,KMSW}. Let $V'$ be another $n$-dimensional hermitian space (maybe the same as $V$), and $\sigma \in \Temp(\U(V'))$, we denote by $\pi \sim \sigma$ if $\pi$ and $\sigma$ have the same local base change. We then have a
well-defined surjective maps
    \[
    \Temp(G) \to \Temp(G)/ \sim, \quad \cX_{\temp}(G) \to \Temp(G)/\sim.
    \]

\subsection{Spectral measures}

Let $\widehat{A_{M}}$ be the set of unitary characters of $A_M$. It admits a
normalized measure $\rd \chi$, cf.~\cite{BP1}*{(2.4.2)}. The map
    \[
    \sqrt{-1} \fa_M^* \to \widehat{A_M}, \quad \chi \otimes \alpha \mapsto
    \left(a \mapsto \chi(a) \abs{a}^{\alpha} \right)
    \]
is locally a homeomorphism. Let $\rd \alpha$ be the measure on $\sqrt{-1}
\fa_M^*$ such that this map is locally measure preserving.

We now assign measures on $\cX_{\temp}(G)$
following~\cite{BP1}*{Section~2.7}. Consider first $\Pi_2(M)$. We assign the
unique measure $\rd \sigma$ to $\Pi_2(M)$ such that the map
    \[
    \Pi_2(M) \to \widehat{A_M}, \quad \sigma \mapsto \omega_{\sigma}|_{A_M}
    \]
is locally measure preserving. Here $\omega_{\sigma}$ stands for the central
character of $\sigma$. Next we consider the map
    \[
    \Ind_P^G: \Pi_2(M) \to \cX_{\temp}(G), \quad \sigma \mapsto \Ind_P^G \sigma.
    \]
This map is quasi-finite and proper, and the image is a collection of some
connected components of $\cX_{\temp}(G)$. We thus equip the image with the
pushforward measure
    \[
    \frac{1}{\abs{W(G, M)}} \Ind_{M, *}^G \rd \sigma.
    \]
The measure on $\cX_{\temp}(G)$ whose restriction to the image of $\Ind_P^G$
equals this one is denoted by $\rd \pi$. Near a point $\pi_0 = \Ind_{P}^G
\sigma \in \cX_{\temp}(G)$, this measure can be described more explicitly as
follows. Let $\cV$ be a sufficiently small $W(G, \sigma)$-invariant
neighbourhood of $\sqrt{-1}\fa_M^*$ such that the map $\Ind_P^G$ induces a
topological isomorphism between $\cV/W(G, \sigma)$ and a small neighbourhood
$\cU$ of $\pi_0$. Then for every $\varphi \in C_c^\infty(\cU)$ we have
    \[
    \int_{\cU} \varphi(\pi) \rd \pi = \frac{1}{\abs{W(G, \sigma)}}
    \int_{\cV} \varphi(\pi_{\alpha}) \rd \alpha.
    \]

Fix a split skew-Hermitian space $V_{qs}$ and $G = \U(V_{qs})$ (or a
product of them). By~\cite{BP1}*{Lemma~2.101}, there is a unique topology on
$\Temp(G)/\sim$ such that the map $\cX_{\temp}(G) \to \Temp(G)/\sim$ is
locally an isomorphism, and for every connected component $\cO \subset
\cX_{\temp}(G)$, the map induces an isomorphism between $\cO$ and a connected
component of $\Temp(G)/\sim$. It follows that there is a unique measure on
$\Temp(G)/\sim$ such that the map $\cX_{\temp}(G) \to \Temp(G)/\sim$ is
locally measure preserving when $\cX_{\temp}(G)$ is given the measure $\rd
\pi$. We denote this measure on $\Temp(G)/\sim$ again by $\rd \pi$.

There is unique Borel measure $\rd \mu_G$ on $\cX_{\temp}(G)$ characterized by
    \begin{equation}    \label{eq:plancherel_general}
    f = \int_{\cX_{\temp}(G)} f_{\pi} \rd \mu_G(\pi),
    \end{equation}
for all $f \in \cS(G)$, cf.~\cites{HC,Wald}. Here the integration of the right hand side is taken
in $\cC^w(G)$, and is absolutely convergent in the sense
of~\cite{BP1}*{Section~2.3}. This measure is called the Plancherel measure on $\cX_{\temp}(G)$.

There is a function $\mu_G^*$ such that $\rd \mu_G(\pi) = \mu^*_G(\pi) \rd
\pi$ for all $\pi \in \cX_{\temp}(G)$. If $G = \U(V)$ where $V$ is an $n$-dimensional Hermitian or skew-Hermitian space, then by~\cite{BP1}*{Theorem~5.53}, we have
    \begin{equation}    \label{eq:formal_degree_conjecture}
    \mu^*_G(\pi) = \frac{\abs{\gamma^*(0, \pi, \Ad, \psi)}}{\abs{S_{\pi}}}.
    \end{equation}
Here $S_{\pi}$ is a certain finite abelian group (the group of centralizers
of the $L$-parameter of $\pi$), cf.~\cite{BP1}*{Section~2.11}, and
$\gamma^*(0, \pi, \Ad, \psi)$ is the normalized value of the adjoint gamma
factor given as follows. If $m$ is the order of poles of the adjoint gamma factor $\gamma(s, \pi, \Ad, \psi)$ at $s = 0$, then 
    \[
    \gamma^*(0, \pi, \Ad, \psi) = \lim_{s \to 0} \zeta_F(s)^n \gamma(s, \pi, \Ad, \psi).
    \]
Here we write $\zeta_F(s)$ for the local zeta function of $F$, i.e.
    \[
    \zeta_F(s) = \left\{
    \begin{aligned} &(1-q_F^{-s})^{-1}, &&\text{$F$ is non-Archimedean};\\
    & \pi^{-s/2} \Gamma(\frac{s}{2}), && F = \R.
    \end{aligned}
    \right.
    \]

\subsection{An estimate of Whittaker functions}
Let $G$ be a quasi-split reductive group over $F$. Fix a Borel subgroup $B=TN$. Let $\xi$ be a generic character of $N$. Let $A_0$ be the maximal split subtorus of $T$, let $\Delta$ be the set of simple roots of $A_0$ action on $N$. For $f \in \cS(G)$ and $g_1,g_2 \in G$, we put
\begin{equation*}
    W_f(g_1,g_2) = \int_N f(g_1^{-1}ug_2) \xi(u) \rd u.
\end{equation*}
For $a \in T$ and $\alpha \in \Delta$, we write $a^{\alpha}$ for $e^{\langle \alpha,H_0(a) \rangle}$.

\begin{lemma} \label{lem:W_f_Estimate}
    For any $N>0$ and $d>0$, there exists a continuous semi-norm $\| \cdot \|$ on $\cS(G)$ such that
    \begin{equation*}
        \lvert W_f(a_1k_1,a_2k_2) \rvert \ll \delta_B(a_1a_2)^{\frac{1}{2}} \prod_{\alpha \in \Delta} (1+a^{\alpha})^{-N} \varsigma(a_1)^d \varsigma(a_2)^{-d}  \lVert f \rVert
    \end{equation*}
    holds for any $a_1,a_2 \in T$ and $k_1,k_2 \in K$.
\end{lemma}

\begin{proof}
    We first show that for any $N,d>0$, there exists a continuous semi-norm $\| \cdot \|$ on $\cS(G)$ such that
    \begin{equation} \label{eq:Whittaker_estimate_1}
        \lvert W_f(a_1k_1,a_2k_2) \rvert \ll \delta_B(a_1a_2)^{\frac{1}{2}} \varsigma(a_1)^d \varsigma(a_2)^{-d}  \lVert f \rVert
    \end{equation}
    holds for any $a_1,a_2 \in T$ and $k_1,k_2 \in K$. Indeed, we can assume $k_1=k_2=1$, then
    \begin{equation*}
        W_f(a_1,a_2) = \int_N f(a_1^{-1}ua_2) \xi(u) \rd u = \delta_B(a_1) \int_N f(ua_1^{-1}a_2) \ll \delta_B(a_1) \varsigma(a_1^{-1}a_2)^{-d} \| f\|,
    \end{equation*}
    for some continuous semi-norm $\|\cdot\|$ on $\cS(G)$. Similarly $W_f(a_1,a_2) \ll \delta_B(a_2)\varsigma(a_1^{-1}a_2)^{-d} \|f\|$. Therefore \eqref{eq:Whittaker_estimate_1} holds.

    By the proof of ~\cite{BP4}*{Lemma 2.5.1}, for any $W \in \cC^w_d(N \backslash G,\xi)$ such that 
        \begin{equation*}
            W(ak) \le C \delta_B(a)^{\frac 12} \varsigma(a)^d
        \end{equation*}
        for any $(a,k) \in T \times K$, then for any $N>0$
        \begin{equation*}
            W(ak) \le C_1 C \prod_{\alpha \in \Delta} (1+a^{\alpha})^{-N} \delta_B(a)^{\frac 12} \varsigma(a)^d,
        \end{equation*} 
        for any $(a,k) \in T \times K$. Moreover,
    \begin{enumerate}
        \item If $F$ is non-Archimedean, the constant $C_1$ only depends on $N$ and the open compact $J \subset G$ such that $W$ is right $J$-invariant.
        \item If $F$ is Archimedean, 
        \begin{equation*}
            C_1 = 1 + \sup_{u \in \cK_N} \sup_{(a,k) \in T \times K} \lvert R(u)W(ak) \rvert \prod_{\alpha \in \Delta} \delta(a)^{-\frac 12} \varsigma(a)^{-d},
        \end{equation*}
        where $\cK_N$ is the compact subset of $\cU(\fg)$ depending on $N$.
    \end{enumerate}
    Apply this to $W(\cdot,a_2)$ as $a_2$ varies, we see that the constant $C_1$ is uniform, the lemma follows from \eqref{eq:Whittaker_estimate_1}.
\end{proof}

\section{Local spherical characters}

\subsection{Spherical character on the unitary groups}
\label{subsec:tempered_intertwining}

Let $V$ be a nondegenerate $n$-dimensional skew-Hermitian space. Recall that we fixed the nontrivial character $\psi$ and the character $\mu$ of $E^\times$. Then we have the Weil representations $\omega$
and $\omega^\vee$ of $\U(V)$ with respect to the characters $\psi, \mu$
and $\psi^{-1}, \mu^{-1}$ respectively. They are realized on
$\cS(L^\vee)$  where $\Res V =
L+L^\vee$ is a polarization of the symplectic space $\Res V$.

Let $\pi$ be a finite length tempered representation of $\U_V$. We put
    \[
    m(\pi) = \dim \Hom_{\U_V'}(\pi \otimeshat \omega^\vee, \C).
    \]
If $\pi$ is irreducible we have $m(\pi) \leq 1$ by~\cites{Sun,SZ}.

As every $\pi \in \Temp(\U_V)$ embeds in a unique representation in
$\cX_{\temp}(\U_V)$, there is a map $\Temp(\U_V) \to \cX_{\temp}(\U_V)$. Let
$\Temp_{\U_V'}(\U_V) \subset \Temp(\U_V)$ be the subspace of $\pi$ with
$m(\pi) \not=0$. By~\cite{Xue6}*{Proposition~3.4}, the natural map
    \[
    \Temp_{\U_V'}(\U_V) \to \cX_{\temp}(\U_V)
    \]
is injective, and the image is a collection of connected components of
$\cX_{\temp}(\U_V)$. We thus give $\Temp_{\U_V'}(\U_V)$ the topology induced
from $\cX_{\temp}(\U_V)$.

For $f_+ \in \cS(\U_{V,+})$, we put
    \begin{equation}    \label{eq:local_ddag_map}
    f_+^\ddag((g_1, g_2), v) = \int_{L^\vee} \omega_{(1)}^\vee(g_1)
    f_+((g_1, g_2), x+l', x-l')
    \psi( - 2\Tr_{E/F} q_V(x, l) ) \rd x,
    \end{equation}
where $g_i \in \U(V)$ and we write $v = l + l'$ where $l \in L$ and $L^\vee$.
This is is local counterpart of~\eqref{eq:global_ddag_map}. The notation $\omega_{(1)}^\vee(g_1)    f_+((g_1, g_2), x+l', x-l')$ is interpreted as follows. First $f_+$ is a Schwartz function on $\U_V \times L^\vee \times L^\vee$. We evaluate it at $(g_1, g_2)$ and obtain Schwartz function on $L^\vee \times L^\vee$. Then $g_1$ acts on the first factor via the Weil representation, which gives another Schwartz function. We then evaluate it at $(x+l', x-l')$. The integral transform~\eqref{eq:local_ddag_map} define a
continuous isomorphism
    \[
    \cS(\U_{V,+}) \to \cS(\U_{V}^+).
    \]
The defining expression makes sense for the functions in $\cC^w(\U_V)
\otimeshat \cS(L^\vee \times L^\vee)$, as the integration takes place only in the
variables in $L$. We thus end up with a map
    \[
    \cC^w(\U_V) \otimeshat \cS(L^\vee \times L^\vee) \to C^\infty(\U_V^+),
    \]
which we still denote by $-^\ddag$. Note that the image of the map is no
longer $\cC^w(\U_V) \otimeshat \cS(V)$.

\begin{lemma}   \label{lemma:J_pi_convergence}
Let $f_+ \in \cC^w(\U_V)\otimeshat \cS(L^\vee \times L^\vee)$. The integral
    \[
    \int_{\U_V'} f_+^\ddag(h, 0) \rd h
    \]
is absolutely convergent and defines a continuous linear form on $\cC^w(\U_V)
\otimeshat \cS(L^\vee \times L^\vee)$.
\end{lemma}

\begin{proof}
This is essentially~\cite{Xue4}*{Lemma~3.3}. Strictly speaking,~\cite{Xue4}
deals with only the case $F$ being Archimedean, but the non-Archimedean case
goes through by the same argument.
\end{proof}

We denote the linear form in the lemma by $\cL$. If $f \in \cC^w(\U_V)$ and
$\phi_1, \phi_2 \in \cS(L^\vee)$ then $\cL$ takes a more familiar form
\begin{equation}
\label{eq:relative_char_unitary_familiar_form}
  \cL(f \otimes \phi_1 \otimes \phi_2) =
    \int_{\U_V'} f(h^{-1}) \langle \omega^\vee(h) \phi_1,
    \overline{\phi_2} \rangle_{L^2} \rd h,
\end{equation}
where we recall that $\langle-,-\rangle_{L^2}$ stands for the $L^2$-inner product.

Recall that if $f \in \cS(\U_V)$ and $\pi$ is a tempered representation of
$\U_V$ of finite length, then we put $f_{\pi}(g) = \Trace(\pi(g^{-1})\pi(f))$
and $f_{\pi} \in \cC^w(\U_V)$. We may extend this to a continuous map
    \[
    \cS(\U_{V, +}) =
    \cS(\U_V) \otimeshat \cS(L^\vee \times L^\vee) \to
    \cC^w(\U_V) \otimeshat \cS(L^\vee \times L^\vee), \quad
     f_+ \mapsto f_{+, \pi}.
    \]

Let $f_+ \in \cS(\U_{V, +})$. We defined a linear form
    \[
    J_{\pi}(f_+) = \cL(f_{+, \pi}).
    \]
The above lemma ensures that this makes sense and $J_{\pi}$ defines a
continuous linear form on $\cS(\U_{V, +})$. If $f_+ = f \otimes \phi_1
\otimes \phi_2$ where $f \in \cS(\U_V)$ and $\phi_1, \phi_2 \in
\cS(L^\vee)$, then the linear form take a more familiar form
    \[
    J_{\pi}(f \otimes \phi_1 \otimes \phi_2) =
    \int_{\U_V'} \Trace (\pi(h) \pi(f))
    \langle \omega^\vee(h) \phi_1,
    \overline{\phi_2} \rangle_{L^2} \rd h.
    \]

\begin{lemma}   \label{lemma:J_pi_continuity}
For a fixed, $f_+ \in \cS(\U_{V,+})$ the function on $\cX_{\temp}(\U_{V, +})$
given by $\pi \mapsto J_{\pi}(f_+)$ is Schwartz. Moreover the map
    \[
    \cS(\U_{V,+}) \to \cS(\cX_{\temp}(\U_{V, +})), \quad
    f_+ \mapsto (\pi \mapsto J_{\pi}(f_+))
    \]
is continuous.
\end{lemma}

\begin{proof}
For a fixed $f_+ \in \cS(\U_{V, +})$, the map
    \[
    \cX_{\temp}(\U_V)  \to \cC^w(\U_V) \otimeshat \cS(L^\vee \times L^\vee), \quad
    \pi \mapsto f_{+, \pi}
    \]
is a Schwartz function valued in $\cC^w(\U_V) \otimeshat \cS(L^\vee \times L^\vee)$ by
\eqref{eq:Temp_is_Schwartz}, and this map depends continuously on $f_+$
by~\cite{BP1}*{(2.6.1)}. The lemma then follows from
Lemma~\ref{lemma:J_pi_convergence}.
\end{proof}

One main result of~\cite{Xue6} is the following.

\begin{prop}    \label{prop:explicit_intertwining}
If $\pi \in \Temp(\U_V)$, then $m(\pi) \not=0$ if and only if $J_{\pi}$ is
not identically zero.
\end{prop}

We now consider the spectral expansion of the unipotent orbital integrals. Let $f_+ \in \cS(\U_{V, +}(F_v))$. Put
    \[
    O(1, f_+) =  f_{+, \natural}(0),
    \]
where we recall that the right hand side is defined right before Corollalry~\ref{coro:infinite_descent_U}. In more concrete terms, if $f_+ = f \otimes \phi_1 \otimes \phi_2$, then
    \[
    O(1, f \otimes \phi_1 \otimes \phi_2) = \int_{\U_V'} f(g^{-1})
    \langle \omega^\vee(g) \phi_1, \overline{\phi_2} \rangle_{L^2} \rd g.
    \]

We have by the Fourier inversion formula
    \begin{equation}    \label{eq:spectral_nilpotent_u}
    O(1, f_+) = \int_{\fy^{V}} \cF_{V}
    f_{+, \natural} (y) \rd y.
    \end{equation}
Clearly the map $f_+ \mapsto O(1, f_+)$ is a continuous linear form on
$\cS(\U_{V, +}(F_v))$.

\begin{lemma}   \label{lemma:plancherel_U}
For all $f_+ \in \cS(\U_{V, +})$ we have
    \begin{equation}    \label{eq:plancherel_U}
    O(1, f_+) = \int_{\cX_{\temp}(\U_V)}
    J_{\pi}(f_+) \rd \mu_{\U_V}(\pi).
    \end{equation}
The right hand side is absolutely convergent and defines a continuous linear
form on $\cS(\U_{V, +})$.
\end{lemma}

\begin{proof}
The absolute convergence and the continuity of the right hand side follow
from Lemma~\ref{lemma:J_pi_continuity}. Clearly the left hand side also
defines a continuous linear form. We may prove the lemma under the additional
assumption $f_+ = f \otimes \phi_1 \otimes \phi_2$ where $f \in \cS(\U_V)$
and $\phi_1, \phi_2 \in \cS(L^\vee)$.

By the Plancherel formula for the group $\U_V$, we have
    \[
    f = \int_{\cX_{\temp}(\U_V)} f_{\pi} \rd \mu_{\U_V}(\pi),
    \]
where the integral is absolutely convergent in $\cC^w(\U_V)$,
cf.~\eqref{eq:plancherel_general}. By Lemma~\ref{lemma:J_pi_convergence} we
have
    \[
    O(1, f \otimes \phi_1 \otimes \phi_2) =
    \int_{\cX_{\temp}(\U_V)} \int_{\U_V'} f_{\pi}(h^{-1})
    \langle \omega(h) \phi_1, \overline{\phi_2} \rangle_{L^2}
    \rd h \rd \mu_{\U_V}(\pi) .
    \]
This proves the lemma.
\end{proof}

\subsection{Spherical character on the linear groups}
\label{subsec:I_pi}
We now consider the spherical characters on the general linear groups.
We will use the following notation.

\begin{itemize}
    \item Recall that $G_n = \Res_{E/F} \GL_{n, E}$, $G_n' = \GL_{n, F}$, $G = G_n \times G_n$,
        $G' = G_n' \times G_n'$, and $H = G_n$ which diagonally embeds in $G$.

    \item Let $B_n$ the minimal parabolic subgroup of $G_n$ consisting of
        upper triangular matrices. Let $T_n$ be the diagonal torus, and
        $N_n$ the unipotent radical.

    \item Let $e_n = (0, \hdots, 0, 1) \in E_n$ and let $P_n$ be the
        mirabolic subgroup of $G_n$ which consists of matrices whose last
        row is $e_n$.

    \item Define subgroups of $G$ by $B = B_n \times B_n$, $T = T_n \times
        T_n$, $N = N_n \times N_n$, and $P = P_n \times P_n$.

    \item We fix a maximal compact subgroup $K$ of $G$.

    \item If $X$ is a subgroup of $G$, we put $X_H = X \cap H$.
\end{itemize}

Let $\xi$ is the character of $N_n$ given by
    \[
    \xi(u) = \psi^E((-1)^n (u_{1,2} + \hdots + u_{n-1, n})), \quad
    u \in N_n.
    \]
Put $\psi_{N} = \xi \boxtimes
\overline{\xi}$. As a note for the notation, $\xi$ is used in Part I of this paper to denote an element in $E^1$. But that element in $E^1$ is always taken to be $1$ in Part II of this paper, and $\xi$ will stand for the character we define here.

For $W \in \cC^w(N \bs G, \psi_{N})$ and $\Phi \in \cS(E_n)$, we define a
linear form
    \begin{equation}    \label{eq:local_RS}
    \lambda(W \otimes \Phi) = \int_{N_H \bs H}
    W(h) \Phi(e_n h) \mu(\det h)^{-1}
    \abs{\det h}_E^{\frac{1}{2}} \rd h.
    \end{equation}

\begin{lemma}
\label{lemma:lambda_continuous} The integral is absolutely convergent and
extends to a continuous linear form on $\cC^w(N \bs G, \psi_{N}) \otimeshat
\cS(E_n)$.
\end{lemma}

\begin{proof}
By the Iwasawa decomposition, the integral is bounded by
    \[
    \int_{T_H} \int_{K_H}
    \abs{W(ak) \Phi(e_n ak)}
    \abs{a_1\cdots a_n}_E^{\frac{1}{2}} \delta_n^{-1}(a) \rd k \rd a ,
    \]
where $a = (a_1, \hdots, a_n) \in T_H$, $a_i \in E^\times$ and $\delta_n$ the
modulus character of the $B_H$. By~\cite{BP1}*{Lemma~2.4.3}, there is a
$d>0$, such that for all $N>0$ we can find a continuous seminorm $\nu_{d, N}$
on $\cC^w(N \bs G, \psi_{N})$, with
    \[
    \abs{W(ak)} \leq \prod_{i=1}^{n-1} (1+\abs{a_ia_{i+1}^{-1}})^{-N}
    \delta_n(a)^{\frac{1}{2}} \varsigma(a)^d \nu_{d, N}(W).
    \]
Here $\varsigma(a)$ is a fixed logarithmic height function on $T_H$. Since
$\Phi \in \cS(E_n)$, for any $N>0$ there is a continuous seminorm $\nu'_N$ on
$\cS(E_n)$ such that we also have
    \[
    {\Phi(e_n a k)} \leq (1+\abs{a_n})^{-N} \nu'_N(\Phi).
    \]
The lemma then reduces to~\cite{BP1}*{Lemma~2.4.4}.
\end{proof}

For $W \in \cC^w(N \bs G, \psi_N)$, define the linear forms
    \begin{equation}    \label{eq:local_FR}
    \beta_{\eta}(W) = \int_{N' \bs P'} W(h) \eta_{G'}(h) \rd h,
    \end{equation}
By~\cite{BP1}*{Lemma~2.15.1} this integral is absolutely convergent and
define a continuous linear form on $\cC^w(N \bs G, \psi_{N})$.

Let $\Pi$ be an irreducible tempered representation of $G$. Let $\cW(\Pi,
\psi_{N})$ be its Whittaker model. Then $\cW(\Pi, \psi_{N}) \subset \cC^w(N
\bs G, \psi_{N})$ by~\cite{BP1}*{Lemma~2.8.1}. We fix an inner product on
$\cW(\Pi, \psi_{N})$ by
    \begin{equation}    \label{eq:whittaker_inner_product_local}
    \langle W_1, W_2 \rangle^{\mathrm{Wh}}_{G} =
    \int_{N \bs P} W_1(p) \overline{W_2(p)} \rd p.
    \end{equation}
Here $\rd p$ stands for the right invariant measure on $P$.

In what follows we will use the following construction. For $f \in \cS(G)$,
we put
    \begin{equation}    \label{eq:definition_W_phi}
    W_{f}(g_1, g_2) = \int_{N}
    f(g_1^{-1} u g_2) \psi_{N}(u) \rd u, \quad g_1, g_2 \in G.
    \end{equation}
Then by~\cite{BP1}*{Section~2.14}, $W_{f}(g_1, \cdot) \in \cS(N \bs G,
\overline{\psi_{N}})$ for any $g_1 \in G$. As explained
in~\cite{BP1}*{Section~2.14}, this construction actually extends to all $f
\in \cC^w(G)$ by replacing $\int_{N}$ by a regularized version $\int^*_{N}$.
By~\cite{BP1}*{Lemma~2.14.1} the resulting function $W_{f}$ lies in
$\cC^w(N\bs G \times N \bs G, \psi_{N} \boxtimes \overline{\psi_{N}})$ and
the map
    \[
    \cC^w(G) \to \cC^w(N \bs G \times N \bs G,
    \psi_{N} \boxtimes \overline{\psi_{N}}), \quad f \mapsto W_{f}
    \]
is continuous.

For $f \in \cS(G)$ and $\Pi \in \cX_{\temp}(G)$ (note that $\Temp(G) =
\cX_{\temp}(G)$ as topological spaces, cf.~\cite{BP1}*{Subsection~4.1,
p.~245-246}), we put
    \begin{equation}    \label{eq:W_f_Pi}
    W_{f, \Pi} = W_{f_{\Pi}} \in
    \cC^w(N \bs G \times N \bs G,
    \psi_{N} \boxtimes \overline{\psi_{N}}).
    \end{equation}
Then the map $f \mapsto W_{f, \Pi}$ is continuous. Moreover
by~\cite{BP1}*{Proposition~2.14.2} the map
    \[
    \cX_{\temp}(G) \to \cC^w(N
    \bs G \times N \bs G, \psi_{N} \boxtimes \overline{\psi_{N}}),
    \quad
    \Pi \mapsto W_{f, \Pi}
    \]
is Schwartz.

We have the Plancherel formula for Whittaker functions,
cf.~\cite{BP1}*{Proposition~2.14.2}. For any $f \in \cS(G)$, we have
    \begin{equation}    \label{eq:Plancherel_Whittaker}
    W_{f} = \int_{\cX_{\temp}(G)} W_{f, \Pi} \rd \mu_{G}(\Pi).
    \end{equation}
where the right hand side is absolutely convergent in $\cC^w(N \bs G \times N
\bs G, \psi_{N} \boxtimes \overline{\psi_{N}})$.

Finally by~\cite{BP1}*{(2.14.3)}
    \begin{equation}    \label{eq:whittaker_expansion}
    W_{f, \Pi}=\abs{\tau}_E^{-n(n-1)}
     \sum_{W}
    \Pi(f)W \otimes \overline{W},
    \end{equation}
where the sum runs over an orthonormal basis of $\cW(\Pi, \psi_{N})$ and the
right hand side is absolutely convergent in $\cC^w(N \bs G \times N \bs G,
\psi_{N} \boxtimes \overline{\psi_{N}})$.

Let $f \in \cS(G)$ and $\Phi \in \cS(E_n)$. Define
\begin{equation}
\label{eq:I_pi_def}
        I_{\Pi}(f \otimes \Phi) =
    \sum_W \lambda(\Pi(f)W, \Phi) \overline{\beta_{\eta}(W)},
\end{equation}
where the sum runs over an orthonormal basis of $\cW(\Pi, \psi_{N})$.

\begin{lemma}   \label{lemma:Schwartz_I_Pi}
For fixed $f$ and $\Phi$, the function on $\cX_{\temp}(G)$ given by
    \[
    \Pi \mapsto I_{\Pi}(f \otimes \Phi)
    \]
is Schwartz. Moreover the map
    \[
    \cS(G) \otimes \cS(E_n) \to \cS(\cX_{\temp}(G)), \quad
    f \otimes \Phi \mapsto (\Pi \mapsto I_\Pi(f \otimes \Phi))
    \]
extends to a continuous map
    \[
    \cS(G_+) \to \cS(\cX_{\temp}(G)).
    \]
\end{lemma}

\begin{proof}
The map
    \[
    \cS(G) \to \cC^w(N \bs G \times N \bs G,
    \psi_{N} \boxtimes \overline{\psi_{N}}), \quad
    f \mapsto W_{f, \Pi}
    \]
extends to a map
    \[
    \cS(G_+) = \cS(G) \otimeshat \cS(E_n) \to
    \cC^w(N \bs G \times N \bs G,
    \psi_{N} \boxtimes \overline{\psi_{N}}) \otimeshat \cS(E_n)
    \]
in an obvious way, which we again denote by $f_+ \mapsto W_{f_+, \Pi}$ where
$f_+ \in \cS(G_+)$.

With this notation, the identity~\eqref{eq:whittaker_expansion} also extends
and holds when $f_+ \in \cS(G_+)$. Here the action of $f_+$ on $W \in
\cW(\Pi, \psi_{N})$ is given by
    \[
    \int_{G} f_+(g, x) \Pi(g) W \rd g
    \]
and is an element in $\cW(\Pi, \psi_{N}) \otimeshat \cS(E_n)$. Since the
sum~\eqref{eq:whittaker_expansion} is absolutely convergent, we conclude that
    \begin{equation}    \label{eq:I_linear_form}
    I_{\Pi}(f_+) = \abs{\tau}_E^{n(n-1)}
    (\lambda \otimeshat \beta_{\eta})(W_{f_+, \Pi}).
    \end{equation}
Here $\lambda \otimeshat \beta_{\eta}$ is the linear form defined
in~\cite{BP1}*{Lemma-Definition~2.41}. The rest of the lemma follows from the
fact that $\Pi \mapsto W_{f_+, \Pi}$ is Schwartz and the map $f_+ \mapsto
W_{f_+, \Pi}$ is continuous.
\end{proof}

\subsection{Regular unipotent orbital integrals in general linear groups}

Set
    \[
    \xi_+ = (-1)^n \left( \begin{pmatrix} 0 & \tau^{-1}\\
    & 0 & \ddots \\ & & \ddots & \tau^{-1} \\ &&& 0 \end{pmatrix},
    (0, \hdots, 0),
    \begin{pmatrix} 0 \\ \vdots \\ 0 \\ \tau^{-1} \end{pmatrix} \right)
    \in \fx(F_v),
    \]
and
    \[
    \xi_- = (-1)^n \left( \begin{pmatrix} 0\\
    \tau & 0 \\ &\ddots & \ddots \\ && \tau & 0 \end{pmatrix},
    (0, \hdots, 0, \tau),
    \begin{pmatrix} 0 \\ \vdots \\ 0 \\ 0 \end{pmatrix} \right)
    \in \fx(F_v).
    \]

Define for $\varphi \in \cS(\fx)$
    \[
    O(s, \xi_+, \varphi) = \int_{G_n'} \varphi(\xi_+\cdot h) \eta(\det h)
    \abs{\det h}^s \rd h.
    \]
By~\cite{BP1}*{Proposition~5.7.1} this integral of $O(s, \xi_+, \varphi)$ is
absolutely convergent if $\re s> 1- \frac{1}{n}$. Moreover the function $s
\mapsto O(s, \xi_+, \varphi)$ has a meromorphic continuation to $\C$ which is
holomorphic at $s = 0$. Setting $O(\xi_+, \varphi) = O(s, \xi_+,
\varphi)|_{s= 0}$, we have
    \begin{equation}    \label{eq:reg_nilpotent}
    \gamma \mu((-1)^{n-1} \tau)^{\frac{n(n+1)}{2}}
    O(\xi_+, \varphi) =
    \int_{\fx} \Omega(x) \widehat{\varphi}(x) \rd x
    = \int_{\cB_{\rs}} \Omega(x) O(x, \widehat{\varphi}) \rd x,
    \end{equation}
where
    \begin{equation}    \label{eq:gamma_constant}
    \gamma = \prod_{i = 1}^n \gamma(1-i, \eta^i, \overline{\psi}).
    \end{equation}

\begin{remark}  \label{remark:compare_BP}
Strictly speaking, the results in~\cite{BP1} describe the nilpotent orbital
integrals for the conjucation action of $G'_n$ on $\fs_{n+1}$ where $G'_n$ is
viewed as a subgroup of $G'_{n+1}$ on the left right corner. This action is
essentially the same as ours, since the lower right corner element in
$\fs_{n+1}$ is invariant under the action of $G'_n$. Our choice of $\xi_\pm$
is made so that they corresponds exactly to the ones used in~\cite{BP1}.
\end{remark}

If $f_+ \in \cS(G_{+})$ we put
    \[
    O_+(f_+) = O(\xi_+, f_{+, \natural}),
    \]
where in the map $f_+ \mapsto f_{+, \natural}$, we pick an
arbitrary small neighbourhood $\cU$ of $1$. The maps $\varphi \mapsto O(\xi_+, \varphi)$ and $f_+
\mapsto f_{+, \natural}$ are continuous. Thus the map $f_+ \mapsto O_+(f_+)$
is also continuous.

\begin{lemma}   \label{lemma:reg_unipotent_Whittaker}
Let $f\in \cS(G)$ and $\Phi \in\cS(E_n)$. Then
    \begin{equation}    \label{eq:reg_unipotent_Whittaker}
    \gamma O_+(f \otimes \Phi)
     = \abs{\tau}_E^{\frac{n(n-1)}{4}}
    \int_{N_H \bs H}
    \int_{N' \bs G'} W_{f}(h, g) \Phi(e_n h) \mu(\det h)^{-1}
    \abs{\det h}_E^{\frac{1}{2}} \eta_{G'}(g) \rd g \rd h.
    \end{equation}
The right hand side is absolutely convergent.
\end{lemma}

\begin{proof}
The absolute convergence of the right hand side follows from Lemma~\ref{lem:W_f_Estimate}. Unwinding the
definitions, we see that the right hand side of
~\eqref{eq:reg_unipotent_Whittaker} equals
    \[
    \begin{aligned}
    &\int_{N_n' \bs G_n'} \int_{N_n' \bs G_n'}
    \int_{N_H \bs H}  \int_{N_n} \int_{N_n}\\
    &f(h^{-1} u_1 g_1, h^{-1} u_2 g_2)
    \xi(u_1u_2^{-1}) \Phi(e_n h)
    \abs{\det h}_E^{\frac{1}{2}} \mu(\det h)^{-1}
    \eta(\det g_1 g_2)^{n+1} \rd u_1  \rd u_2  \rd h \rd g_2 \rd g_1 .
    \end{aligned}
    \]
Making a change of variables $u_2 \mapsto u_1u_2$ and combine the integration
over $u_1$ and $h$ we obtain
    \[
    \begin{aligned}
    &\int_{N_n' \bs G_n'} \int_{N_n' \bs G'_n}
    \int_{H}  \int_{N_n} \\
    &f(h^{-1} g_1, h^{-1} u_2 g_2)
    \xi(u_2)^{-1} \Phi(e_n h)
    \abs{\det h}_E^{\frac{1}{2}} \mu(\det h)^{-1}
    \eta(\det g_1 g_2)^{n+1} \rd u_2 \rd h \rd g_2 \rd g_1.
    \end{aligned}
    \]
Make another change of variable $h \mapsto g_1 h$, split the integral of $u
\in N_n$ into $N_n'$ and $N_n/N_n'$ and absorb $N_n'$ into the integration
over $g_2 \in N_n' \bs G_n'$. We obtain
    \[
    \begin{aligned}
    &\int_{N_n' \bs G_n'} \int_{G_n'}
    \int_{G_n}  \int_{N_n/N_n'} \\
    &f(h^{-1}, h^{-1} g_1^{-1} u_2 g_2)
    \xi(u_2)^{-1} (\mathrm{R}_{\mu^{-1}}(h)\Phi)(e_n g_1)
    \abs{\det g_1}
    \eta(\det g_1)^{n} \eta(\det g_2)^{n+1} \rd u_2 \rd h \rd g_2 \rd g_1.
    \end{aligned}
    \]

Recall that for any $\Phi \in \cS(E_n)$ we have defined the partial Fourier
transform
    \[
    \Phi^\dag(x^-, y^-) = \int_{F_n} \Phi(x+x^-)
    \psi((-1)^{n} \tau x y^-) \rd x.
    \]
Then we have an inversion formula
    \begin{equation}    \label{eq:local_FT_inverstion_Phi}
    \Phi(x+x^-) = \abs{\tau}_E^{\frac{n}{2}}
    \int_{E^{-, n}} \Phi^\dag(x^-, y^-)
    \psi((-1)^{n+1} \tau xy^-) \rd y^-.
    \end{equation}
It follows that
    \[
    (\mathrm{R}_{\mu^{-1}}(h)\Phi)(e_n g_1) = \abs{\tau}_E^{\frac{n}{2}}
    \abs{\det g_1}^{-1}
    \int_{E^{-, n}}
    (\mathrm{R}_{\mu^{-1}}(h)\Phi)^\dag(0, g_1^{-1} y)
    \psi((-1)^{n+1} \tau e_n y) \rd y.
    \]
Plugging this back into the previous integral we obtain that the right hand
side of~\eqref{eq:reg_unipotent_Whittaker} equals
    \begin{equation}    \label{eq:convergence_needs_argument}
    \begin{aligned}
    &\abs{\tau}_E^{\frac{n}{2}}
    \int_{N_n' \bs G_n'} \int_{G_n'}
    \int_{G_n}  \int_{N_n/N_n'} \int_{E^{-, n}}
    f(h^{-1}, h^{-1} g_1^{-1} u_2 g_2) \\
    &
    \xi(u_2)^{-1} (\mathrm{R}_{\mu^{-1}}(h)\Phi)^\dag(0, g_1^{-1} y)
    \psi((-1)^{n+1} \tau e_n y)
    \eta(\det g_1)^{n} \eta(\det g_2)^{n+1}
    \rd y \rd u_2 \rd h \rd g_2 \rd g_1.
    \end{aligned}
    \end{equation}
The inner four integrals (the integrals apart from $g_1$) are absolutely
convergent. This can be seen as follows. Recall that $\mathrm{R}_{\mu^{-1}}^\dag$ is the
unique representation of $G_n$ on $\cS(E_n^- \times E^{-, n})$ such that
$\mathrm{R}_{\mu^{-1}}^\dag(h) \Phi^\dag = (\mathrm{R}_{\mu^{-1}}(h)\Phi)^\dag$ for all $\Phi
\in \cS(E_n^- \times E^{-, n})$. It is isomorphic to $\mathrm{R}_{\mu^{-1}}$ and thus
of moderate growth. Thus the function
    \[
    (g_1, g_2, x, y) \mapsto f(g_1^{-1}, g_1^{-1} g_2)
    (\mathrm{R}_{\mu^{-1}}(g_1)\Phi)^\dag(x, y)
    \]
where $g_1, g_2 \in G_n$, $x \in E_n^-$ and $y \in E^{-,n}$, is again
Schwartz function on $G_n \times G_n \times E_n^- \times E^{-,n}$. Moreover
the map
    \[
    G_n \times N_n/N_n' \times G_n' \to G_n \times G_n, \quad
    (g, u, h_2) \mapsto (g^{-1}, g^{-1}uh_2)
    \]
is a closed embedding. It follows that the inner four integrals
of~\eqref{eq:convergence_needs_argument} are integrating a Schwartz function
on $G_n(E) \times G_n(E) \times E_n^- \times E^{-,n}$ over a closed
submanifold, and are thus convergent. Therefore we can switch the order of
the inner four integrals to conclude
    \[
    \begin{aligned}
    \eqref{eq:convergence_needs_argument} = \abs{\tau}_E^{\frac{n}{2}}
    \int_{N_n' \bs G_n'} \int_{N_n/N_n'} \int_{E^{-, n}}
    &\varphi_{f \otimes \Phi}( g_1^{-1}
    u_2 \overline{u_2}^{-1} g_1, 0, g_1^{-1} y)\\
    &\xi(u_2)^{-1} \psi((-1)^{n+1} \tau e_n y)
    \rd y \rd u_2 \rd h_1.
    \end{aligned}
    \]
Here $\varphi_{f \otimes \Phi}$ is defined
by~\eqref{eq:simplified_test_function_GL} (or rather the local counterpart).
By definition we have
    \[
    \xi(u_2)^{-1} \psi((-1)^{n+1} \tau e_n y)
    = \overline{\psi(\langle \xi_-, (A,0, y) \rangle)},
    \]
where $A \in \fs_n$ and $\fc(A) = u_2 \overline{u_2}^{-1}$, and
the pairing $\langle-,-\rangle$ on the right hand side is the one on
$\fx$ defined in Subsection~\ref{subsec:infinitesimal_GL}
by~\eqref{eq:pairing_x_n}. By our choice of the measures the last integral
equals
    \[
    \abs{\tau}_E^{\frac{n}{2}}
    \int_{N_n' \bs G_n'} \int_{\fn_{X}}
    (f \otimes \Phi)_\natural( h_1^{-1} Y h_1)
    \overline{\psi(\langle \xi_-, Y \rangle)}
    \rd Y \rd h_1,
    \]
which equals
    \[
    \abs{\tau}_E^{\frac{n}{2}-\frac{n(n+1)}{4}} \gamma O_+(f \otimes \Phi),
    \]
by~\eqref{eq:regular_nilpotent_intermediate}. This proves the lemma.
\end{proof}

\subsection{Spectral expansions of regular unipotent integrals}

Denote by $V_{qs}$ a fixed split $n$-dimensional skew-Hermitian space.
Also recall that $\gamma$ is the constant defined in~\eqref{eq:gamma_constant}.

\begin{prop}    \label{prop:plancherel_gl}
Let $f_+ \in \cS(G_+)$. We have
    \begin{equation}    \label{eq:plancherel_gl}
    \gamma O_+(f_+) = \abs{\tau}_E^{-\frac{n(n-1)}{4}}
    \int_{\Temp(\U_{V_{qs}})/\sim} I_{\BC(\pi)}(f_+)
    \frac{\abs{\gamma^*(0,\pi, \Ad, \psi)}}{\abs{S_{\pi}}} \rd \pi.
    \end{equation}
The right hand side is absolutely convergent.
\end{prop}

We need some preparations before we prove this proposition. 
Put $N_S$ be the image of $N_n$ in $S_n$ and $\fn_S$ be its tangent space at
$1$. They are given the quotient measure. Then the Cayley transform preserves
the measures on $N_S$ and $\fn_S$. Put $\fn_{X} = \fn_S \times \{0\} \times
E^{-,n}$ which is given the obvious measure. Then we have $\fn_{X} \cong
N_n/N_n' \times \{0\} \times E^{-,n}$, which is measure preserving. Recall
that we have a right action of $G_n'$ on $\fx$ given
by~\eqref{eq:action_s_infinitesimal}. It is proved
in~\cite{BP1}*{Lemma~5.73} (cf. Remark~\ref{remark:compare_BP}
after~\eqref{eq:reg_nilpotent}) that for any $\varphi \in \cS(\fx)$ we
have
    \begin{equation}    \label{eq:regular_nilpotent_intermediate}
    \gamma O(\xi_+, \varphi) = \abs{\tau}_E^{\frac{n(n+1)}{4}}
    \int_{G_n'/N_n'} \int_{\fn_X} \varphi(x \cdot h)
    \overline{\psi(\langle \xi_-, x \rangle)}
    \eta(\det h) \rd x \rd h.
    \end{equation}
This is convergent as an iterated integral.

\begin{proof}[Proof of Proposition~\ref{prop:plancherel_gl}]
The absolute convergence of the right hand side of~\eqref{eq:plancherel_gl}
follows from Lemma~\ref{lemma:Schwartz_I_Pi}, and the fact the function
$\frac{\abs{\gamma^*(0, \pi, \Ad, \psi)}}{\abs{S_{\pi}}}$ is of moderate
growth, cf.~\cite{BP1}*{Lemma~2.45} and~\cite{BP1}*{(2.7.4)}. The continuity
also follows from Lemma~\ref{lemma:Schwartz_I_Pi}.

Since both sides are continuous in $f_+$, we may additionally assume that the
test function $f_+$ is of the form $f \otimes \Phi$, where $f \in \cS(G(F))$ and
$\Phi \in \cS(E_n)$.

By Lemma~\ref{lemma:reg_unipotent_Whittaker} we have
    \[
    \gamma O_+(f \otimes \Phi)
     = \abs{\tau}_E^{\frac{n(n-1)}{4}}
    \int_{N_H \bs H}
    \int_{N' \bs G'} W_{f}(h, g) \Phi(e_n h) \mu(\det h)^{-1}
    \abs{\det h}_E^{\frac{1}{2}} \eta_{G'}(g) \rd g \rd h.
    \]

For any fixed $h \in H$, by~\cite{BP1}*{Corollary~3.51, Theorem~4.22} the
inner integral equals
    \[
    \abs{\tau}_E^{\frac{n(n-1)}{2}}
    \int_{\Temp(\U_{V_{qs}})/\sim} \beta_{\eta}(W_{f, \BC(\pi)}(h, \cdot))
    \frac{\abs{\gamma^*(0, \pi, \Ad, \psi)}}{\abs{S_{\pi}}} \rd \pi.
    \]
Here the linear form $\beta_{\eta}$ applies to the variable $\cdot$ and the integral
over $\pi$ is absolutely convergent. Thus
    \[
    \begin{aligned}
    \gamma O_+(f \otimes \Phi) =
    \abs{\tau}_E^{\frac{n(n-1)}{4}+\frac{n(n-1)}{2}}
    &\int_{N_H \bs H}
    \int_{\Temp(\U_{V_{qs}})/\sim} \beta_{\eta}(W_{f, \BC(\pi)}(h, \cdot))\\
    &\Phi(e_n h) \mu(\det h)^{-1}
    \abs{\det h}_E^{\frac{1}{2}}
    \frac{\abs{\gamma^*(0, \pi, \Ad, \psi)}}{\abs{S_{\pi}}} \rd \pi \rd h.
    \end{aligned}
    \]
For a fixed $\Phi$, we denote the linear form $W \mapsto \lambda(W, \Phi)$ by
$\lambda_\Phi$. Then we have
    \[
    \begin{aligned}
    &\gamma O_+(f \otimes \Phi)\\ =
    &\abs{\tau}_E^{\frac{n(n-1)}{4}+\frac{n(n-1)}{2}}
    \int_{\Temp(\U_{V_{qs}})/\sim}
    (\lambda_{\Phi} \otimeshat \beta_{\eta})\left(W_{f, \BC(\pi)}\right)
    \frac{\abs{\gamma^*(0, \pi, \Ad, \psi)}}{\abs{S_{\pi}}} \rd \pi.
    \end{aligned}
    \]
Moreover by~\eqref{eq:I_linear_form}
    \[
    (\lambda_{\Phi} \otimeshat \beta_{\eta})\left(W_{f, \BC(\pi)}\right) =
    \abs{\tau}_E^{-n(n-1)} I_{\BC(\pi)}(f \otimes \Phi).
    \]
It follows that
    \[
    \gamma O_+(f \otimes \Phi) = \abs{\tau}_E^{-\frac{n(n-1)}{4}}
    \int_{\Temp(\U_{V_{qs}})/\sim} I_{\BC(\pi)}(f \otimes \Phi)
    \frac{\abs{\gamma^*(0,\pi, \Ad, \psi)}}{\abs{S_{\pi}}} \rd \pi.
    \]
This proves the proposition.
\end{proof}

\section{Local relative trace formulae}

\subsection{Local trace formula on the unitary groups}
Let $V$ be a $n$-dimensional skew-Hermitian space. We consider in this
subsection $f_{1, +}, f_{2, +} \in \cS(\U_{V, +})$. Put
    \begin{equation}    \label{eq:kernel_U}
    T(f_{1, +}, f_{2, +})\\
    = \int_{\U_V'} \int_{\U_V'} \int_{\U_V} \int_{V}
    f_{1, +}^\ddag(h_1g h_2, vh_2) \overline{f_{2, +}^\ddag(g, v)}
    \rd v \rd g \rd h_1 \rd h_2.
    \end{equation}

\begin{lemma}   \label{lemma:kernel_convergence}
The integral~\eqref{eq:kernel_U} is absolutely convergent and defines a
continuous Hermitian form on $\cS(\U_{V,+})$.
\end{lemma}

\begin{proof}
We will assume that $F$ is Archimedean. The non-Archimedean case is similar
and easier. As $f_{i, +}^\ddag \in \cS(\U_V^+)$, Fatou's lemma implies that
the lemma is deduced from the following fact. For any $f_i \in \cS(\U_V)$ and
$\phi_i \in \cS(V)$, $i = 1, 2$, there are continuous semi-norms $\nu$ on
$\cS(\U_V)$ and $\nu'$ on $\cS(V)$ such that
    \begin{equation}    \label{eq:kernel_U_pure_tensor}
    \int_{\U_V'} \int_{\U_V'} \int_{\U_V} \int_{V}
    \abs{f_1(h_1g h_2) \phi_1(vh_2) f_2(g)\phi_2(v)}
    \rd g \rd v \rd h_1 \rd h_2
    \leq \nu(\varphi_1) \nu(\varphi_2) \nu'(\phi_1) \nu'(\phi_2).
    \end{equation}

Let us fix a norm function on $V$ as follows. We may choose a norm
$\aabs{\cdot}_V$ on $V$ that is invariant under the translation of $K$ (a
fixed maximal compact subgroup of $\U(V)$). The convergence
of~\eqref{eq:kernel_U_pure_tensor} then reduces to that there exists a $d>0$
such that
    \[
    \int_{\U_V'} \int_{\U_V'} \int_{\U_V} \int_V \abs{f_1(h_1g h_2) f_2(g)}
    (1+ \aabs{v})^{-d}
    (1+ \aabs{v h_2})^{-d} \rd v \rd g
    \rd h_1 \rd h_2
    \leq \nu(\varphi_1) \nu(\varphi_2).
    \]
Integrate over $h_1$ and $g_1$ first and change of variables. We are reduced to
prove that
    \[
    \int_{\U_V'} \int_{\U_V'} \int_V f_3(h_2^{-1} g_2 h_2) f_4(g_2)
    (1+ \aabs{v})^{-d}
    (1+ \aabs{v h_2})^{-d} \rd v \rd h_2 \rd g_2
    \]
is absolutely convergent,
where $f_3, f_4$ are positive Schwartz functions on $\U(V)$.

Since $\cS(\U(V))$ is contained in $\cC(\U(V))$ (the embedding is
continuous), we need to prove
    \[
    \int_{\U(V)} \int_{\U(V)} \int_V
    \Xi^{\U(V)}(h_2^{-1} g_2 h_2) \Xi^{\U(V)}(g_2) \varsigma(g_2)^{-d}
    (1+ \aabs{v})^{-d}
    (1+ \aabs{v h_2})^{-d} \rd v \rd h_2 \rd g_2
    \]
is convergent for sufficiently large $d$ (this is a rather crude reduction,
which however works). Using the doubling principle for $\Xi^{\U(V)}$, we are
reduced to the convergence of
    \[
    \int_{\U(V)} \Xi(g_2)^2 \varsigma(g_2)^{-d} \rd g_2
    \]
and
    \[
    \int_{\U(V)} \int_V \Xi(h_2)^2
    (1+ \aabs{v})^{-d}
    (1+ \aabs{v h_2})^{-d} \rd v \rd h_2
    \]
when $d$ is large. The first one is ~\cite{Wald}*{Lemma~II.1.5}. The second
one follows from the Cartan decomposition $h_2 = k_1 a k_2$, $a \in A^+$,
$k_1, k_2 \in K$ and
    \[
    \int_V
    (1+ \aabs{v})^{-d}
    (1+ \aabs{v a})^{-d} \rd v \leq
    C \times \abs{a_1\cdots a_r}_E
    \]
for some constant $C$.
\end{proof}

\begin{prop}    \label{prop:lrtf_u}
Let $f_{1, +}, f_{2, +} \in \cS(\U_{V, +})$. We have
    \[
    \int_{\cX_{\temp}(\U_V)} J_{\pi}(f_{1, +})
    \overline{J_{\pi}(f_{2, +})} \rd \mu_G(\pi) =
    \int_{\cA_{\rs}} O(y, f_{1, +})
    \overline{O(y, f_{2, +})} \rd y.
    \]
Both sides are absolutely convergent and define continuous Hermitian form on
$\cS(\U_{V, +})$. Note that if $y \in \cA$ is not in the image of $\U_{V}^+$
(recall by our convention this all mean their $F$-points), then we take the
convention that $O(y, f_{i, +}) = 0$.
\end{prop}

\begin{proof}
We compute the integral~\eqref{eq:kernel_U} in two ways. By
Lemma~\ref{lemma:kernel_convergence} we may change the order of integration.
We make the change of variables
    \[
    h_1 \mapsto h_1 h_2^{-1} g_1^{-1}, \quad g_2 \mapsto g_1 g_2,
    \]
then integrate $h_1$ and $g_1$ first. We combine the variables $g_2 \in
\U(V)$ and $v \in V$ as a single variable $y \in Y^V$. We then end up with
    \[
    T(f_1, f_2) =
    \int_{Y^V(F)} \int_{\U(V)}  \varphi_{f_{1, +}^\ddag}(y \cdot h_2)
    \overline{\varphi_{f_{2, +}^\ddag}(y)}
    \rd h_2 \rd y.
    \]
Here we recall that $\varphi_{f_{i, +}^\ddag}$ is the function defined
by~\eqref{eq:simplified_test_function_U}. As $Y^V_{\rs}$ has a measure zero
complement in $Y^V$ (recall the convention that they stands for the $F$-point
of the underline algebraic varieties), this equals
    \[
    \int_{Y^V_{\rs}} O(y, f_{1, +})
    \overline{\varphi_{f_{2, +}^\ddag}(y)} \rd y.
    \]
By the definition of the measure on $\cA_{\rs}$, cf.~Lemma~\ref{lemma:matching_measure}, this equals
    \[
    \int_{\cA^V_{\rs}(F)} O(y, f_{1, +})
    \overline{O(y, f_{2, +})} \rd y.
    \]

We now compute~\eqref{eq:kernel_U} spectrally and show that
    \[
    \eqref{eq:kernel_U} = \int_{\cX_{\temp}(\U_V)} J_{\pi}(f_{1, +})
    \overline{J_{\pi}(f_{2, +})} \rd \mu_G(\pi).
    \]
Note that the right hand side is absolutely convergent by Lemma~\ref{lemma:J_pi_continuity}.
Since both sides are continuous in $f_{1, +}$ and $f_{2, +}$, we may assume
that $f_{1, +} = f_1 \otimes \phi_1 \otimes \phi_2$, $f_{2, +} = f_2 \otimes
\phi_3 \otimes \phi_4$ where $f_1, f_2 \in \cS(\U_V)$ and $\phi_1, \hdots
\phi_4 \in \cS(L)$. First integrate over $v \in V$. Since $L^2$-norm is
preserved under the Fourier transform, the integral~\eqref{eq:kernel_U}
becomes
    \[
    \int_{\U_V'} \int_{\U_V'} \int_{\U_V}
    f_1(h_1g h_2) \overline{f_2(g)}
    \overline{\langle \omega^\vee(h_1) \phi_3, \phi_1 \rangle}
    \langle \omega^\vee(h_2) \phi_2, \phi_4 \rangle
    \rd g \rd h_1 \rd h_2,
    \]
which equals
    \[
    \int_{\U_V'} \int_{\U_V'}
    \left( f_2^* * L(h_1^{-1}) f_1 \right)(h_2)
    \overline{\langle \omega^\vee(h_1) \phi_3, \phi_1 \rangle}
    \langle \omega^\vee(h_2) \phi_2, \phi_4 \rangle
    \rd h_2 \rd h_1.
    \]
Here $f_2^*(g) = \overline{f_2(g^{-1})}$, $*$ stands for the usual
convolution product in $\cS(G(F))$, and $L(h_1^{-1}) f_1$ is the function $g \mapsto f_1(h_1 g)$. Using the Plancherel
formula~\eqref{eq:plancherel_general}, this integral equals
    \[
    \begin{aligned}
    \int_{\U_V'} \int_{\U_V'} \int_{\cX_{\temp}(\U_V)}
    &\Trace \left( \pi(h_2^{-1}) \pi( f_2^*) \pi(h_1^{-1}) \pi(f_1) \right)\\
    &\overline{ \langle \omega^\vee(h_1) \phi_3, \phi_1 \rangle}
    \langle \omega^\vee(h_2) \phi_2, \phi_4 \rangle
    \rd \mu_{\U_V}(\pi)\rd h_2 \rd h_1.
    \end{aligned}
    \]
Here we need to invoke results from~\cite{Xue6}*{Section~3}. We defined a linear form
    \[
    \cL_{\pi^J}: (\pi \otimeshat \omega^\vee) \otimeshat
    \overline{\pi \otimeshat \omega^\vee} \to \C^\times,
    \]
and a map
    \[
    L_{\pi}^{\phi, \phi'}: \pi \to \Hom(\overline{\pi},  \C^\times)
    \]
in~\cite{Xue6}*{Section~3.1}. Integrating over $h_1$ first,
by~\cite{Xue6}*{(3.4)}, the above integral equals
    \[
    \int_{\U_V'} \int_{\cX_{\temp}(\U_V)}
    \Trace \left( \pi(h_2^{-1}) \pi( f_2^*) L_{\pi}^{\phi_1, \phi_3} \pi(f_1) \right)\\
    \langle \omega^\vee(h_2) \phi_2, \phi_4 \rangle
    \rd \mu_{\U_V}(\pi)\rd h_2.
    \]
Integrating over $h_2$ we get
    \[
    \int_{\cX_{\temp}(\U_V)} \cL_{\pi^J}
    \left( \left( \pi( f_2^*) L_{\pi}^{\phi_1, \phi_3} \pi(f_1) \right)^{\phi_4, \phi_2} \right)
    \rd \mu_{\U_V}(\pi).
    \]
By~\cite{Xue6}*{Lemma~3.6} we have
    \[
    ( \pi( f_2^*) L_{\pi}^{\phi_1, \phi_3} \pi(f_1))^{\phi_4, \phi_2} =
    \pi( f_2^*)^{\phi_4, \phi_3} L_{\pi^J} \pi(f_1)^{\phi_1, \phi_2},
    \]
and by~\cite{Xue6}*{Lemma~3.5} we have
    \[
    \cL_{\pi^J}(\pi( f_2^*)^{\phi_4, \phi_3} L_{\pi^J} \pi(f_1)^{\phi_1, \phi_2}) =
    \cL_{\pi^J}(\pi( f_2^*)^{\phi_4, \phi_3}) \cL_{\pi^J}(\pi(f_1)^{\phi_1, \phi_2})).
    \]
By definition
    \[
    \cL_{\pi^J}(\pi( f_2^*)^{\phi_4, \phi_3}) =
    \overline{J_{\pi}(f_2 \otimes \phi_3 \otimes \phi_4)}, \quad
    \cL_{\pi^J}(\pi(f_1)^{\phi_1, \phi_2}) =
    J_{\pi}(f_1 \otimes \phi_1 \otimes \phi_2).
    \]
This proves the proposition.
\end{proof}

\subsection{Local trace formula on the linear groups}

Let us now consider $f_{1, +}, f_{2, +} \in \cS(G_+(F))$. Put
    \begin{equation}    \label{eq:kernel_gl}
    T'(f_{1, +}, f_{2, +}) = \int_{H} \int_{G'} \int_{G_+}
    f_{1, +}^\dag(x \cdot (h, g')) \overline{f_{2, +}^\dag(x)}
    \eta_{G'}(g')
    \rd x \rd h \rd g'.
    \end{equation}

\begin{lemma}   \label{lemma:lrtf_gl_geometric}
The integral~\eqref{eq:kernel_gl} is absolutely convergent. It defines a
continuous Hermitian form on $\cS(G'(F) \times E_n)$. Moreover
    \begin{equation}    \label{eq:lrtf_gl_geometric}
    T'(f_{1, +}, f_{2, +}) = \int_{\cA_{\rs}(F)} O(x, f_{1, +})
    \overline{O(x, f_{2, +})} \rd x,
    \end{equation}
where the integral is absolutely convergent.
\end{lemma}

\begin{proof}
The proof is essentially the same as the orbital integral part of Proposition~\ref{prop:lrtf_u}.
\end{proof}

We are going to calculate $T'(f_{1, +}, f_{2, +})$ spectrally. We will assume
that $f_{1, +} = f_1 \otimes \Phi_1$ and $f_{2, +} = f_2 \otimes \Phi_2$
where $f_1, f_2 \in \cS(G)$ and $\Phi_1, \Phi_2 \in \cS(E_n)$. We need some
preparations. The proof of the following lemma is analogous to Lemma
\ref{lemma:kernel_convergence} and ~\cite{BP1}*{Lemma~5.4.2~(ii)}, which we
omit.

\begin{lemma} \label{lemma:linear_convergence}
For every $\Phi_1, \Phi_2 \in \cS(E_n)$ and $\phi \in \cC^w(G)$ the integral
    \begin{equation*}
    \int_{H} \int_{E_n} \phi(h) \abs{\det h}_E^{\frac{1}{2}}
    \Phi_1(vh) \overline{\Phi_2(v)}dv dh
    \end{equation*}
is absolutely convergent and defines a continuous linear form in $\phi \in
\cC^w(G)$. In particular, if $\varphi_2 \in \cC^w(G_n)$, then for every $\varphi_1
\in \cC^w(G_n)$ the integral
    \begin{equation*}
    \int_{G_n} \int_{E_n} \varphi_1(h) \varphi_2(h) \abs{\det h}_E^{\frac{1}{2}}
    \Phi_1(vh) \overline{\Phi_2(v)}dv dh
    \end{equation*}
is absolutely convergent and defines a continuous linear form on $\varphi_1 \in
\cC^w(G_n)$.
\end{lemma}

Following~\cite{BP1}, for any irreducible tempered representation $\Pi$ of
$G$, we define a continuous Hermitian form on $\cS(G)$ by
    \[
    \langle f_1, f_2 \rangle_{X, \Pi} = \sum_{W \in \cW(\Pi, \psi_{N'})}
    \beta(\Pi(f_1^\vee) W) \overline{\beta(\Pi(f_2^\vee) W)}.
    \]
where $f^\vee(x)=f(x^{-1})$. Here the linear form $\beta$ is a variant of
$\beta_\eta$, cf.~\eqref{eq:local_FR}, and is given by
    \[
    \beta: \cC^w(N \bs G, \psi_N) \to \C, \quad W \mapsto
    \beta(W) = \int_{N' \bs P'} W(p) \rd p.
    \]
It is a continuous linear form on $\cC^w(N \bs G, \psi_N)$
by~\cite{BP1}*{Lemma~2.15.1}. The subscript $X$ stands for the symmetric
variety $G' \bs G$. For our purposes we do not need the details, but only
treat $\langle -, -\rangle_{X, \Pi}$ as a single piece of notation.

\begin{lemma}   \label{lemma:lrtf_gl_prep}
The function
    \[
    \Pi \mapsto \langle f_1, f_2 \rangle_{X, \Pi}
    \]
is Schwartz. The map
    \[
    \cS(G)^2 \to \cS(\Temp(G)), \quad
    (f_1, f_2) \mapsto (\Pi \mapsto \langle f_1, f_2 \rangle_{X, \Pi})
    \]
is continuous. Moreover
    \begin{equation}    \label{eq:lrtf_gl_prep}
    \begin{aligned}
    &\int_{H} \int_{E_n} (\mathrm{L}(h^{-1})f_1, f_2 )_{X, \Pi}
    \mu(\det h)^{-1} \abs{\det h}_E^{\frac{1}{2}}
    \Phi_1(vh) \overline{\Phi_2(v)} \rd v \rd h\\ =
    &\abs{\tau}_E^{-\frac{n(n-1)}{2}}
    I_{\Pi}(f_1\otimes \Phi_1) \overline{I_{\Pi}(f_2 \otimes \Phi_2)}.
    \end{aligned}
    \end{equation}
\end{lemma}

\begin{proof}
Let us focus on the proof of the equality~\eqref{eq:lrtf_gl_prep}. The proof
of the rest is exactly the same as those in~\cite{BP1}*{Lemma~4.21,
Proposition~5.43}.

First as~\cite{BP1}*{Proof of~(5.4.2)} the sum of functions
    \[
    g \mapsto \sum_{W_1, W_2 \in \cW(\Pi, \psi_{N})}
    \langle \Pi(g) W_1, W_2 \rangle^{\mathrm{Wh}}
    \beta(\Pi(f_1^\vee) W_1) \overline{\beta(\Pi(f_2^\vee) W_2)}
    \]
is absolutely convergent in $\cC^w(G)$. Here we call that $\langle-, -\rangle^{\mathrm{Wh}}$ is the inner product on the Whittaker models given by~\eqref{eq:whittaker_inner_product_local}. It follows that the expression
    \[
    \begin{aligned}
    \int_{H} \int_{E_n}
    &\sum_{W_1, W_2 \in \cW(\Pi, \psi_{N})}
    \langle \Pi(h) W_1, W_2 \rangle^{\mathrm{Wh}}
    \beta(\Pi(f_1^\vee) W_1) \overline{\beta(\Pi(f_2^\vee) W_2)}\\
    &\mu(\det h)^{-1} \abs{\det h}_E^{\frac{1}{2}}
    \Phi_1(vh) \overline{\Phi_2(v)} \rd v \rd h
    \end{aligned}
    \]
is absolutely convergent and equals the left hand side
of~\eqref{eq:lrtf_gl_prep}. We can thus switch the order of the sum and the
integrals. We will prove that for all $W_1, W_2 \in \cW(\Pi, \psi_{N})$ we
have
    \begin{equation}    \label{eq:local_I_Pi}
    \begin{aligned}
    &\int_{H} \int_{E_n}
    \langle \Pi(h) W_1, W_2 \rangle^{\mathrm{Whitt}}
    \mu(\det h)^{-1} \abs{\det h}_E^{\frac{1}{2}}
    \Phi_1(vh) \overline{\Phi_2(v)} \rd v \rd h\\
    =
    &\abs{\tau}_E^{-\frac{n(n-1)}{2}}
    \lambda(W_1, \Phi_1) \overline{\lambda(W_2, \Phi_2)},
    \end{aligned}
    \end{equation}
which implies~\eqref{eq:lrtf_gl_prep} directly.

First both sides of~\eqref{eq:local_I_Pi} are continuous linear forms on
$W_1$ and $W_2$ by~\cite{BP1}*{(2.6.1)} and
Lemma~\ref{lemma:linear_convergence}, and thus we may assume that $\Pi = \Pi'
\boxtimes \Pi''$ where $\Pi', \Pi''$ are irreducible tempered representations
of $G_n$, and $W_i = W_i' \otimes W_i''$ where
    \[
    W_i' \in \cW(\Pi', \xi), \quad W_i'' \in \cW(\Pi'', \overline{\xi}).
    \]
We claim that for any $\phi \in \cC^w(G_n)$ and any $W_1'', W_2'' \in
\cW(\Pi'', \overline{\xi})$, we have
    \begin{equation}    \label{eq:a_trick_for_convergence}
    \begin{aligned}
    &\int_{G_n} \int_{E_n}
    \phi(h)
    \langle \Pi(h)W''_1, W''_2 \rangle^{\mathrm{Wh}}
    \mu(\det h)^{-1} \abs{\det h}_E^{\frac{1}{2}}
    \Phi_1(vh) \overline{\Phi_2(v)} \rd v \rd h\\ =
    & \iint_{(N_n \bs G_n)^2} W_{\phi}(h, g)
    W_1''(h) \overline{W_2''(g)}
    \mu(\det hg)^{-1} \abs{\det hg}_E^{\frac{1}{2}}
    \Phi_1(e_n h) \overline{\Phi_2(e_n g)} \rd h \rd g,
    \end{aligned}
    \end{equation}
where we recall that $W_\phi$ was defined in~\eqref{eq:definition_W_phi}.

Let us first explain that this implies~\eqref{eq:local_I_Pi}. The left hand
side of~\eqref{eq:local_I_Pi} equals
    \[
    \int_{G_n} \int_{E_n}
    \langle \Pi'(h)W_1', W_2' \rangle^{\mathrm{Whitt}}
    \langle \Pi''(h)W_1'', W_2'' \rangle^{\mathrm{Whitt}}
    \mu(\det h)^{-1} \abs{\det h}_E^{\frac{1}{2}}
    \Phi_1(vh) \overline{\Phi_2(v)} \rd v \rd h.
    \]
We obtain~\eqref{eq:local_I_Pi} by
applying~\eqref{eq:a_trick_for_convergence} to
    \[
    \phi(g) = \langle \Pi'(g)W_1', W_2' \rangle_{G_n}^{\mathrm{Whitt}}
    \]
and using~\eqref{eq:whittaker_expansion}.

Finally we prove~\eqref{eq:a_trick_for_convergence}. Both sides are
continuous linear forms in $\phi$ (this follows from Lemma
\ref{lemma:linear_convergence} for the left hand side, and from Lemma
\ref{lemma:lambda_continuous} for the right), and therefore we may assume
that $\phi \in \cS(G_n)$, which makes everything absolutely convergent. We
first replace the integral over $v \in E_n$ by $g \in P_n \bs G_n$. The
left hand side of~\eqref{eq:a_trick_for_convergence} equals
    \[
    \int_{H} \int_{P_n \bs G_n}
    \phi(h)
    \langle \Pi(h)W_1, W_2 \rangle^{\mathrm{Wh}}
    \mu(\det h)^{-1} \abs{\det h}_E^{\frac{1}{2}} \abs{\det g}_E
    \Phi_1(e_ngh) \overline{\Phi_2(e_ng)} \rd g \rd h.
    \]
Make a change of variable $h \mapsto g^{-1}h$ we end up with
    \[
    \int_{P_n \bs (G_n \times G_n)}
    \phi(g^{-1} h)
    \langle \Pi(h)W_1, \Pi(g) W_2 \rangle_{G_n}^{\mathrm{Wh}}
    \mu(\det g^{-1} h)^{-1} \abs{\det hg}_E^{\frac{1}{2}}
    \Phi_1(e_nh) \overline{\Phi_2(e_ng)} \rd g \rd h.
    \]
Here $P_n$ embeds in $G_n \times G_n$ diagonally. Plugging in the definition
of $\langle-,-\rangle^{\mathrm{Wh}}$ we obtain
    \[
    \int_{N_n \bs (G_n \times G_n)}
    \phi(g^{-1} h)
    W_1(h)\overline{W_2(g)}
    \mu(\det g^{-1} h)^{-1} \abs{\det hg}_E^{\frac{1}{2}}
    \Phi_1(e_nh) \overline{\Phi_2(e_ng)} \rd g \rd h.
    \]
Finally we decompose the integration of over $N_n \bs (G_n \times
G_n)$ as a integral over $N_n$ followed by a double integral over
$(N_n \bs G_n)^2$ and conclude that the above integral equals
    \[
    \int_{(N_n \bs G_n)^2} \int_{N_n}
    \phi(g^{-1} u h) \overline{\xi(u)}
    W_1(h)\overline{W_2(g)}
    \mu(\det g^{-1} h)^{-1} \abs{\det hg}_E^{\frac{1}{2}}
    \Phi_1(e_nh) \overline{\Phi_2(e_ng)} \rd u \rd g \rd h.
    \]
By the definition of $W_{\phi}$, this equals the right hand side
of~\eqref{eq:a_trick_for_convergence}. This finishes the proof of the lemma.
\end{proof}

\begin{prop}    \label{prop:lrtf_gl}
We have
    \[
    \begin{aligned}
    &\abs{\tau}_E^{-\frac{n(n-1)}{2}}
    \int_{\Temp(\U_{V_{qs}})/\sim} I_{\BC(\pi)}(f_{1, +})
    \overline{I_{\BC(\pi)}(f_{2, +})}
    \frac{\abs{\gamma^*(0,\pi, \Ad, \psi)}}{\abs{S_{\pi}}} \rd \pi\\
    = &
    \int_{\cA_{\rs}(F)} O(x, f_{1, +})
    \overline{O(x, f_{2, +})} \rd x.
    \end{aligned}
    \]
\end{prop}

\begin{proof}
By Lemma~\ref{lemma:lrtf_gl_geometric}, we just need to prove that
    \begin{equation}    \label{eq:lrtf_gl_spectral}
    T'(f_{1, +}, f_{2, +}) = \abs{\tau}_E^{-\frac{n(n-1)}{2}}
    \int_{\Temp(\U_{V_{qs}})/\sim} I_{\BC(\pi)}(f_{1, +})
    \overline{I_{\BC(\pi)}(f_{2, +})}
    \frac{\abs{\gamma^*(0,\pi, \Ad, \psi)}}{\abs{S_{\pi}}} \rd \pi.
    \end{equation}
The absolute convergence of the integral follows from the fact that the
function $\Pi \mapsto I_\Pi(f_i')$ is Schwartz, and the function
$\frac{\abs{\gamma^*(0,\pi, \Ad, \psi)}}{\abs{S_{\pi}}}$ is of moderate
growth.

Since both sides of~\eqref{eq:lrtf_gl_spectral} are continuous in both $f_{1,
+}$ and $f_{2, +}$ by Lemma~\ref{lemma:Schwartz_I_Pi}, we may assume that
they both lie in $\cS(G) \otimes \cS(E_n)$. Thus we are reduced to calculate
    \[
    T'(f_1 \otimes \Phi_1, f_2 \otimes \Phi_2)
    \]
where $f_1, f_2 \in \cS(G)$ and $\Phi_1, \Phi_2 \in \cS(E_n)$.

Since partial Fourier transform preserves the $L^2$-norm, integrating over
$E_n^- \times E^{-, n}$ first gives
    \[
    \int_{H} \int_{G'} \int_{G} \int_{E_n}
    f_1(h x g') \overline{f_2(x)}
    \Phi_1(v h)
    \overline{\Phi_2(v)}
    \mu(\det h)^{-1} \abs{\det h}_E^{\frac{1}{2}}
    \eta_{n+1}(g') \rd v \rd x \rd g' \rd h.
    \]

This expression is again absolutely convergent, and hence we can integrate
$g'$ and $x$ first. By~\cite{BP1}*{Proposition~5.61} the integrals over $g'$
and $x$ gives gives
    \[
    \int_{\Temp(\U_{V_{qs}})/\sim}
    \langle \mathrm{L}(h_1^{-1})f_1, f_2\rangle_{X, \BC(\pi)}
    \Phi_1(v h_1)
    \overline{\Phi_2(v)}
    \frac{\abs{\gamma^*(0,\pi, \Ad, \psi)}}{\abs{S_{\pi}}}
    \rd \pi.
    \]
Strictly speaking the statement in~\cite{BP1} is not exactly the same as
this. But as in~\cite{BP1}*{Proposition~5.61}, since both sides of continuous
linear forms on $f_1$ and $f_2$ we may assume that $f_1, f_2 \in \cS(G_n)
\otimes \cS(G_n)$, and thus the desired equality boils down to two equalities
on $G_n$ and both follow directly from~\cite{BP1}*{Theorem~4.22} (the case
in~\cite{BP1} boils down to one on $G_n$ and the other on $G_{n+1}$).

We thus obtain
    \[
    \int_{H} \int_{E_n}
    \int_{\Temp(\U_{V_{qs}})/\sim}
    \langle \mathrm{L}(h^{-1})f_1, f_2\rangle_{X, \BC(\pi)}
    \Phi_1(v h)
    \overline{\Phi_2(v)}
    \frac{\abs{\gamma^*(0,\pi, \Ad, \psi)}}{\abs{S_{\pi}}}
    \rd \pi \rd v \rd h.
    \]
As the case of~\cite{BP1}*{(5.6.9)}, this expression is absolutely
convergent, and the identity~\eqref{eq:lrtf_gl_spectral} follows from
Lemma~\ref{lemma:lrtf_gl_prep}.
\end{proof}

\subsection{Character identities}
We are ready to compare the relative characters. We begin with the following
proposition, which is a weaker version of it.

\begin{prop}    \label{prop:weak_comparison}
For any $\pi \in \Temp_{\U_V'}(\U_V)$, there is a constant $\kappa(\pi)$ such
that
    \[
    J_{\pi}(f^V_+) = \kappa(\pi) I_{\BC(\pi)}(f_+),
    \]
for matching test functions $f_+ \in \cS(G_+)$ and $f^V_+ \in \cS(\U_{V,
+})$. Moreover the function
    \[
    \pi \in \Temp_{\U_V'}(\U_V) \mapsto \kappa(\pi)
    \]
is continuous. Here the topology on $\Temp_{\U_V'}(\U_V)$ is discussed in
Subsection~\ref{subsec:tempered_intertwining}.
\end{prop}

\begin{proof}
The proof is the same as~\cite{BP}*{Proposition~4.2.1} and is left to the
interested reader, cf.~\cite{BP1}*{Proposition~5.81}.
\end{proof}

We now compute $\kappa(\pi)$. In theory we could compute it directly using
the spectral expansions of the unipotent orbital integrals. But when $F$ is
Archimedean, due to the lack the matching for all test functions, we need an
a priori computation of $\abs{\kappa(\pi)}$, which we could only achieve
using the local relative trace formulae.

\begin{prop}    \label{prop:abs_value_kappa}
For any $\pi \in \Temp_{\U_V'}(\U_V)$ we have $\abs{\kappa(\pi)} =
\abs{\tau}_E^{-\frac{n(n-1)}{4}}$. In particular it is independent of $\pi$.
\end{prop}

\begin{proof}
Take $f^V_{1, +}, f^V_{2, +} \in \cS(\U_{V, +})$ and $f_{1, +}, f_{2, +} \in
\cS(G_+)$. Assume that they match, then by Proposition~\ref{prop:lrtf_u} and
Proposition~\ref{prop:lrtf_gl} we have
    \[
    \begin{aligned}
    &\abs{\tau}_E^{-\frac{n(n-1)}{2}}
    \int_{\Temp(\U_{V_{qs}})/\sim} I_{\BC(\pi)}(f_{1, +})
    \overline{I_{\BC(\pi)}(f_{2, +})}
    \frac{\abs{\gamma^*(0,\pi, \Ad, \psi)}}{\abs{S_{\pi}}} \rd \pi\\
     =
    & \sum_{V \in \cH} \int_{\Temp_{\U_V'}(\U_V)} J_{\pi}(f_{1, +}^V)
    \overline{J_{\pi}(f_{2, +}^V)}
    \rd \mu_{\U_V}(\pi).
    \end{aligned}
    \]
It follows from Proposition \ref{prop:weak_comparison}
that
    \[
    \begin{aligned}
    &\abs{\tau}_E^{-\frac{n(n-1)}{2}}
    \int_{\Temp(\U_{V_{qs}})/\sim} I_{\BC(\pi)}(f_{1, +}) \overline{I_{\BC(\pi)}(f_{2, +})}
    \frac{\abs{\gamma^*(0,\pi, \Ad, \psi)}}{\abs{S_{\pi}}} \rd \pi\\
     =
    & \sum_{V \in \cH} \int_{\Temp_{\U_V'}(\U_V)}
    I_{\BC(\pi)}(f_{1, +}) \overline{I_{\BC(\pi)}(f_{2, +})}
    \abs{\kappa(\pi)}^2
    \rd \mu_{\U_V}(\pi).
    \end{aligned}
    \]
Using the integration formula~\cite{BP1}*{(2.10.1)}, and the formal degree
conjecture for unitary groups~\cite{BP1}*{Theorem~5.53} (this is one of the
main results of~\cite{BP1}), we conclude
    \begin{equation}    \label{eq:lrtf_comparison}
    \begin{aligned}
    &\abs{\tau}_E^{-\frac{n(n-1)}{2}}
    \int_{\Temp(\U_{V_{qs}})/\sim} I_{\BC(\pi)}(f_{1, +})
    \overline{I_{\BC(\pi)}(f_{2, +})}
    \frac{\abs{\gamma^*(0,\pi, \Ad, \psi)}}{\abs{S_{\pi}}} \rd \pi\\
     =
    & \int_{\Temp(\U_{V_{qs}})/\sim} \left( \sum_V
    \sum_{\substack{\sigma \in \Temp_{\U_V'}(\U_V)\\
    \sigma \sim \pi}} \abs{\kappa(\sigma)}^2 \right)
    I_{\BC(\pi)}(f_{1, +}) \overline{I_{\BC(\pi)}(f_{2, +})}
    \frac{\abs{\gamma^*(0,\pi, \Ad, \psi)}}{\abs{S_{\pi}}} \rd \pi.
    \end{aligned}
    \end{equation}

A priori, this identity holds for all transferrable
$f_{1, +}, f_{2, +}$. We now explain that it holds for all $f_{1, +}, f_{2,
+}$. If $F$ is non-Archimedean, there is nothing to prove as all test
functions are transferable. Assume $F$ is Archimedean. Since transferable
functions form a dense subspace of $\cS(G_+)$, it is enough to explain that
both sides define continuous linear forms in $f_{1, +}, f_{2, +}$. Since for
any $\Pi \in \Temp(G)$, $I_{\Pi}$ is a continuous linear form, it is enough
to explain that both sides of~\eqref{eq:lrtf_comparison} are absolutely
convergent. This is clear for the left hand side, as the map $\pi \to
I_{\BC(\pi)}$ and $\frac{\abs{\gamma^*(0,\pi, \Ad, \psi)}}{\abs{S_{\pi}}}$ is of
moderate growth in $\pi$. It remains to explain the absolute convergence of
the right hand side. For this we apply the Cauchy--Schwarz inequality to
conclude that if $f_{1, +}, f_{2, +}$ are transferable, then the right hand
side of~\eqref{eq:lrtf_comparison}, when $I_{\BC(\pi)}(f_{i, +})$'s are
replaced by their absolute values, is bounded by
    \[
    \abs{\tau}_E^{-\frac{n(n-1)}{2}}
    \nu(f_{1, +})^{\frac{1}{2}} \nu(f_{2, +})^{\frac{1}{2}}
    \]
where $\nu$ us the continuous seminorm on $\cS(G_+)$ given by
    \[
    \nu(f_{i, +}) =
    \int_{\Temp(\U_{V_{qs}})/\sim} \abs{I_{\BC(\pi)}(f_{i, +})}^2
    \frac{\abs{\gamma^*(0,\pi, \Ad, \psi)}}{\abs{S_{\pi}}} \rd \pi.
    \]
Then by Fatou's lemma and the density of transferable functions we conclude
that the same holds for all $f_{1, +}, f_{2, +}$. This proves the absolute
convergence of the right hand side.

Once we have that~\eqref{eq:lrtf_comparison} holds for all $f_{1, +}, f_{2,
+}$, we apply the spectrum-separating technique of~\cite{BP1}*{Lemma~5.82} and
conclude that
    \[
    \abs{\tau}_E^{-\frac{n(n-1)}{2}}
    I_{\BC(\pi)}(f_{1, +})
    \overline{I_{\BC(\pi)}(f_{2, +})} =
    \sum_V
    \sum_{\substack{\sigma \in \Temp_{\U_V'}(\U_V)\\
    \sigma \sim \pi}} \abs{\kappa(\sigma)}^2
    I_{\BC(\pi)}(f_{1, +})
    \overline{I_{\BC(\pi)}(f_{2, +})},
    \]
for each $\pi \in \Temp(\U_{V_{qs}})/\sim$. By the local Gan--Gross--Prasad
conjecture~\cites{GI2,Xue6}, there is a unique $\sigma$ in the double sum on
the right hand side such that $\sigma \in \Temp_{\U_V'}(\U_V)$. Since
$I_{\BC(\pi)}$ is not identically zero, which can be deduced from
Proposition~\ref{prop:explicit_intertwining}, we conclude that
    \[
    \abs{\kappa(\sigma)}^2 = \abs{\tau}_E^{-\frac{n(n-1)}{2}}
    \]
for this $\sigma$.
\end{proof}

\begin{lemma}   \label{lemma:nilpotent_matching}
Let $f_+ \in \cS(G_+)$ and the collection $f_+^V \in
\cS(\U_{V,+})$ be matching test functions. Then
    \[
    \gamma O_+(f_+) =
    \sum_{V \in \cH} \eta((2\tau)^n \disc V)^n
    \mu((-1)^{n-1} \tau)^{-\frac{n(n+1)}{2}}
    \lambda(\psi)^{\frac{n(n+1)}{2}}
    O(1, f_+^V).
    \]
We recall that $\gamma$ is the constant defined in~\eqref{eq:gamma_constant}.
\end{lemma}

\begin{proof}
By~\eqref{eq:reg_nilpotent} we have
    \[
    \gamma \mu((-1)^{n-1} \tau)^{\frac{n(n+1)}{2}}
    O_+(f_+) =
    \int_{\cB_{\rs}} \Omega(x)
    O(x, \widehat{f_{+, \natural}}) \rd x.
    \]
By Lemma~\ref{lemma:matching_measure}, matching of orbits preserves
measures, and by Theorem~\ref{thm:Fourier_matching}, $\widehat{f_\natural}$
and the collection of functions $\lambda(\psi)^{\frac{n(n+1)}{2}}
\eta((2\tau)^n \disc V)^n \cF_V f^V_{+, \natural}$ match. Thus we have
    \[
    \begin{aligned}
    &\gamma \mu((-1)^{n-1} \tau)^{\frac{n(n+1)}{2}}
    O_+(f_+)\\ = &\sum_V
    \lambda(\psi)^{\frac{n(n+1)}{2}} \eta((2\tau)^n \disc V)^n
     \int_{\cB_{\mathrm{rs}}}
    O(y, \cF_{V} f^V_{+, \natural}) \rd y.
    \end{aligned}
    \]
By~\eqref{eq:spectral_nilpotent_u}, the right hand side equals
    \[
    \sum_V
    \lambda(\psi)^{\frac{n(n+1)}{2}} \eta((2\tau)^n \disc V)^n
    O(1, f_+^V).
    \]
This proves the lemma.
\end{proof}

\begin{theorem} \label{thm:comparison_character}
For all $\pi \in \Temp_{\U_V'}(\U_V)$, we have
    \[
    \kappa(\pi) = \abs{\tau}_E^{-\frac{n(n-1)}{4}}
    \eta((2\tau)^n \disc V)^n
    \mu((-1)^{n-1} \tau)^{\frac{n(n+1)}{2}}
    \lambda(\psi)^{\frac{n(n+1)}{2}}.
    \]
\end{theorem}

\begin{proof}
For each $n$-dimensional skew-Hermitian space $V$, we put
    \[
    \epsilon_V=\eta((2\tau)^n \disc V)^n
    \mu((-1)^{n-1} \tau)^{-\frac{n(n+1)}{2}}
    \lambda(\psi)^{\frac{n(n+1)}{2}}.
    \]

If $f_+$ and the collection $f_+^V$ match, by
Lemma~\ref{lemma:nilpotent_matching} we have
    \[
    \gamma
    O_+(f_+) =
    \sum_{V} \epsilon_V
    O(1, f_+^V),
    \]
where the sum runs over all (isomorphism classes of) nondegenerate $n$-dimensional
skew-Hermitian space. By Lemma~\ref{lemma:plancherel_U} and
Proposition~\ref{prop:plancherel_gl}, we have
    \[
    \abs{\tau}_E^{-\frac{n(n-1)}{4}}
    \int_{\Temp(\U_{V_{qs}})/\sim} I_{\BC(\pi)}(f_+)
    \frac{\abs{\gamma^*(0,\pi, \Ad, \psi)}}{\abs{S_{\pi}}}
    \rd \pi = \sum_V \epsilon_V
    \int_{\Temp_{\U_V'}(\U_V)}
    J_{\pi}(f_+^V) \rd \mu_{\U_V}(\pi).
    \]
Using Proposition~\ref{prop:weak_comparison}, the integration
formula~\cite{BP1}*{(2.10.1)} and the formal degree
conjecture~\cite{BP1}*{Theorem~5.53} we have
    \begin{equation}    \label{eq:final_comparison}
    \begin{aligned}
    &\abs{\tau}_E^{-\frac{n(n-1)}{4}}
    \int_{\Temp(\U_{V_{qs}})/\sim} I_{\BC(\pi)}(f')
    \frac{\abs{\gamma^*(0,\pi, \Ad, \psi)}}{\abs{S_{\pi}}}
    \rd \pi\\
    = & \int_{\Temp(\U_{V_{qs}})/\sim} \left( \sum_V \epsilon_V
    \sum_{\substack{\sigma \in \Temp_{H^V(F)}(G^V(F))\\
    \sigma \sim \pi}} \kappa(\sigma) \mu^*_{G^V}(\sigma) \right)
    I_{\BC(\pi)}(f') \frac{\abs{\gamma^*(0,\pi, \Ad, \psi)}}{\abs{S_{\pi}}}
    \rd \pi.
    \end{aligned}
    \end{equation}

This identity a priori holds for transferable $f'$. We now explain that it
holds for all $f'$. If $F$ is non-Archimedean, the all $f'$ are transferable.
If $F$ is Archimedean, this holds for all transferable $f'$. Assume that $F$
is Archimedean. Then as in the proof of
Proposition~\ref{prop:abs_value_kappa}, we need to explain that both sides
are absolutely convergent. The left hand side is clear as $\pi \mapsto
I_{\BC(\pi)}$ is Schwartz. The absolute convergence of the right hand side follows
from the additional fact that $\abs{\kappa(\sigma)}$ is a constant for all
$\sigma \in \Temp_{\U(V)}(\U_V)$ by Proposition~\ref{prop:abs_value_kappa},
and that the inner double sum contains a unique nonzero term by the local
Gan-Gross-Prasad conjecture.

Using the spectrum separating technique as explained
in~\cite{BP1}*{Lemma~5.82}, thanks to Proposition
\ref{prop:explicit_intertwining}, we conclude that
    \[
    \sum_V \epsilon_V
    \sum_{\substack{\sigma \in \Temp_{\U(V)}(\U_V)\\
    \sigma \sim \pi}} \kappa(\sigma) =
    \abs{\tau}_E^{-\frac{n(n-1)}{2}}.
    \]
By the local GGP conjecture again, there is a unique nonzero term on the
right hand side. The theorem then follows.
\end{proof}

\subsection{Character identity and matching}
Let $V$ be a nondegenerate $n$-dimensional skew-Hermitian space. For any $\pi \in
\Temp(\U_V)$, let $\kappa(\pi)$ be as
Theorem~\ref{thm:comparison_character}. The following theorem characterizes
the matching of test functions via spherical character identities.

\begin{theorem} \label{thm:spectral_characterization_transfer}
Let $f_+^V \in \cS(\U_{V,+})$, $f_+ \in \cS(G_+)$. The following are
equivalent.
\begin{enumerate}
    \item For all $\pi \in \Temp_{\U_V'}(\U_V)$ we have
      \[
    \kappa(\pi) I_{\BC(\pi)}(f_+)=J_{\pi}(f^V_+).
    \]
    \item For all matching $x \in G_+$ and $y \in \U_{V,+}$ we have
    \[
    \Omega(x) O(x, f_+) = O(y, f^V).
    \]
\end{enumerate}
\end{theorem}

\begin{proof}
(2) $\implies$ (1) is Proposition~\ref{prop:weak_comparison}. Let us prove
that (1) $\implies$ (2). Let $f_{1, +}^V \in \cS(\U_{V, +})$ and $f_{1, +}
\in \cS(G_+)$ be matching test functions. By Proposition~\ref{prop:lrtf_u},
Proposition~\ref{prop:lrtf_gl} and Theorem~\ref{thm:comparison_character}, we
have
    \begin{equation}    \label{eq:spectral_characterization}
    \int_{\cA_{rs}(F)} O(x, f_{+})
    \overline{ O(x, f_{1, +}) } \rd x
    = \int_{\cA_{rs}(F)} O(y, f^V_+)
    \overline{O(y, f^V_{1, +})} \rd y.
    \end{equation}
Let $x \in \cA_{rs}$ be a regular semisimple point, and $U \subset \cA_{rs}$
a small neighbourhood of it. Since the maps $G^+ \to \cA$ and $\U_V^+ \to
\cA$ are locally a fibration near $x$, we see that for any $\varphi \in
C_c^\infty(U)$ there are $f_{1, +}$ and $f_{1, +}^V$ such that
    \[
    \varphi(x) =  \Omega(x) O(x, f_{1,+})  =
    O(x, f_{1, +}^V).
    \]
By definition $f_{1, +}$ and $f_{1, +}^V$ match. Since $\varphi$ is
arbitrary, we conclude from~\eqref{eq:spectral_characterization} that
    \[
    \Omega(x) O(x, f_+) = O(y, f_+^V).
    \]
This proves (1) $\implies$ (2).
\end{proof}

\section{The split case}    \label{sec:split_place_comparison}

\subsection{Setup}
We now make a few remarks on the case $E = F \times F$. The Galois conjugation
$\mathsf{c}$ swaps the two factor of $E$. Again
we can speak of regular semisimple elements, orbital integrals, and
spherical characters.

We can find a $\tau_1 \in F^\times$ such that $\tau = (\tau_1, -\tau_1)$. We
have a nontrivial additive character $\psi$ of $F$ and $\psi^{E}((x_1, x_2))
= \psi(\tau_1(x_1 -x_2))$. The character $\mu$ takes the form $\mu = (\mu_1,
\mu_1^{-1})$ where $\mu_1$ is a character of $F^\times$. We have the
identification $G = G_n \times G_n = (G_n' \times G'_n) \times (G'_n \times
G'_n)$, $H = G_n =  G_n' \times G_n'$ embedded in $G$ via $(h_1, h_2) \mapsto
((h_1, h_2), (h_1, h_2))$, and $G' = G_n' \times G_n'$ embedded by $(g_1, g_2)
\mapsto ((g_1, g_1), (g_2, g_2))$. The symmetric space $S_n$ consists of
elements of the form $(a, a^{-1})$ where $a \in G'_n$. It is identified with
$G_n'$ via the projection to the first factor. The projection $\nu : G_n
\to S_n$ is given by $(a,b) \mapsto (ab^{-1},a^{-1}b)$.
The partial Fourier transform
$-^\dag$ is given by
    \[
    \cS(E_{n}) = \cS(F_n \times F_n) \to \cS(F_{n} \times F^n), \quad
    \Phi^\dag(x, y) = \int_{F_n} \Phi(l+x, l-x)
    \psi((-1)^n \tau_1 ly) \rd l
    \]
Here we have made the identification that $E^- = \{(x, -x) \mid x\in F\} =
F$. Recall that if $f = f_1 \otimes f_2 \in \cS(G)$ where $f_i \in \cS(G_n)$,
and $\Phi \in \cS(F_n \times F_n)$, we have defined a function $\varphi_{(f
\otimes \Phi)^\dag} \in \cS(X)$ by~\eqref{eq:simplified_test_function_GL}.
For $a \in G_n'$, $x = ((a, a^{-1}), w, v) \in X = S_n \times F_{n} \times
F^n$, we have
    \begin{equation}
    \begin{aligned}
    \varphi_{(f \otimes \Phi)^\dag}(x) =
     \int_{G_n'} \int_{H}
    f_1(h^{-1}) f_2(h^{-1} (a, 1) g)
    (R_{\mu^{-1}}(h)\Phi)^\dag (w, v)
    \rd h \rd g.
    \end{aligned}
    \end{equation}
The orbital integral equals
    \begin{equation}    \label{eq:oi_gl_split}
    \begin{aligned}
    &O((((1, 1), (a, 1)), w, v), f \otimes \Phi) \\
    =
    & \int f_1(h_1^{-1} h_2^{-1}, h_2^{-1})
    f_2(h_1^{-1} g_1^{-1} a g_2, g_1^{-1} g_2)
    \left( \mathrm{R}_{\mu^{-1}}((h_1, 1)) \Phi\right)^\dag (wg_1, g_1^{-1} v)
    \rd h_1 \rd h_2 \rd g_1 \rd g_2.
    \end{aligned}
    \end{equation}
Here the integration is over $h_1, h_2, g_1, g_2 \in G_n'$.

We identify the skew-Hermitian space with $V = F^n \times F^n$ with the
skew-Hermitian form given by
    \[
    q_{V}((x_1, x_2), (y_1, y_2)) =
    (\tp{x_1} y_2,  - \tp{x_2} y_1).
    \]
A polarization of $\Res V$ is given by
    \[
    \Res V = L + L^\vee, \quad L = \{(0, x) \mid x \in F^n\}, \quad
    L^\vee = \{(x, 0) \mid x \in F^n\}.
    \]
Identify $\U(V)$ with the subgroup $\{(g, \tp{g}^{-1}) \mid g\in G_n'\}$ of
$G_n$, which is isomorphic to $G_n'$ via the projection to the first factor.
Then $\U_V$ is identified with $G_n' \times G_n'$ and $\U(V) = G_n'$ embeds
in $\U_V$ diagonally. Identify both $L$ and $L^\vee$ with $F^{n}$. The Weil
representation $\omega$ is realized on $\cS(L^\vee)$ and is given by
    \[
    \omega(g) \phi(x) = \abs{\det g}^{\frac{1}{2}} \mu_1(\det g) \phi(\tp{g}x),
    \quad
    \phi \in \cS(F^{n}).
    \]
We observe that the representation $\omega^\vee \otimes \omega$ of $G_n'
\times G_n'$ is isomorphic to $\mathrm{R}_{\mu^{-1}}$. The partial Fourier transform
$-^\ddag$ is
    \[
    \cS(F^{n} \times F^{n}) \to \cS(F^{n} \times F^{n}), \quad
    (\phi_1 \otimes \phi_2)^\ddag(x, y) = \int_{F^{n}} \phi_1(u+x)
    \phi_2(u-x) \psi(-  2 \tp{u} y) \rd l,
    \]
cf.~\eqref{eq:local_ddag_map}.

Recall that if $f^V = f_1^V \otimes f_2^V \in
\cS(G_n' \times G_n')$, $f_1^V, f_2^V \in \cS(G_n')$, and $\phi_1^V, \phi_2^V
\in \cS(F^{n})$, we have defined a function $\varphi_{(f^V \otimes \phi_1^V
\otimes \phi_2^V)^\ddag}$ by~\eqref{eq:simplified_test_function_U}. Though
there is only one $V$, we write the supscript to distinguish the notation
from the $\GL$-side.  If $a \in G_n'$ and $y = (a, w, v) \in Y = (G_n' \times
G'_n) \times F^{n} \times F^n$ is regular semisimple, then we have
    \begin{equation}
    \varphi_{(f \otimes \phi_1 \otimes \phi_2)^\ddag}(y) =
    \int_{G_n'}
    f(g^{-1}(1, a))
    (\omega^\vee(g)\phi_1 \otimes \phi_2)^\ddag(x, y)
    \rd g.
    \end{equation}
The orbital integral equals
    \begin{equation}    \label{eq:oi_u_split}
    \begin{aligned}
    &O(((1, a), (x, y)), f^V \otimes \phi_1^V \otimes \phi_2^V)\\ =
    & \int_{G_n'} \int_{G_n'}
    f_1^V(h^{-1}) f_2^V(h^{-1} g^{-1} a g)
    \left( \omega^\vee(h) \phi_1^V \otimes \phi_2^V \right)^\ddag
    (\tp{g}x, g^{-1}y)
    \rd h \rd g.
    \end{aligned}
    \end{equation}

\subsection{Matching}   \label{subsec:matching_split}
Let us now consider the matching of orbital integrals. Let
    \[
    x = ((g_1, g_2), w, v) \in G^+, \quad
    y = ((a_1, a_2), x, y) \in \U_V^+,
    \]
be regular semisimple elements. Put $\gamma = \nu(g_1^{-1} g_2) \in S_n
\simeq G_n'$ and $\delta \in a_1^{-1} a_2 \in \U(V) \simeq G_n'$. By
definition they match if $\gamma$ and $\delta$ are conjugate in $G_n'$ and $w
\gamma^i v = 2 (-1)^{n-1} \tau_1^{-1} \tp{x} \delta^{i} y$
for all $i = 0, \hdots, n-1$. The latter condition is equivalent to
$(-1)^n \tau_1 w \gamma^i v = -2 \tp{x} \delta^{i} y$.

\begin{lemma}   \label{lemma:matching_split}
Take $f_{11}, f_{12}, f_{21}, f_{22} \in \cS(G_n')$, $\phi_1, \phi_2 \in
\cS(F_{n})$. Put
    \[
    f_1 = f_{11} \otimes f_{12},\ f_2 = f_{21} \otimes f_{22} \in \cS(G_n), \quad
    f = \mu_1(-1)^{\frac{n(n+1)}{2}} f_1 \otimes f_2 \in \cS(G),
    \]
and
    \[
    f_1^V = f_{11} * f_{12}^\vee,\ f_2^V = f_{21} * f_{22}^\vee \in \cS(G'_n),
    \quad
    f^V = f_1^V \otimes f_2^V \in \cS(G'_n \times G_n'),
    \]
and a function $\Phi \in \cS(E_n)$ by
    \[
    \Phi(\tp{x}, y) = \phi_1^V (x) \phi_2^V(y),
    \quad x, y \in F^n.
    \]
Then $f \otimes \Phi$ and $f^V \otimes \phi_1^V \otimes \phi_2^V$ match.
\end{lemma}

\begin{proof}
This follows from the explicit form of the orbital
integrals~\eqref{eq:oi_gl_split} and~\eqref{eq:oi_u_split}. The factor
$\mu_1(-1)^{\frac{n(n+1)}{2}}$ comes from the transfer factor. Indeed, in
this case, for any regular semisimple $((g_1, g_2), w, v) \in G^+$, we have
$\Omega(((g_1, g_2), w, v)) = \mu_1(-1)^{\frac{n(n+1)}{2}}$.
\end{proof}

We can define the spherical characters as before. With the above matching, we
have the same character identity. Let $\kappa$ be the constant as defined in
Theorem~\ref{thm:comparison_character}. Since $E = F \times F$, it simplifies
to
    \[
    \kappa = \abs{\tau_1}^{-\frac{n(n-1)}{2}} \mu_1(-1)^{\frac{n(n+1)}{2}}.
    \]
The following is the split version of Theorem~\ref{thm:comparison_character}.

\begin{prop}    \label{prop:comparison_character_split}
Let $\pi = \sigma_1 \otimes \sigma_2$ be an irreducible tempered of $G'_n
\times G'_n$. Let $\Pi_i = \sigma_i \otimes \sigma_i^\vee$ be the (split)
base change of $\sigma_i$ to $G_n$. Put $\Pi = \Pi_1 \otimes \Pi_2$, which is
an irreducible tempered representation of $G$. Let $f, \Phi, f^V, \phi_1^V,
\phi_2^V$ be the functions defined in Lemma~\ref{lemma:matching_split}. Then
    \[
    J_{\pi}(f^V \otimes \phi_1^V \otimes \phi_2^V) = \kappa
    I_{\Pi}(f \otimes \Phi).
    \]
\end{prop}

\begin{proof}
Let
    \[
    \xi_1(u) = \psi(\tau_1(u_{12}+ \cdots + u_{n-1, n})), \quad u \in N_n'
    \]
be a generic character of $N_n'$ and we let $\cW_1 = \cW(\sigma_1, \xi_1)$
and $\cW_2 = \cW(\sigma_2, \overline{\xi_1})$ be Whittaker models of
$\sigma_1$ and $\sigma_2$ respectively. The factor $\tau_1$ appears for
compatibility with the Whittaker model used in the nonsplit case. Let
$\cW_1^\vee = \cW(\sigma_1^\vee, \overline{\xi_1})$ and $\cW_2^\vee =
\cW(\sigma_2^\vee, \xi_1)$ be the Whittaker models for $\sigma_1^\vee$ and
$\sigma_2^\vee$ respectively. If $W \in \cW_i$, we define $W^\vee \in
\cW_i^\vee$ by
    \[
    W^\vee(g) = W(w_n \tp{g}^{-1})
    \]
where $w_n$ is the longest Weyl group element in $G_n'$ whose antidiagonal
elements equal one. If $W$ runs over an orthogonal basis of $\cW_i$, then
$W^\vee$ runs over an orthogonal basis of $\cW_i^\vee$.

By definition we have
    \[
    J_{\pi} (f^V \otimes \phi_1^V \otimes \phi_2^V) =
    \sum_{W_1, W_2}
    \int_{G_n'}
    \langle \sigma_1(h) \sigma_1(f_1^V) W_1, W_1 \rangle^{\mathrm{Wh}}
    \langle \sigma_2(h) \sigma(f_2^V) W_2, W_2 \rangle^{\mathrm{Wh}}
    \langle \omega^\vee(h) \phi_1, \overline{\phi_2} \rangle \rd h.
    \]
where $W_1, W_2$ range over orthonormal basis of $\cW_1$ and $\cW_2$
respectively.

We now compute $I_{\Pi}$. The linear form $\beta$ in this case reduces to a
pairing between $\cW_i$ and $\cW_i^\vee$, $i = 1, 2$. It follows that
    \[
    I_{\Pi}(f \otimes \Phi) = \mu_1(-1)^{\frac{n(n+1)}{2}}  \sum_{W_1, W_2}
    \lambda(\sigma_1(f_{11}) W_1 \otimes \sigma_2(f_{21}) W_2 \otimes \phi_1)
    \lambda(\sigma_1(f_{12}) W_1^\vee \otimes
    \sigma_2(f_{22}) W_2^\vee \otimes \phi_2),
    \]
where $W_1$ and $W_2$ run over orthonormal bases of $\cW_1$ and $\cW_2$
respectively. It also equals
    \[
    I_{\Pi}(f \otimes \Phi) = \mu_1(-1)^{\frac{n(n+1)}{2}}  \sum_{W_1, W_2}
    \lambda(\sigma_1(f_{1}) W_1 \otimes \sigma_2(f_{2}) W_2 \otimes \phi_1)
    \lambda(W_1^\vee \otimes
    W_2^\vee \otimes \phi_2).
    \]

We identify $\cW_i^\vee$ with $\overline{\cW_i}$. Then the proposition
follows from the next lemma. It is part of the proof of this
proposition, but we list it as a separate lemma as it is maybe of some
independent interest.
\end{proof}

\begin{lemma}   \label{lemma:RS}
Let the notation be as in the above proof. For any $W_i \in \cW_i$, $i = 1,
2$, and $\phi\in \cS(F_{n})$, we have
    \[
    \begin{aligned}
    &\abs{\tau_1}^{-\frac{n(n-1)}{2}}
    \lambda(W_1, W_2, \phi)\overline{\lambda(W_1, W_2, \phi)}\\
     = &\int_{G_n'}
    \langle \sigma_1(h) W_1, W_1 \rangle^{\mathrm{Wh}}
    \langle \sigma_2(h) W_2, W_2 \rangle^{\mathrm{Wh}}
    \langle \omega^\vee(h) \phi, \phi \rangle \rd h.
    \end{aligned}
    \]
\end{lemma}

\begin{proof}
For the proof of this lemma, it is more convenient to use the Whittaker
models $\cW(\sigma_i, \xi)$, $i = 1, 2$ where
    \[
    \xi(u) = \psi(u_{12}+ \cdots + u_{n-1, n}), \quad u \in N_n'
    \]
is a generic character of $N_n'$, i.e. without the factor $\tau_1$. Put
temporarily
    \[
    \epsilon = \diag[\tau_1^{n-1}, \tau_1^{n-2}, \hdots, 1] \in G_n(F).
    \]
If $W_i \in \cW(\sigma_i, \xi)$, then $W_i^{\epsilon} \in \cW(\sigma_i,
\xi_1)$ where
    \[
    W_i^{\epsilon}(g) = W_i(\epsilon g).
    \]

With this change, for $W_i \in \cW(\sigma_i, \xi)$, $i = 1, 2$ we have
    \[
    \lambda(W_1^{\epsilon}, W_2^{\epsilon}, \Phi) =
    \abs{\tau_1}^{\frac{n(n+1)(n-1)}{6}}
    \lambda(W_1, W_2, \Phi), \quad
    \langle W_i^{\epsilon}, W_i^{\epsilon} \rangle^{\mathrm{Wh}} =
    \abs{\tau_1}^{\frac{n(n-1)(n-2)}{6}}
    \langle W_i, W_i \rangle^{\mathrm{Wh}}.
    \]
The desired equality in the lemma thus reduces to
    \begin{equation}    \label{eq:RS_split_standard_Whittaker}
    \lambda(W_1, W_2, \phi_1)\overline{\lambda(W_1, W_2, \phi_2)}
     = \int_{G_n(F)}
    \langle \sigma_1(h) W_1, W_1 \rangle^{\mathrm{Wh}}
    \langle \sigma_2(h) W_2, W_2 \rangle^{\mathrm{Wh}}
    \langle \omega^\vee(h) \phi, \phi \rangle \rd h
    \end{equation}
for all $W_i \in \cW(\sigma_i, \xi)$.

Both sides of~\eqref{eq:RS_split_standard_Whittaker} define nonzero elements
in
    \[
    \Hom_{G_n'}(\sigma_1 \otimeshat \sigma_2 \otimeshat \cS(F_{n}), \C)
    \otimes \overline{
    \Hom_{G_n'}(\sigma_1 \otimeshat \sigma_2 \otimeshat \cS(F_{n}), \C)
    },
    \]
and this $\Hom$ space is one dimensional by~\cites{Sun,SZ}. It follows
that~\eqref{eq:RS_split_standard_Whittaker} holds if we could
find some $W_i \in \cW(\sigma_i, \xi)$, $i= 1, 2$, and a $\Phi \in \cS(F_n)$
such that~\eqref{eq:RS_split_standard_Whittaker} holds for this choice and
is nonzero. To achieve this,
we are going to reduce it to an analogues
statement for the Rankin--Selberg integral for $G_n' \times G_{n+1}'$, proved
in~\cite{Zhang2}*{Proposition~4.10}.

We will consider a $G_n'$ naturally as a subgroup of $G_{n+1}'$, by $g
\mapsto
\begin{pmatrix} g \\ & 1 \end{pmatrix}$, $g \in G_n'$. Recall that if
$\sigma_3$ is an irreducible tempered representation of $G_{n+1}'$, then it
is well-known that that $C_c^\infty(N_n' \bs G_n', \xi)$ is contained in
$\cW(\sigma_i, \xi)$, cf.~\cite{GK75}*{Theorem~6}
and~\cite{Kem15}*{Theorem~1}. Let $f \in C_c^\infty(N_{n-1}'\bs
G_{n-1}', \xi)$, then there is an $W_f \in \cW(\sigma_2, \xi)$ such that
$W_f|_{G_{n}'} = f$. Let $\phi \in C_c^\infty(F_{n-1} \times F^\times)$.
Consider the function on
$G_n'$ given by
    \[
    \varphi(g) = W_f(g) \mu(\det g)^{-1} \abs{\det g}^{\frac{1}{2}}
    \phi(e_n g).
    \]
By the above choices we conclude that $\varphi \in C_c^\infty(N_n' \bs G_n',
\xi)$. Let $\sigma_3$ be an irreducible tempered representation of
$G_{n+1}'$. Then there is a $W_\varphi \in \cW(\sigma_3, \xi)$, such that its
restriction to $G_n'$ equals $\varphi$.

By~\cite{Zhang2}*{Proposition~4.10} for any $W_1 \in \cW(\sigma_1, \xi)$ we
have
    \begin{equation}    \label{eq:RS_n_n+1}
    \lambda'(W_1, W_{\varphi}) \overline{\lambda'(W_1, W_{\varphi})}
    = \int_{G_n'} \langle \sigma_1(h)W_1, W_1 \rangle^{\mathrm{Wh}}
    \langle \sigma_3(h) W_\varphi, W_{\varphi} \rangle^{\mathrm{Wh}}
    \rd h.
    \end{equation}
Here $\lambda'$ stands for the Rankin--Selberg integral for $G_n' \times
G_{n+1}'$, i.e.
    \[
    \lambda'(W_1, W_{\varphi}) = \int_{G_n'} W_1(h) W_{\varphi}(h) \rd h.
    \]
By our choices we have
    \[
    \lambda'(W_1, W_{\varphi}) = \lambda(W_1, W_f, \phi), \quad
    \langle W_\varphi, W_\varphi \rangle^{\mathrm{Wh}} =
    \langle W_f, W_f \rangle^{\mathrm{Wh}}
    \langle \phi, \phi \rangle.
    \]
Thus the equality~\eqref{eq:RS_n_n+1} reduces
to~\eqref{eq:RS_split_standard_Whittaker}. It is clear from the above
construction that we can choose $W_1, f$ and $\phi$ such that the above
integrals do not vanish. This finishes the proof of the lemma, and hence
Proposition~\ref{prop:comparison_character_split}.
\end{proof}

The local trace formulae Proposition~\ref{prop:lrtf_u} and~\ref{prop:lrtf_gl}
played important roles in the proof of the local character identity
Theorem~\ref{thm:comparison_character}. Let us mention that the same local
relative trace formula hold under the current assumption that $E = F \times
F$. The proof goes through with only obvious modification. Note that the
counterpart of~\cite{BP2}*{Theorem~5.53}, i.e. the formal degree conjecture,
holds for general linear groups by~\cite{HII}.


\begin{bibdiv}
\begin{biblist}

\bib{AG}{article}{
    author={Aizenbud, Avraham},
    author={Gourevitch, Dmitry},
    title={Schwartz functions on Nash manifolds},
    journal={Int. Math. Res. Not. IMRN},
    date={2008},
    number={5},
    pages={Art. ID rnm 155, 37},
    issn={1073-7928},
    review={\MR{2418286 (2010g:46124)}},
    doi={10.1093/imrn/rnm155},
}

\bib{BP2}{article}{
   author={Beuzart-Plessis, Rapha\"{e}l},
   title={A local trace formula for the Gan-Gross-Prasad conjecture for
   unitary groups: the Archimedean case},
   language={English, with English and French summaries},
   journal={Ast\'{e}risque},
   number={418},
   date={2020},
   pages={viii + 299},
   issn={0303-1179},
   isbn={978-2-85629-919-7},
   review={\MR{4146145}},
   doi={10.24033/ast},
}


\bib{BP}{article}{
   author={Beuzart-Plessis, Rapha\"{e}l},
   title={Comparison of local relative characters and the Ichino-Ikeda
   conjecture for unitary groups},
   journal={J. Inst. Math. Jussieu},
   volume={20},
   date={2021},
   number={6},
   pages={1803--1854},
   issn={1474-7480},
   review={\MR{4332778}},
   doi={10.1017/S1474748019000707},
}



\bib{BP1}{article}{
   author={Beuzart-Plessis, Rapha\"{e}l},
   title={Plancherel formula for ${\rm GL}_n(F)\backslash {\rm GL}_n(E)$ and
   applications to the Ichino-Ikeda and formal degree conjectures for
   unitary groups},
   journal={Invent. Math.},
   volume={225},
   date={2021},
   number={1},
   pages={159--297},
   issn={0020-9910},
   review={\MR{4270666}},
   doi={10.1007/s00222-021-01032-6},
}

\bib{BP3}{article}{
   author={Beuzart-Plessis, Rapha\"{e}l},
   title={A new proof of the Jacquet-Rallis fundamental lemma},
   journal={Duke Math. J.},
   volume={170},
   date={2021},
   number={12},
   pages={2805--2814},
   issn={0012-7094},
   review={\MR{4305382}},
   doi={10.1215/00127094-2020-0090},
}

\bib{BP4}{incollection}{
  title={Archimedean Theory and $\epsilon$-Factors for the {A}sai {R}ankin-{S}elberg Integrals},
  author={Beuzart-Plessis, Rapha{\"e}l},
  booktitle={Relative Trace Formulas},
  pages={1--50},
  year={2021},
  publisher={Springer}
}




\bib{BLX1}{article}{
    author={Boisseau, Paul},
    author={Lu, Weixiao},
    author={Xue, Hang},
    title={The global Gan--Gross--Prasad conjecture for Fourier--Jacobi periods on unitary groups I: coarse spectral expansions},
    note={preprint},
}



\bib{CZ}{article} {
    title={Le transfert singulier pour la formule des traces de Jacquet–Rallis},
    volume={157},
    DOI={10.1112/S0010437X20007599},
    number={2},
    journal={Compositio Mathematica},
    publisher={London Mathematical Society},
    author={Chaudouard, Pierre-Henri},
    author={Zydor, Micha\l},
    year={2021},
    pages={303–434},
}




\bib{GK75}{article}{
  title={{Representations of the group $GL(n,K)$ where $K$ is a local field}},
  author={Gelfand, I.M.},
  author={Kajdan, D.A.},
  year={1975},
}





\bib{GI2}{article}{
   author={Gan, Wee Teck},
   author={Ichino, Atsushi},
   title={The Gross-Prasad conjecture and local theta correspondence},
   journal={Invent. Math.},
   volume={206},
   date={2016},
   number={3},
   pages={705--799},
   issn={0020-9910},
   review={\MR{3573972}},
   doi={10.1007/s00222-016-0662-8},
}





\bib{HC}{article}{
   author={Harish-Chandra},
   title={Harmonic analysis on real reductive groups. III. The Maass-Selberg relations and the Plancherel formula},
   journal={Ann. of Math.},
   volume={104},
   date={1976},
   pages={117--201},
}





\bib{HII}{article}{
   author={Hiraga, Kaoru},
   author={Ichino, Atsushi},
   author={Ikeda, Tamotsu},
   title={Formal degrees and adjoint $\gamma$-factors},
   journal={J. Amer. Math. Soc.},
   volume={21},
   date={2008},
   number={1},
   pages={283--304},
   issn={0894-0347},
   review={\MR{2350057}},
   doi={10.1090/S0894-0347-07-00567-X},
}







\bib{Kem15}{article}{
   author={Kemarsky, Alexander},
   title={A note on the Kirillov model for representations of ${\rm
   GL}_n(\mathbb{C})$},
   language={English, with English and French summaries},
   journal={C. R. Math. Acad. Sci. Paris},
   volume={353},
   date={2015},
   number={7},
   pages={579--582},
   issn={1631-073X},
   review={\MR{3352025}},
   doi={10.1016/j.crma.2015.04.002},
}
	

\bib{KMSW}{article}{
    author={Kaletha, Tasho},
    author={Minguez, Alberto},
    author={Shin, Sug Woo},
    author={White, Paul-James},
    title={Endoscopic Classification of Representations: Inner Forms of Unitary
    Groups},
    note={arXiv:1409.3731v3},
}








	
\bib{Liu}{article}{
   author={Liu, Yifeng},
   title={Relative trace formulae toward Bessel and Fourier-Jacobi periods
   on unitary groups},
   journal={Manuscripta Math.},
   volume={145},
   date={2014},
   number={1-2},
   pages={1--69},
   issn={0025-2611},
   review={\MR{3244725}},
   doi={10.1007/s00229-014-0666-x},
}




\bib{Mok}{article}{
   author={Mok, Chung Pang},
   title={Endoscopic classification of representations of quasi-split
   unitary groups},
   journal={Mem. Amer. Math. Soc.},
   volume={235},
   date={2015},
   number={1108},
   pages={vi+248},
   issn={0065-9266},
   isbn={978-1-4704-1041-4},
   isbn={978-1-4704-2226-4},
   review={\MR{3338302}},
   doi={10.1090/memo/1108},
}
	







\bib{Sun}{article}{
   author={Sun, Binyong},
   title={Multiplicity one theorems for Fourier-Jacobi models},
   journal={Amer. J. Math.},
   volume={134},
   date={2012},
   number={6},
   pages={1655--1678},
   issn={0002-9327},
   review={\MR{2999291}},
   doi={10.1353/ajm.2012.0044},
}





\bib{SZ}{article}{
   author={Sun, Binyong},
   author={Zhu, Chen-Bo},
   title={Multiplicity one theorems: the Archimedean case},
   journal={Ann. of Math. (2)},
   volume={175},
   date={2012},
   number={1},
   pages={23--44},
   issn={0003-486X},
   review={\MR{2874638}},
   doi={10.4007/annals.2012.175.1.2},
}





\bib{Wald}{article}{
   author={Waldspurger, J.-L.},
   title={La formule de Plancherel pour les groupes $p$-adiques (d'apr\`es
   Harish-Chandra)},
   language={French, with French summary},
   journal={J. Inst. Math. Jussieu},
   volume={2},
   date={2003},
   number={2},
   pages={235--333},
   issn={1474-7480},
   review={\MR{1989693}},
   doi={10.1017/S1474748003000082},
}

\bib{Yun}{article}{
   author={Yun, Zhiwei},
   title={The fundamental lemma of Jacquet and Rallis},
   note={With an appendix by Julia Gordon},
   journal={Duke Math. J.},
   volume={156},
   date={2011},
   number={2},
   pages={167--227},
   issn={0012-7094},
   review={\MR{2769216}},
   doi={10.1215/00127094-2010-210},
}





\bib{Xue1}{article}{
   author={Xue, Hang},
   title={The Gan-Gross-Prasad conjecture for ${\rm U}(n)\times{\rm U}(n)$},
   journal={Adv. Math.},
   volume={262},
   date={2014},
   pages={1130--1191},
   issn={0001-8708},
   review={\MR{3228451}},
   doi={10.1016/j.aim.2014.06.010},
}





\bib{Xue3}{article}{
   author={Xue, Hang},
   title={On the global Gan-Gross-Prasad conjecture for unitary groups:
   approximating smooth transfer of Jacquet-Rallis},
   journal={J. Reine Angew. Math.},
   volume={756},
   date={2019},
   pages={65--100},
   issn={0075-4102},
   review={\MR{4026449}},
   doi={10.1515/crelle-2017-0016},
}





\bib{Xue4}{article}{
   author={Xue, Hang},
   title={Bessel models for real unitary groups: the tempered case},
   journal={Duke Math. J.},
   volume={172},
   date={2023},
   number={5},
   pages={995--1031},
   issn={0012-7094},
   review={\MR{4568696}},
   doi={10.1215/00127094-2022-0018},
}




\bib{Xue6}{article}{
    author={Xue, Hang},
    title={Fourier--Jacobi models for real unitary groups},
    note={To appear in American Journal of Math.},
}




\bib{Zhang1}{article}{
   author={Zhang, Wei},
   title={Fourier transform and the global Gan-Gross-Prasad conjecture for
   unitary groups},
   journal={Ann. of Math. (2)},
   volume={180},
   date={2014},
   number={3},
   pages={971--1049},
   issn={0003-486X},
   review={\MR{3245011}},
}




\bib{Zhang2}{article}{
   author={Zhang, Wei},
   title={Automorphic period and the central value of Rankin-Selberg
   L-function},
   journal={J. Amer. Math. Soc.},
   volume={27},
   date={2014},
   number={2},
   pages={541--612},
   issn={0894-0347},
   review={\MR{3164988}},
   doi={10.1090/S0894-0347-2014-00784-0},
}


\bib{Zhang3}{article}{
   author={Zhang, W.},
   title={Weil representation and arithmetic fundamental lemma},
   journal={Ann. of Math. (2)},
   volume={193},
   date={2021},
   number={3},
   pages={863--978},
   issn={0003-486X},
   review={\MR{4250392}},
   doi={10.4007/annals.2021.193.3.5},
}











\bib{Zydor2}{article}{
   author={Zydor, Micha\l },
   title={La variante infinit\'{e}simale de la formule des traces de
   Jacquet-Rallis pour les groupes lin\'{e}aires},
   language={French, with English and French summaries},
   journal={J. Inst. Math. Jussieu},
   volume={17},
   date={2018},
   number={4},
   pages={735--783},
   issn={1474-7480},
   review={\MR{3835522}},
   doi={10.1017/S1474748016000141},
}



\bib{Zydor3}{article}{
   author={Zydor, Micha\l },
   title={Les formules des traces relatives de Jacquet-Rallis grossi\`eres},
   language={French, with English and French summaries},
   journal={J. Reine Angew. Math.},
   volume={762},
   date={2020},
   pages={195--259},
   issn={0075-4102},
   review={\MR{4195660}},
   doi={10.1515/crelle-2018-0027},
}



\end{biblist}
\end{bibdiv}
\end{document}